\RequirePackage{fix-cm}
\documentclass[smallextended]{svjour3_alt}       
\smartqed  
\usepackage{graphicx}
\usepackage[labelfont=bf]{caption} 
\usepackage{geometry}

\usepackage{algorithm}
\usepackage{algpseudocode}
\usepackage{bm}
\usepackage{amsmath,amssymb}
\usepackage[utf8]{inputenc}
\usepackage{fourier}
\usepackage{times}

\usepackage{commath} 
\usepackage{stmaryrd} 
\usepackage{enumerate} 
\usepackage[numbers]{natbib} 
\usepackage{hyperref}

\usepackage{tikz}
\usetikzlibrary{spy,calc}
\usepackage{pgfplots}
\pgfplotsset{compat=1.13}
\usepackage{subcaption}

\newcommand{\algrule}[1][.2pt]{\par\vskip.5\baselineskip\hrule height #1\par\vskip.5\baselineskip}

\hypersetup{
    colorlinks,
    citecolor=blue,
    filecolor=black,
    linkcolor=black,
    urlcolor=blue
}

\newcommand{\algoref}{Algorithm~\ref}

\renewcommand{\secref}{Section~\ref}
\newcommand{\inner}[1]{\langle #1 \rangle}

\newcommand{\eps}{\varepsilon}

\newcommand{\NN}{\mathbb{N}} \newcommand{\PP}{\mathbb{P}}
\newcommand{\QQ}{\mathbb{Q}} \newcommand{\RR}{\mathbb{R}}

 \newcommand{\calD}{\mathcal{D}}

\DeclareMathOperator{\co}{co}

\DeclareMathOperator{\Div}{div}

\DeclareMathOperator{\Mod}{mod}

\DeclareMathOperator{\TV}{TV}
\DeclareMathOperator{\TGV}{TGV}

\DeclareMathOperator{\sgn}{sgn}

\DeclareMathOperator{\SSIM}{SSIM}

\DeclareMathOperator{\Sym}{Sym}
\DeclareMathOperator*{\argmin}{arg\,min}

\numberwithin{equation}{section}
\numberwithin{figure}{section}

\newtheorem{thm}{Theorem}
\numberwithin{thm}{section}

\newtheorem{defn}[thm]{Definition}
\newtheorem{lem}[thm]{Lemma}
\newtheorem{prop}[thm]{Proposition}

\newtheorem{rem}[thm]{Remark}
\newtheorem{ex}[thm]{Example}
\renewenvironment{proof}{\emph{Proof.}}{\qed}

\usepackage{verbatim}
\newenvironment{figuretmp}{\begin{figure}}{\end{figure}}

\begin{document}

\title{A geometric integration approach to nonsmooth, nonconvex optimisation\thanks{All authors acknowledge support from CHiPS (Horizon 2020 RISE project grant). E. S. R, M. J. E. and C.-B. S. acknowledge support from the Cantab Capital Institute for the Mathematics of Information.  M. J. E. and C.-B. S. acknowledge support from Leverhulme Trust project “Breaking the non-convexity barrier”, EPSRC grant “EP/M00483X/1”,  and EPSRC centre “EP/N014588/1”. G. R. W. Q. acknowledges support from the Australian Research Council, and is grateful to the Mittag-Leffler Institute for a productive stay. Moreover, C.-B. S. acknowledges support from the RISE project NoMADS and the Alan Turing Institute.}
}



\titlerunning{Geometric integration for nonsmooth problems}        

\author{Erlend S. Riis \and
		Matthias J. Ehrhardt  \and
		G. R. W. Quispel \and
         Carola-Bibiane Sch{\"o}nlieb
}


\institute{Erlend S. Riis \at
              Department of Applied Mathematics and Theoretical Physics, University of Cambridge, UK \\
              \email{e.s.riis@damtp.cam.ac.uk} 
           \and
           Matthias J. Ehrhardt \at
              Department of Applied Mathematics and Theoretical Physics, University of Cambridge, UK \\
              \email{m.j.ehrhardt@damtp.cam.ac.uk}           
           \and
           G. R. W. Quispel \at
           	Department of Mathematics and Statistics, La Trobe University, Victoria 3086, Australia \\
           	\email{r.quispel@latrobe.edu.au}
           \and
           Carola-Bibiane Sch{\"o}nlieb \at
              Department of Applied Mathematics and Theoretical Physics, University of Cambridge, UK \\
              \email{cbs31@cam.ac.uk} 
}

\date{}

\maketitle

\begin{abstract}
The optimisation of nonsmooth, nonconvex functions without access to gradients is a particularly challenging problem that is frequently encountered, for example in model parameter optimisation problems. Bilevel optimisation of parameters is a standard setting in areas such as variational regularisation problems and supervised machine learning. We present efficient and robust derivative-free methods called randomised Itoh--Abe methods. These are generalisations of the Itoh--Abe discrete gradient method, a well-known scheme from geometric integration, which has previously only been considered in the smooth setting. We demonstrate that the method and its favourable energy dissipation properties are well-defined in the nonsmooth setting. Furthermore, we prove that whenever the objective function is locally Lipschitz continuous, the iterates almost surely converge to a connected set of Clarke stationary points. We present an implementation of the methods, and apply it to various test problems. The numerical results indicate that the randomised Itoh--Abe methods are superior to state-of-the-art derivative-free optimisation methods in solving nonsmooth problems while remaining competitive in terms of efficiency.
\end{abstract}

\section{Introduction}
\label{intro}

We consider the unconstrained optimisation problem
\begin{equation}
\label{eq:main}
\min_{x \in \RR^n} V(x),
\end{equation}
where the objective function \(V\) is locally Lipschitz continuous, bounded below and coercive---the latter meaning that \(\set{x \in \RR^n \; : \; V(x) \leq M}\) is compact for all \(M \in \RR\). The function may be nonconvex and nonsmooth, and we assume no knowledge besides point evaluations \(x \mapsto V(x)\). To solve \eqref{eq:main}, we present \emph{randomised Itoh--Abe methods}, a generalisation of the \emph{Itoh--Abe discrete gradient method}. The latter is a derivative-free optimisation scheme,\footnote{Not to be confused with another derivative-free method with the same name proposed by Bagirov et al. \citep{bag08}, which uses a different concept of a discrete gradient.} that has previously only been considered for differentiable functions.

Discrete gradient methods, a tool from geometric numerical integration, are optimisation schemes that inherit the energy dissipation of continuous gradient flow. The iterates of the methods monotonically decrease the objective function, for all time steps, and Grimm et al. \citep{gri17} recently provided a convergence theory for solving \eqref{eq:main} in the continuously differentiable setting. We extend the concepts and results of their work and show that the Itoh--Abe discrete gradient method can be applied in the nonsmooth case, and, furthermore, that the favourable dissipativity property of the methods extends to this setting. Furthermore, we prove that for locally Lipschitz continuous functions the iterates converge to a set of stationary points, defined by the Clarke subdifferential framework.

\subsection{Gradient flow and the discrete gradient method}

For a differentiable function \(V:\RR^n \to \RR\), gradient flow is the ODE system defined by
\begin{equation}
\label{eq:gradient_flow}
\dot{x} = -\nabla V(x), \qquad x(0) = x_0 \in \RR^n,
\end{equation}
where the dot represents differentiation with respect to time. By applying the chain rule, we compute
\begin{equation}
\label{eq:gradient_flow_dissipation}
\dod{}{t}V(x(t)) = \inner{\nabla V(x(t)), \dot{x}(t)} = - \|\nabla V(x(t))\|^2 = - \|\dot{x}(t)\|^2 \leq 0,
\end{equation}
where \(\|x\|\) denotes the 2-norm \(\sqrt{\inner{x,x}}\). This implies that gradient flow is inherently an energy dissipative system.

In the field of geometric numerical integration, one studies methods for numerically solving ODEs that also preserve structures of the continuous system---see \citep{hai06, mcl01} for an introduction. Discrete gradient methods can be applied to first-order gradient systems to preserve energy conservation laws, dissipation laws, as well as Lyapunov functions \citep{gon96, ito88, mcl99, qui96}. They are defined as follows.
\begin{defn}
Let \(V: \RR^n \to \RR\) be continuously differentiable. A \emph{discrete gradient} is a continuous mapping \(\overline{\nabla} V: \RR^n \times \RR^n \to \RR^n\) that satisfies the two following properties.
\[
\left\{ \begin{alignedat}{2}
\inner{ \overline{\nabla} V(x, y), y-x} &= V(y) - V(x) & \quad \mbox{(consistency)}& \\
\lim_{y \to x} \overline{\nabla} V(x,y) &= \nabla V(x) & \quad \mbox{(mean value property)}&
\end{alignedat}
\right.
\qquad \mbox{for all } x, y \in \RR^n. 
\]
\end{defn}

We now introduce the discrete gradient method for optimisation. For \(x^0 \in \RR^n\) and time steps \(\tau_k > 0\), \(k \in \NN\), we solve
\begin{equation}
\label{eq:dg_method}
x^{k+1} = x^k - \tau_k \overline{\nabla} V(x^k, x^{k+1}).
\end{equation}
We apply the above mean value property to derive that the iterates decrease \(V\).
\begin{align}
V(x^{k+1}) - V(x^k) &= \inner{\overline{\nabla}V(x^k, x^{k+1}), x^{k+1} - x^k} \nonumber \\
\label{eq:dissipation}
&= - \tau_k \|\overline{\nabla}V(x^k, x^{k+1})\|^2 = - \frac{1}{\tau_k} \|x^{k+1} - x^k\|^2.
\end{align}
Note that the decrease holds for all time steps \(\tau_k > 0\), and that \eqref{eq:dissipation} can be seen as a discrete analogue of the dissipative structure of gradient flow \eqref{eq:gradient_flow_dissipation}, replacing derivatives by finite differences.

Grimm et al. \citep{gri17} proved that for coercive, continuously differentiable functions, the iterates of \eqref{eq:dg_method} converge to a set of stationary points, provided that there are strictly positive constants \(\tau_{\min}, \tau_{\max}\) such that \(\tau_k \in [\tau_{\min}, \tau_{\max}]\) for all \(k \in \NN\).

\subsubsection{Itoh--Abe methods}

The \emph{Itoh--Abe discrete gradient} \citep{ito88} (also known as coordinate increment discrete gradient)\footnote{There are infinitely many discrete gradients, each with a corresponding discrete gradient method. See \citep{gri17} for further examples.} is defined as
\begin{equation*}
\label{eq:itoh_abe}
\overline{\nabla} V(x,y) = \begin{pmatrix}
\frac{V\del{y_1, x_2, \ldots, x_n} - V(x)}{y_1 - x_1}  \\
\frac{V\del{y_1, y_2, x_3, \ldots, x_n} - V\del{y_1, x_2, \ldots, x_n}}{y_2 - x_2} \\
\vdots \\
\frac{V(y) - V\del{y_1, \ldots, y_{n-1},x_n}}{y_n - x_n}
\end{pmatrix}.
\end{equation*}
Solving an iterate of the discrete gradient method \eqref{eq:dg_method} with the Itoh--Abe discrete gradient is equivalent to successively solving \(n\) scalar equations of the form
\begin{equation*}
\begin{aligned}
x^{k+1}_1 &= x^k_1 - \tau_k \frac{V(x^{k+1}_1, x^k_2, \ldots, x_n^k)- V(x^k)}{x^{k+1}_1 - x^k_1} \\
x^{k+1}_2 &= x^k_2 - \tau_k \frac{V(x^{k+1}_1, x^{k+1}_2, x^k_3, \ldots, x_n^k)- V(x^{k+1}_1, x^k_2, \ldots, x_n^k)}{x^{k+1}_2 - x^k_2} \\
&\;\vdots \\
x^{k+1}_n &= x^k_n - \tau_k \frac{V(x^{k+1})- V(x^{k+1}_1, x^{k+1}_2, \ldots, x^{k+1}_{n-1}, x_n^k)}{x^{k+1}_n - x^k_n}.
\end{aligned}
\end{equation*}

We generalise the Itoh--Abe discrete gradient method to \emph{randomised Itoh--Abe methods} accordingly. Let \((d^k)_{k \in \NN} \subset S^{n-1}\) be a sequence of directions, where \(S^{n-1}\) denotes the unit sphere \(\set{x \in \RR^n \; : \; \|x\| = 1}\).  The directions can be drawn from a random distribution or chosen deterministically. At the \(k\)th step, we update
\begin{equation}
\label{eq:itoh_abe_method}
x^{k+1} \mapsto x^k - \tau_{k} \beta_{k} d^{k}, \; \mbox{where } \beta_{k} \neq 0 \mbox{ solves } \; \beta_{k} = - \frac{V(x^k - \tau_{k} \beta_{k} d^{k}) - V(x^k)}{\tau_{k} \beta_{k}},
\end{equation}
If no such \(\beta_{k}\) exists, we set \(x^{k+1} = x^k\). We formalise this method in \algoref{algo:main}. We assume throughout the paper that the time steps \((\tau_k)_{k \in \NN}\) are bounded between two strictly positive constants \(\tau_{\min}, \tau_{\max}\), which can take arbitrary values.

\begin{algorithm}[!ht]
\caption{Randomised Itoh--Abe method}
    \textbf{Input:}
        starting point \(x^0\),
        directions $(d^k)_{k \in \NN}$,
        time steps \((\tau_k)_{k\in\NN}\).
\begin{algorithmic}
\algrule
	\For{\(k = 0, 1, 2, \ldots\)}
		\State Update \(x^{k+1} = x^k - \tau_k \beta_k d^k\) via \eqref{eq:itoh_abe_method}
	\EndFor
\end{algorithmic}
\label{algo:main}
\end{algorithm}

Observe that if \((d^k)_{k\in\NN}\) cycle through the standard coordinates \((e^i)_{i=1}^n\) with the rule \(d^k = e^{[k \Mod n] + 1}\), then computing \(n\) steps of \eqref{eq:itoh_abe_method} corresponds to one step of \eqref{eq:dg_method} with the Itoh--Abe discrete gradient. Furthermore, the dissipation properties \eqref{eq:dissipation} can be rewritten as
\begin{align}
\label{eq:dissipation_dg}
V(x^{k+1}) - V(x^k) &= -\tau_k\del{\frac{V(x^{k+1}) - V(x^k)}{\|x^{k+1} - x^k\|}}^2 = -\frac{1}{\tau_k} \|x^{k+1} - x^k\|^2.
\end{align}
Consequently, the dissipative structure of the Itoh--Abe methods is well-defined in a derivative-free setting.

Ehrhardt et al. \citep{rii18smooth} studied Itoh--Abe methods in the smooth setting, asserting linear convergence rates for functions that satisfy the Polyak--{\L}ojasiewicz inequality \citep{kar16}. Furthermore, the application of Itoh--Abe discrete gradient methods to smooth optimisation problems is well-documented. Applications include convex variational regularisation problems for image analysis \citep{gri17}, nonconvex image inpainting problems with Euler's elastica regularisation \citep{rin17}, for which it outperformed gradient-based schemes, and the popular Gauss--Seidel method and successive-over-relaxation (SOR) methods for solving linear systems \citep{miy17}.

\subsection{Bilevel optimisation and blackbox problems}

An important application for developing derivative-free solvers is \emph{model parameter optimisation problems}. The setting for this class of problems is as follows. A model depends on some tunable parameters \(\alpha \in \RR^n\), so that for a given parameter choice \(\alpha\), the model returns an output \(u_\alpha\). There is a cost function \(\Phi\), which assigns to output \(u_\alpha\) a numerical score \(\Phi(u_\alpha) \in \RR\), which we want to minimise. The associated model parameter optimisation problem becomes
\[
\alpha^* \in \argmin_{\alpha \in \RR^n} \Phi(u_\alpha).
\]
A well-known example of parameter optimisation problems is supervised machine learning.

In this paper, we consider one instance of such problems in image analysis, called \emph{bilevel optimisation} of variational regularisation problems. Here, the model is given by a variational regularisation problem for image denoising,
\[
u_\alpha \in \argmin_{u} \frac{1}{2}\|u - f^\delta\|^2 + R_\alpha(u),
\]
where \(f^\delta\) is a noisy image and \(\alpha\) is the regularisation parameter. For training data with desired reconstruction \(u^\dagger\), we consider a scoring function \(\Phi\) that estimates the discrepancy between \(u^\dagger\) and the reconstruction \(u_\alpha\). In \secref{sec:bilevel}, we apply Itoh-Abe methods to solve these problems.

Bilevel optimisation problems, and model parameter optimisation problems in general, pose several challenges. They are often nonconvex and nonsmooth, due to the nonsmoothness and nonlinearity of \(\alpha \mapsto u_\alpha\). Furthermore, the model simulation \(\alpha \mapsto u_\alpha\) is an algorithmic process for which gradients or subgradients cannot easily be estimated. Such problems are termed \emph{blackbox optimisation problems}, as one only has access to point evaluations of the function. It is therefore of great interest to develop efficient and robust derivative-free methods for such optimisation problems. 

There is a rich literature on bilevel optimisation for variational regularisation problems in image analysis, c.f. e.g. \citep{cal15, de17, kun13, och15}. Furthermore, model parameter optimisation problems appear in many other applications. These include optimising for the management of water resources \citep{fow08}, approximation of a transmembrane protein structure in computational biology \citep{gra04}, image registration in medical imaging \citep{oeu07}, the building of wind farms \citep{dup16}, and solar energy utilisation in architectural design \citep{kam10}, to name a few.

\subsection{Related literature on nonsmooth, nonconvex optimisation}

Although nonsmooth, nonconvex problems are known for their difficulty compared to convex problems, a rich optimisation theory has grown since the 1970s. As the focus of this paper is derivative-free optimisation, we will compare the methods' convergence properties and performance to other derivative-free solvers. Audet and Hare \citep{aud17} recently provided a reference work for this field.

While there a myriad of derivative-free solvers, few provide convergence guarantees for nonsmooth, nonconvex functions. Audet and Dennis Jr \citep{aud06} introduced the \emph{mesh adaptive direct search} (MADS) method for constrained optimisation, with provable convergence guarantees to stationary points for nonsmooth, nonconvex functions. Direct search methods evaluate the function at a finite polling set, compare the evaluations, and update the polling set accordingly. Such methods only consider the ordering of the evaluations, rather than the numerical differences. A significant portion of derivative-free methods are direct search methods, and the most well known of these is the Nelder--Mead method (also known as the downhill simplex method) \citep{nel65}. 

Alternatively, model-based methods that build a local quadratic model based on evaluations are well-documented \citep{car18, pow06, pow09}. While such methods tend to work well in practice, they are normally designed only for smooth functions, so their performance on nonsmooth functions is not guaranteed.

Fasano et al. \citep{fas14} formulated a derivative-free line search method termed DFN and analyse its convergence properties for nonsmooth functions for the Clarke subdifferential, in the constrained setting. The Itoh--Abe methods share many similarities with DFN, such as performing line searches along dense directions, and they employ a similar convergence analysis. However, the line searches of these methods differ, and the Itoh--Abe methods are in particular motivated by discrete gradient methods and preserving the dissipative structure of gradient flow. Furthermore, our convergence analysis is more comprehensive, considering both stochastic and deterministic methods, as well as obtaining stronger convergence results in the deterministic case. Building on the aforementioned algorithm, Liuzzi and Truemper \citep{liu17} formulated a derivative-free method that is a hybrid between DFN and MADS.

Furthermore, we note the resemblance of randomised Itoh--Abe methods \eqref{eq:itoh_abe_method}, when \((d^k)_{k\in\NN}\) is randomly, independently drawn from \(S^{n-1}\), to the random search method proposed by Polyak in \citep{pol87} and studied for nonsmooth, convex functions by Nesterov in \citep{nes17}, given by 
\[
x^{k+1} = x^k - \tau_k \frac{V(x^k + \beta_k d^k) - V(x^k)}{\beta_k} d^k,
\]
where \(d^k\) is randomly, independently drawn from \(S^{n-1}\). The implicit equation \eqref{eq:itoh_abe_method} can be treated as a line search rule for the above method, with constraints imposed by \(\tau_{\min}\), \(\tau_{\max}\).

While our focus is on derivative-free methods, we also mention some popular methods for nonsmooth, nonconvex optimisation that use gradient or subgradient information. Central in nonsmooth optimisation are \emph{bundle methods}, where a \emph{subgradient} \citep{cla90} is required at each iterate to construct a linear approximation to the objective function---see \citep{kiw85} for an introduction. A close alternative to bundle methods are \emph{gradient sampling methods} (see \citep{bur18} for a recent review by Burke et al.), where the descent direction is determined by sampling gradients in a neighborhood of the current iterate. Curtis and Que \citep{cur15} formulated a hybrid method between the gradient sampling scheme of \citep{cur13} and the well known quasi-Newton method BFGS adapted to nonsmooth problems \citep{lew13}. These methods have convergence guarantees in the Clarke subdifferential framework, under the assumption that the objective function is differentiable on an open, dense set. Last, we mention a derivative-free scheme based on gradient sampling methods, proposed by Kiwiel \citep{kiw10}, where gradients are replaced by Gupal's estimates of gradients of the Steklov averages of the objective function. This method has convergence guarantees in the Clarke subdifferential framework, but has a high computational cost in terms of function evaluations per iterate.

\subsection{Contributions}

In this paper, we formulate randomised Itoh--Abe methods for nonsmooth functions. We prove that the method always admits a solution, and that the iterates converge to a set of Clarke stationary points, for any locally Lipschitz continuous function, and both for deterministic and randomly chosen search directions. Consequently, the scope of discrete gradient methods for optimisation is significantly broadened, and we conclude that the dissipativity properties of gradient flow can be preserved even beyond differentiability. Ultimately, this provides a new, robust, and versatile optimisation scheme for nonsmooth, nonconvex functions.

The theoretical convergence analysis for the Itoh--Abe methods is thorough and foundational, and we provide examples that demonstrate that the conditions of the convergence theorem are not just sufficient, but necessary. Furthermore, the statements and proofs are sufficiently general so that they can be adapted to other schemes, such as the aforementioned DFO method, thus enhancing the theory of these methods as well. 

We show that the method works well in practice, by solving bilevel optimisation problems for variational regularisation problems, as well as solving benchmark problems such as Rosenbrock functions.

The rest of the paper is structured as follows. \secref{sec:nonconvex} provides a background on the Clarke sub\-differential for nonsmooth, nonconvex analysis. In \secref{sec:theoretical}, the main theoretical results of the paper are presented, namely existence and optimality results in the stochastic and deterministic setting. In \secref{sec:rotated}, we briefly discuss the Itoh--Abe discrete gradient for general coordinate systems. In \secref{sec:implementation} and \secref{sec:examples}, the numerical implementation is described and results from example problems are presented. Finally, a conclusion is given in \secref{sec:conclusion}.

\section{Nonconvex optimisation}

\label{sec:nonconvex}

In this section, we introduce the Clarke subdifferential framework \citep{cla90} for nonsmooth, nonconvex optimisation. This is the most popular framework for this setting, due to its nice analytical properties. It generalises the gradient of a differentiable function, as well as the subdifferential \citep{eke99} of a convex function, hence the term \emph{Clarke subdifferential}. Francis H. Clarke introduced the framework in his doctoral thesis in 1973 \citep{cla73}, in which he termed it \emph{generalised gradient}.

\subsection{The Clarke subdifferential}

Throughout the paper, for \(\eps > 0\) and \(x \in \RR^n\), we denote by \(B_\eps(x)\) the open ball \(\{y \in \RR^n \; : \; \|y - x\| < \eps \}\).

\begin{defn}
\(V\) is \emph{Lipschitz of rank \(L\) near \(x\)} if there exists \(\eps > 0\) such that for all \(y, z \in B_\eps(x)\), one has
\[
|V(y) - V(z)| \leq L\|y-z\|.
\]
\(V\) is locally Lipschitz continuous if the above property holds for all \(x \in \RR^n\).
\end{defn}

\begin{defn}
For \(V\) Lipschitz near \(x\) and for a vector \(d \in \RR^n\),  the \emph{Clarke directional derivative} is given by
\[
V^o(x ; d) = \limsup_{y \to x, \; \lambda \downarrow 0} \dfrac{V(y+\lambda d) - V(y)}{\lambda}.
\]
\end{defn}
\begin{defn}
\label{defn:subdifferential}
Let \(V\) be locally Lipschitz and \(x \in \RR^n\). The \emph{Clarke subdifferential} of \(V\) at \(x\) is given by
\[
\partial V(x) = \set{p \in \RR^n \; : \; V^o(x; d) \geq \inner{d, p} \mbox{ for all } d \in \RR^n}.
\]
An element of \(\partial V(x)\) is called a \emph{Clarke subgradient}.
\end{defn}
The subdifferential \(\partial V\) is well-defined for locally Lipschitz functions, coincides with the standard subdifferential for convex functions \citep[Proposition 2.2.7]{cla90}, and coincides with the derivative at points of strict differentiability \citep[Proposition 2.2.4]{cla90}. It can equivalently be characterised as
\[
\partial V(x) = \co\set{d \in \RR^n : \exists \, (x^k)_{k \in \NN} \subset \calD(V) \; \mbox{ s.t. } \; x^k \to x \; \mbox{ and } \; \nabla V(x^k) \to d},
\]
where \(\calD(V)\) is the set of differentiable points, and \(\co\) denotes the convex hull of the set \citep[Theorem 2.5.1]{cla90}. We additionally state two useful results, both of which can be found in Chapter 2 of \citep{cla90}.
\begin{prop}
\label{prop:subdiff_properties}
Suppose \(V\) is locally Lipschitz continuous. Then
\begin{enumerate}[(i)]
	\item \(\partial V(x)\) is nonempty, convex and compact, and if \(V\) is Lipschitz of rank \(L\) near \(x\), then \(\partial V(x) \subseteq B_L(0)\).
	\item \(\partial V(x)\) is outer semicontinuous. That is, for all \(\eps > 0\), there exists \(\delta > 0\) such that
	\[
	\partial V(y) \subset \partial V(x) + B_\eps(0), \quad \mbox{ for all } y \in x + B_\delta(0).
	\]
\end{enumerate}
\end{prop}

There are alternative frameworks for generalising differentiability of nonsmooth, nonconvex functions. For example, the Michel--Penot subdifferential \citep{mic84} coincides with the G\^{a}teaux derivative when this exists, unlike the Clarke subdifferential, which is larger and only coincides with strict derivatives \citep{gil99}. However, the Clarke sub\-differential is outer semicontinuous, making it in most cases the preferred framework for analysis. See \citep{bor99}  by Borwein and Zhu for a survey of various subdifferentials, published on the 25th birthday of the Clarke subdifferential.

\subsubsection{Discrete gradients versus subgradients}

By definition, when \(V\) is continuously differentiable, any discrete gradient \(\overline{\nabla} V(x,y)\) converges to the gradient \(\nabla V(x)\) as \(y \to x\). However, for nondifferentiable \(V\), discrete gradients do not necessarily approximate a subgradient or even an \emph{\(\eps\)-approximate subgradient}.\footnote{For convex functions, \(p \in \RR^n\) is an \(\eps\)-approximate subgradient if, for all \(y \in \RR^n\), we have \(V(y) \geq V(x) + \inner{p, y-x} - \eps\) \citep{hin01}.} This is demonstrated by the following example.

\begin{ex}
Let \(V(x_1, x_2) := \sqrt{x_1^2 + x_2^2}\), and set \(x^k = [\tfrac{1}{k}, 0]^T\) and \(x_2^k = [0, \tfrac{1}{k}]^T\). Then, for all \(k\), the Itoh--Abe discrete gradient is
\[
\overline{\nabla} V(x_1^k, x_2^k) = [1,1]^T.
\]
Thus, \(x_1^k \to [0,0]^T\), \(x_2^k \to [0,0]^T\) and \(\overline{\nabla} V(x_1^k, x_2^k) \to [1,1]^T\). However, \([1,1]^T\) is not in \(\partial V(0,0) = B_1(0,0)\). In fact, for all \(\eps > 0\), we have \([1,1]^T \notin \partial_\eps V(0,0)\).
\end{ex}

\subsubsection{Clarke stationary points}

\begin{defn}
\(x^* \in \RR^n\) is a \emph{Clarke stationary point} of \(V\) if \(0 \in \partial V(x^*)\).
\end{defn}
For our purposes, we also define Clarke directional stationarity.
\begin{defn}[Directional Clarke stationarity]
For a direction \(d \in \RR^n \setminus \set{0}\), we say that \(V\) is \emph{Clarke directionally stationary at \(x^*\) along \(d\)} if 
\[
\min\set{V^o(x; d), V^o(x; -d)} \geq 0.
\]
\end{defn}
\begin{rem}
A point \(x^*\) is Clarke stationary if and only if \(V\) is Clarke directionally stationary at \(x^*\) along \(d\) for all \(d \in \RR^n\) such that \(\|d\| = 1\).
\end{rem}

Any local maxima and minima are stationary. If \(V\) is convex, then stationary points coincide with the global minima. For more general classes of functions, the concept of Clarke stationary points also reduces to convex first-order optimality conditions.
\begin{defn}[\citep{pen97}]
A locally Lipschitz continuous function \(V\) is \emph{pseudoconvex} if, for all \(x, y \in \RR^n\),
\[
f(y) < f(x) \implies \forall p \in \partial V(x), \quad \inner{p, y-x} < 0.
\]
\end{defn}
If \(V\) is pseudoconvex, then any Clarke stationary point is a global minimum \citep{aus98}. Clarke \citep{cla90} also introduced the notion of \emph{regularity}.
\begin{defn}
A function \(V\) is \emph{regular} at \(x\) if the directional derivative 
\[
V'(x; d) := \lim_{\lambda \downarrow 0} \frac{V(x + \lambda d) - V(x)}{\lambda}
\]
exists for all \(d \in \RR^n\) and \(V^o(x; d) = V'(x; d)\). If this holds for all \(x\), we say that \(V\) is regular.
\end{defn}
For a regular function, a point is Clarke stationary if and only if the directional derivative is nonnegative in all directions. For example, convex functions are regular, and strict differentiability at a point implies regularity at a point. However, for nonregular functions, \(x^*\) can simultaneously be Clarke stationary and have negative directional derivatives in a neighbourhood of directions.

\section{The discrete gradient method for nonsmooth optimisation}
\label{sec:theoretical}

In this section, we present the main theoretical results for the randomised Itoh--Abe methods. In particular, we prove that the update \eqref{eq:itoh_abe_method},
\[
x^{k+1} \mapsto x^k - \tau_{k} \beta_{k} d^{k}, \quad  \mbox{where } \beta_{k} \neq 0 \mbox{ solves } \quad \beta_{k} = - \frac{V(x^k - \tau_{k} \beta_{k} d^{k}) - V(x^k)}{\tau_{k} \beta_{k}},
\]
admits a solution for all \(\tau_k>0\). We also prove under minimal assumptions on \(V\) and \((d^k)_{k\in \NN}\) that the iterates converge to a connected set of Clarke stationary points, both in a stochastic and deterministic setting.

\subsection{Existence result}
\label{sec:existence}

\begin{lem}
\label{lem:existence}
Suppose \(V\) is a continuous function bounded below, and that \(x \in \RR^n\), \(d \in S^{n-1}\) and \(\tau > 0\). Then one of the following statements hold.
\begin{enumerate}[(i)]
	\item There is a \(\beta \neq 0\) that solves \eqref{eq:itoh_abe_method}, i.e. that satisfies \(\tfrac{V(x - \tau\beta d) - V(x)}{\tau\beta} = -\beta\).
	\item \(V\) is Clarke directionally stationary at \(x\) along \(d\).
\end{enumerate}
\end{lem}

\begin{proof}
Suppose the second statement does not hold. Then there is \(\eps > 0\) such that 
\[
\min\set{V^o(x; -d), V^o(x; d)} < -\eps,
\]
so assume without loss of generality that \(V^o(x; -d) < - \eps\). By definition of \(V^o\), there is \(\delta > 0\) such that for all \(\beta \in (0, \delta)\),
\[
\frac{V(x - \tau \beta d) - V(x)}{\tau \beta} \leq -\eps/2.
\]
Taking \(\beta \to 0\), we get that
\[
\lim_{\beta \to 0^+} \dfrac{V(x - \tau\beta d) - V(x)}{\tau \beta^2} \leq \lim_{\beta \to 0^+} -\dfrac{\eps}{2\beta} = -\infty,
\]
so there is a \(\beta_1 \in (0, \delta)\) such that 
\[
\frac{V(x - \tau \beta_1 d) - V(x)}{\tau \beta_1^2} < -1.
\]
On the other hand, as \(V\) is bounded below, we have
\[
\liminf_{\beta \to \infty} \dfrac{V(x - \tau \beta_2 d) - V(x)}{\tau \beta_2^2} \geq \lim_{\beta \to \infty} \dfrac{\min\set{0,V(x - \tau \beta_2 d)} - V(x)}{\tau \beta_2^2} = 0.
\]
Thus there is  \(\beta_2\) such that
\[
\dfrac{V(x - \tau \beta_2 d) - V(x)}{\tau \beta_2^2} > -1.
\]
Since the mapping \(\beta \mapsto \frac{V(x - \tau \beta d) - V(x)}{\tau \beta^2}\) is continuous for \(\beta \in (0, \infty)\), we conclude by the intermediate value theorem \citep[Theorem 4.23]{rud76} that there is \(\beta \in (\beta_1, \beta_2)\) that solves the discrete gradient equation
\[
\frac{V(x - \tau \beta d) - V(x)}{\tau \beta^2} = -1.
\]
\end{proof}

\begin{rem}
Note that by the above proof it is straightforward to identify an interval in which a solution to \eqref{eq:itoh_abe_method} exists, allowing for the use of standard root solver algorithms.
\end{rem}

The following lemma, which is an adaptation of \citep[Theorem 1]{gri17} for the nonsmooth setting, summarises some useful properties of the methods.
\begin{lem}
\label{lem:starting_point}
Suppose that \(V\) is continuous, bounded from below and coercive, and let \((x^k)_{k\in\NN}\) be the iterates produced by \eqref{eq:itoh_abe_method}. Then, the following properties hold.
\begin{enumerate}[(i)]
	\item \(V(x^{k+1}) \leq V(x^k)\).
	\item \(\lim_{k \to \infty} \tfrac{V(x^{k+1}) - V(x^k)}{\|x^{k+1} - x^k\|} = 0\).
	\item \(\lim_{k \to \infty} \|x^k - x^{k+1}\| = 0\).
	\item \((x^k)_{k\in\NN}\) has an accumulation point \(x^*\).
\end{enumerate}
\end{lem}

\begin{proof}
Property \emph{(i)} follows from the equation \(V(x^{k+1}) - V(x^k) = - \tau_k \beta_k^2\).

Next we show properties \emph{(ii)} and \emph{(iii)}. Since \(V\) is bounded below and \((V(x^k))_{k \in \NN}\) is decreasing, \(V(x^k) \to V^*\) for some limit \(V^*\). Therefore, by \eqref{eq:dissipation_dg}
\begin{align*}
V(x^0) - V^* = \sum_{k = 0}^\infty V(x^k) - V(x^{k+1}) &= \sum_{k = 0}^\infty \tau_k \del{\dfrac{V(x^k) - V(x^{k+1})}{\|x^{k+1} - x^k\|}}^2 \\
&\geq \tau_{\min} \sum_{k = 0}^\infty \del{\dfrac{V(x^k) - V(x^{k+1})}{\|x^{k+1} - x^k\|}}^2.
\end{align*}
Similarly, by \eqref{eq:dissipation_dg}
\begin{align*}
V(x^0) - V^* = \sum_{k = 0}^\infty V(x^k) - V(x^{k+1}) = \sum_{k = 0}^\infty \dfrac{1}{\tau_k} \|x^k - x^{k+1}\|^2 \geq \frac{1}{\tau_{\max}} \sum_{k = 0}^\infty \|x^k - x^{k+1}\|^2.
\end{align*}
We conclude
\[
\lim_{k \to \infty} \frac{V(x^k) - V(x^{k+1})}{\|x^{k+1} - x^k\|} = \lim_{k \to \infty} \|x^{k+1} - x^k\| = 0,
\]
which shows properties \emph{(ii)} and \emph{(iii)}.

Last, we show \emph{(iv)}. Since \((V(x^k))_{k \in \NN}\) is a decreasing sequence, the iterates \((x^k)_{k \in \NN}\) belong to the set \(\left\{x \in \RR^n \; : \; V(x) \leq V(x^0)\right\}\). Therefore, by coercivity of \(V\), the iterates \((x^k)_{k \in \NN}\) are bounded, and admit an accumulation point.
\end{proof}

We denote by \(S\) the limit set of \((x^k)_{k\in \NN}\), which is the set of accumulation points,
\[
S = \set{x^* \in \RR^n \; : \; \exists (x^{k_j})_{j \in \NN} \; \mbox{ s.t. } \; x^{k_j} \to x^*}.
\]
By the above lemma, \(S\) is nonempty. We now prove further properties of the limit set.

\begin{lem}
The limit set \(S\) is compact, connected and has empty interior. Furthermore, \(V\) is constant on \(S\).
\end{lem}
\begin{proof}
Boundedness of \(S\) follows from coercivity of \(V\) combined with the fact that \(S\) belongs to the level set \(\{x \in \RR^n \; : \; V(x) \leq V(x^0)\}\). Since any accumulation point of \(S\) is also an accumulation point of \((x^k)_{k \in \NN}\), \(S\) is closed. Hence \(S\) is compact.

We prove connectedness by contradiction. Suppose there are two disjoint, nonempty open sets \(A\) and \(B\) such that \(S \subset A \cup B\). The sequence \((x^k)_{k \in \NN}\) jumps between \(A\) and \(B\) infinitely many times and \(\|x^{k+1} - x^k\| \to 0\), which implies that there is a subsequence of \((x^{k_j})_{j \in \NN}\) in \(\RR^n \setminus (A \cup B)\). However, \((x^{k_j})_{j\in\NN}\) is a bounded sequence and has an accumulation point, which must belong in \(\RR^n \setminus (A \cup B)\). This contradicts the assumption that all accumulation points of \((x^k)_{k\in\NN}\) are in \(A \cup B\).

We show that \(S\) has empty interior by contradiction. Suppose \(S\) contains an open ball \(B_\eps(x)\) in \(\RR^n\). Then as \(\|x^{k+1} - x^k\| \to 0\), there is a \(j \in \NN\) such that \(x^j \in B_\eps(x) \subset S\). However, as \(V\) takes the same value on all of \(S\), we deduce that \(V(x^j) = \lim_{k \to \infty} V(x^x)\). Since \((V(x^k))_{k\in\NN}\) is a decreasing sequence, \(V(x^k) = V(x^j)\) for all \(k > j\). It follows from \eqref{eq:dissipation_dg} that \(x^k = x^j\) for all \(k > j\). Therefore, \(S = \{x^j\}\), which contradicts the assumption that \(S\) has nonempty interior.

Last, since \((V(x^k))_{k \in \NN}\) is a decreasing sequence and \(V(x^*) = \lim_{k \to \infty} V(x^k)\) for all \(x^* \in S\), it follows that \(V\) is constant on \(S\).
\end{proof}

\subsection{Optimality result}

We now proceed to the main result of this paper, namely that all points in the limit set \(S\) are Clarke stationary. We consider the stochastic case and the deterministic case separately.

In the stochastic case,  we assume that the directions \((d^k)_{k\in\NN}\) are randomly, independently drawn, and that the support of the probability density of \(\Xi\) is dense in \(S^{n-1}\). It is straightforward to extend the proof to the case where \((d^{nk+1}, \ldots, d^{n(k+1)})\) are drawn as an orthonormal system under the assumptions that the directions \((d^{nk+1})_{k \in \NN}\) are independently drawn from \(S^{n-1}\) and that the support of the density of the corresponding marginal distribution is dense in \(S^{n-1}\).

We define \(X\) to be the set of nonstationary points,
\begin{equation}
\label{eq:set_X}
X = \{x \in \RR^n \; : \;  0 \notin \partial V(x)\}.
\end{equation}

\begin{thm}
\label{thm:main}
Let \((x^k)_{k\in\NN}\) solve \eqref{eq:itoh_abe_method} where \((d^k)_{k\in \NN}\) are independently drawn from the random distribution \(\Xi\), and suppose that the support of the density of \(\Xi\) is dense in \(S^{n-1}\). Then \(\PP(S \cap X \neq \varnothing) = 0\), i.e.  the limit set \(S\) is almost surely in the set of stationary points.
\end{thm}

\begin{proof}
We will construct a countable collection of open sets \((B_j)_{j \in \NN}\), such that \(X \subset \bigcup_{j \in \NN} B_j\) and so that for all \(j \in \NN\) we have \(\PP(S \cap B_j \neq \varnothing) = 0\). Then the result follows from countable additivity of probability measures.

First, we show that for every \(x \in X\), there is \(d \in S^{n-1}\), \(\eps > 0\), and \(\delta > 0\) such that
\begin{equation}
\label{eq:property}
\frac{V(y - \lambda e) - V(y)}{\lambda} \leq -\eps, \quad \forall y \in B_{\delta}(x), \; e \in B_\delta(d) \cap S^{n-1}, \; \lambda \in (0, \delta).
\end{equation}
To show this, note that if \(x \in X\), then by definition there is \(d \in S^{n-1}\) and \(\eps > 0\) such that
\[
V^o(x; -d) = \limsup_{\substack{y \to x \\ \lambda \downarrow 0}} \frac{V(y - \lambda d) - V(y)}{\lambda} \leq -\eps.
\]
Therefore, there is \(\eta > 0\) such that for all \(\lambda \in (0, \eta)\) and all \(y \in B_\eta(x)\), we have
\[
\frac{V(y - \lambda d) - V(y)}{\lambda} \leq -\eps/2.
\]
As \(V\) is Lipschitz continuous around \(B_\eta(x)\), it is clear that the mapping
\[
e \mapsto \frac{V(y - \lambda e) - V(y)}{\lambda},
\]
is also locally Lipschitz continuous (of the same rank). It follows that there exists \(\delta \in (0, \eta)\) such that for all \(y \in B_{\delta}(x)\), all \(e \in B_\delta(d)\cap S^{n-1}\), and all \(\lambda \in (0, \delta)\), we have
\[
\frac{V(y - \lambda e) - V(y)}{\lambda} \leq -\eps/3.
\]
This concludes the first part.

Next, for \(m \in \NN \), we define the set
\[
X_m = \set{x \in X \; : \; \eqref{eq:property} \mbox{ holds for some }  d\in S^{n-1}, \eps > 0 \mbox{ and all } \delta < 1/m \;}.
\]
Clearly
\[
X = \bigcup_{m \in \NN} X_m.
\]
Let \((y^{i})_{i \in \NN}\) be a dense sequence in \(X_m\), which exists because \(\QQ^n\) is both countable and dense in \(\RR^n\). We define \(Y_i^{(m)} = B_{\delta}(y^{i})\), where \(\delta = \tfrac{1}{m+1}\). Therefore,
\[
X_m \subset \bigcup_{i \in \NN} Y_i^{(m)} \quad \implies \quad  X \subset \bigcup_{m \in \NN} \bigcup_{i \in \NN} Y_i^{(m)}.
\]
Since a countable union of countable sets is countable, we conclude with the following statement. For each \(i \in \NN\) there is \(y^i \in \RR^n\), \(\eps_i, \delta_i > 0\), and \(d^i \in S^{n-1}\), such that for all \(z \in B_{\delta_i}(y^i)\), all \(e \in B_{\delta_i}(d^i) \cap S^{n-1}\), and all \(\lambda \in (0, \delta_i)\), we have
\[
\frac{V(z - \lambda e) - V(z)}{\lambda} \leq -\eps_i,
\]
and such that
\begin{equation}
\label{eq:cover_X}
X \subset \bigcup_{i \in \NN} B_{\delta_i}(y^i).
\end{equation}

Finally, we show that for each \(i \in \NN\), almost surely, \(S \cap B_{\delta_i}(y^i) = \varnothing\). For a given \(i\), write \(B_i := B_{\delta_i}(y^i)\), and define \(m := \min_{x \in B_j}V(x)\), \(M := \max_{x \in B_j}V(x)\). We argue accordingly: The existence of an accumulation point of \((x^k)_{k \in \NN}\) in \(B_i\) would imply that there is a subsequence \((x^{k_j})_{j \in \NN} \subset B_i\). Suppose \(x^{k_j} \in B_i\) and \(d^{k_j} \in B_{\delta_i}(d^i)\), so that \(x^{k_j+1} = x^{k_j} - \lambda d^{k_j}\) for some \(\lambda > 0\). If \(\lambda < \delta_i\), then 
\[
V(x^{k_j} - \lambda d^{k_j}) - V(x^{k_j}) \leq -\eps_i \lambda = -\eps_i \|x^{k_j+1} - x^{k_j}\|.
\]
However,
\[
V(x^{k_j+1}) - V(x^{k_j}) = - \frac{1}{\tau_{k_j}} \|x^{k_j+1} - x^{k_j}\|^2,
\]
so, combining these equations, we get
\[
\eps_i \tau_{k_j} \leq \|x^{k_j+1} - x^{k_j}\|.
\]
This in return implies
\[
V(x^{k_j}) - V(x^{k_j+1}) \geq \eps_i^2 \tau_{\min}.
\]
On the other hand, if \(\lambda \geq \delta_i\), then
\[
V(x^{k_j}) - V(x^{k_j+1}) \geq \frac{\delta_i^2}{\tau_{\max}}.
\]
Setting \(\mu = \min\set{\eps_i^2 \tau_{\min}, \tfrac{\delta_i^2}{\tau_{\max} }}\), it follows that whenever \(x^{k_j} \in B_i\) and \(d^{k_j} \in B_{\delta_i}(d^i)\), then
\[
V(x^{k_j}) - V(x^{k_j+1}) \geq \mu.
\] 
Choosing \(K \in \NN\) such that \(K \mu > M - m\), we know that this event only has to occur \(K\) times for \((x^{k_j})_{j \in \NN}\) to leave \(B_i\). In other words, almost surely, there is no subsequence \((x^{k_j})_{j\in\NN} \subset B_i\). This concludes the proof.
\end{proof}

\subsubsection{Deterministic case}

We now cover the deterministic case, in which \((d^k)_{k\in\NN}\) is required to be \emph{cyclically dense}.
\begin{defn}
\label{defn:cyclically_dense}
A sequence \((d^k)_{k \in \NN} \subset S^{n-1}\) is \emph{cyclically dense} in \(S^{n-1}\) if, for all \(\eps > 0\), there is \(N \in \NN\) such that for any \(k \in \NN\), the set \(\set{d^{k}, \ldots, d^{k+N-1}}\) forms an \(\eps\)-cover of \(S^{n-1}\),
\[
S^{n-1} \subset \bigcup_{i = k+1}^{k+N-1} B_\eps(d^i).
\]
\end{defn}
\begin{rem}
Randomly drawn sequences are almost surely not cyclically dense, hence the separate treatment of the stochastic and deterministic methods.
\end{rem}
Many constructions of dense sequences are also cyclically dense. We provide an example of such a sequence on the unit interval \([0,1]\).
\begin{ex}
Let \(\sigma \in (0,1)\) be an irrational number and define the sequence \((\lambda_k)_{k \in \NN}\) in \([0,1]\) by
\[
\lambda_k = (\sigma k) \quad (\Mod \;\; 1) = \sigma k - \left \lfloor{\sigma k}\right \rfloor,
\]
where \(\left \lfloor{\sigma k}\right \rfloor\) denotes the largest integer less than or equal to \(\sigma k\).

To see that \((\lambda_k)_{k \in \NN}\) is cyclically dense in \([0,1]\), set \(\eps > 0\) and note by sequential compactness of \([0,1]\) that there is \(k, r \in \NN\) such that \(|\lambda_{k} - \lambda_{k+r}| < \eps\). We can write \(\delta = |\lambda_k - \lambda_{k+r}| > 0\), where we know that \(\delta\) is strictly positive, as no value can be repeated in the sequence due to \(\sigma\) being irrational. By modular arithmetic, we have for any \(l \in \NN\),
\[
\lambda_{k+rl} = \lambda_k + l\delta \quad (\Mod \;\; 1).
\]
In other words, the subsequence \((\lambda_{k+rl})_{l \in \NN}\) moves in increments of \(\delta< \eps\) on \([0,1]\). Setting \(N = r \left \lceil{\tfrac{1}{\delta}}\right \rceil + k\),  where \(\left \lceil{\delta}\right \rceil\) denotes the smallest integer greater than or equal to \(\delta\),  it is clear that for any \(j \in \NN\), the set \(\{\lambda_{j}, \lambda_{j+1}, \ldots, \lambda_{j+N-1}\}\) forms an \(\eps\)-cover of \([0,1]\).

One could naturally extend this construction to higher dimensions \([0,1]^n\), by choosing \(n\) irrational numbers such that the ratio of any two of these numbers is also irrational.
\end{ex}

\begin{thm}
Let \((x^k)_{k\in\NN}\) solve \eqref{eq:itoh_abe_method}, where \((d^k)_{k \in \NN}\) are cyclically dense. Then all accumulation points \(x^* \in S\) satisfy \(0 \in \partial V(x^*)\).
\end{thm}
\begin{proof}
We consider the setup in the proof to \thmref{thm:main}, where \(X\) is the set of nonstationary points \eqref{eq:set_X} and is covered by a countable collection of open balls \eqref{eq:cover_X},
\[
X \subset \bigcup_{i \in \NN} B_{\delta_i}(y^i).
\]
We will show that an accumulation point \(x^* \in S\) cannot belong to the ball \(B_{\delta_i}(y^i)\), from which it follows that \(S\) is a subset of the set of stationary points. For contradiction, suppose that there is a subsequence \((x^{k_j})_{j \in \NN} \to x^* \in B_{\delta_i}(y^i)\).  By \lemref{lem:starting_point} \emph{(iii)}, since \(\|x^k - x^{k+1}\| \to 0\) as \(k \to \infty\), we deduce that for any \(N \in \NN\), there is \(j \in \NN\) such that
\[
\{ x^{k_j}, x^{k_j+1},  \ldots, x^{k_j+N-1}\} \subset B_{\delta_i}(y^i).
\]
By cyclical density, we can choose \(N\) such that the corresponding directions \(\{ d^{k_j}, d^{k_j+1},  \ldots, d^{k_j+N-1} \}\) form an \(\eps_i\)-cover of \(S^{n-1}\). Therefore, there exists \(x^{k} \in B_{\delta_i}(y^i)\) and \(d^k \in B_{\delta_i}(e^i)\), so we can argue as in \thmref{thm:main}, that
\[
V(x^{k}) -  V(x^{k+1}) \geq \mu,
\]
where \(\mu = \min\set{\eps_i^2 \tau_{\min}, \tfrac{\delta_i^2}{\tau_{\max} }}\). If \((x^{k_j})_{j \in \NN}\) had a limit in \(B_{\delta_i}(y^i)\), this would happen arbitrarily many times, which is a contradiction. This concludes the proof.
\end{proof}

\subsection{Necessity of search density and Lipschitz continuity}
\label{sec:direction_necessity}

For nonsmooth problems, it is necessary to employ a set of directions \((d^k)_{k \in \NN}\) larger than the set of basis coordinates \(\{e^1,\ldots, e^n\}\). To see this, consider the function \(V(x,y) = \max\set{x,y}\) and the starting point \(x^0 = [1,1]^T\). With the standard Itoh--Abe discrete gradient method, the iterates would remain at \(x^0\), even though this point is nonstationary.

We show with a simple example that the assumption of density of \((d^k)_{k \in \NN}\) in \thmref{thm:main} is not only sufficient, but also necessary.
\begin{ex}
\label{ex:narrow_descent}
We suppose \(V: \RR^2 \to \RR\) is defined by \(V(x_1,x_2) = |x_1| + N |y_2|\) for some \(N \in \NN\), and set \(x^0 = [-1,0]^T\). For \(\theta \in [-\pi/2, \pi/2]\), let \(d = [\cos\theta, \sin\theta]^T\). Then \(-d\) is a direction of descent if and only if \(\theta \in (-\arctan(1/N), \arctan(1/N))\). This interval can be made arbitrarily small by choosing \(N\) to be sufficiently large. Therefore, for an Itoh--Abe method to descend from \(x^0\) for arbitrary functions, the directions \((d^k)_{k \in \NN}\) need to include a convergent subsequence to the direction \([1,0]^T\). As this direction is arbitrary, we deduce that \((d^k)_{k \in \NN}\) must be dense.
\end{ex}

\thmref{thm:main} also assumes that \(V\) is locally Lipschitz continuous. We briefly discuss why this assumption is necessary, and provide an example to show that for functions that are merely continuous, the theorem no longer holds.

By Proposition 2.1.1. (b) in \citep{cla90}, the mapping \((y, d) \mapsto V^o(y; d)\) is upper semicontinuous for \(y\) in a neighbourhood of \(x\), due to the local Lipschitz continuity of \(V\) near \(x\). That is,
\[
V^o(y^*; d^*) \geq \limsup_{y \to y^*, d \to d^*} V^o(y; d).
\]
This property is crucial for the convergence analysis of Itoh--Abe methods, as it implies
\[
V^o(x^*, d^*) \geq \limsup_{k \in \NN} V^o(x^k; d^k) = 0.
\]
Without local Lipschitz continuity, it is possible to have
\[
x^k \to x^*, \quad d^k \to d^*, \; \mbox{and } V^o(x^k; d^k) \to 0, \qquad \mbox{but } V^o(x^*; d^*) < 0.
\]
In this case, there is no guarantee that the limit \(x^*\) is Clarke stationary. We demonstrate this with an example.
\begin{ex}
We will first state the iterates \((x^k)_{k\in\NN}\) and then construct a function \(V:\RR^2 \to \RR\) that fits these iterates. Let \((d^k)_{k \in \NN}\) be a cyclically dense sequence in \(S^{1}\) and assume without loss of generality that \([0, 1]^T \notin (d^k)_{k \in \NN}\). Replacing \(d^k\) with \(-d^k\) does not change the step in \eqref{eq:itoh_abe_method}, so we assume that \(d^k_1 < 0\) for all \(k\). We set \(x^0 = [0,0]^T\) and define \((x^k)_{k\in \NN}\) and \((V(x^k))_{k\in\NN}\) to be
\[
x^{k+1} = x^k - \frac{1}{(k+1)^2} d^k, \qquad V(x^{k+1}) = V(x^k) - \frac{1}{(k+1)^4}, \qquad V(x^0) = 0.
\]
Since \(\sum_{k \in \NN} \|x^k - x^{k+1}\| < \infty\), it follows that \(x^k\) converges to some limit \(x^*\), and \(V(x^k)\) clearly decreases to a limit \(V^* \in \RR\). Furthermore, these steps satisfy \eqref{eq:itoh_abe_method} with \(\tau_k = 1\). We then define \(V\) on the line segments \([x^k, x^{k+1}] = \{\lambda x^k + (1-\lambda)x^{k+1} \; : \; \lambda \in [0,1]\}\) by interpolating linearly from \((x^k, V(x^k))\) to \((x^{k+1}, V(x^{k+1}))\).

Next, we define \(V\) on \(\RR^2\) as a function that linearly decreases everywhere in the direction \([0,1]^T\), and so that its value is consistent with the values given on the predefined line segments \([x^k, x^{k+1}]\). Note that this is a well-defined and continuous function, since each line in the direction \([0,1]^T\) crosses at most one point on at most one line segment, due to our assumptions on \((d^k)_{k\in\NN}\).

We conclude the example by noting that the limit \(x^*\) is not Clarke stationary---in fact, no point is Clarke stationary---since \(V^o(x; [0, 1]^T) = -1\) for all \(x\).
\end{ex}

\subsection{Nonsmooth, nonconvex functions with further regularity}

For a large class of nonsmooth optimisation problems (convex and nonconvex), the objective function is sufficiently regular so that the standard Itoh--Abe discrete gradient method is also guaranteed to converge to Clarke stationary points.  These are functions \(V\) for which \(x^* \in \RR^n\) is Clarke stationary if and only if \(V^o(x^*; e^i) \geq 0\) for \(i = 1, \ldots, n\). One example is functions of the form
\[
V(x) = E(x) + \lambda \|x\|_1,
\]
where \(E\) is a continuously differentiable function that may be nonconvex, \(\|x\|_1\) denotes \(|x_1| + \ldots + |x_n|\), and \(\lambda > 0\). See for example Proposition 2.3.3 and the subsequent corollary in \citep{cla90}, combined with the fact that the nonsmooth component of \(V\), i.e. \(\|\cdot\|_1\), separates into \(n\) coordinate-wise scalar functions. This implies that the Clarke subdifferential is given by
\[
\partial V(x) = \{\nabla E(x)\} + \lambda \prod_{i=1}^n \sgn(x_i),
\]
where \(\prod\) denotes the Cartesian product and
\[
\sgn(x_i) := \begin{cases}
\{1\}, \quad &\mbox{if } x_i > 0, \\
\{-1\}, \quad &\mbox{if } x_i < 0, \\
[-1,1], \quad &\mbox{if } x_i = 0.
\end{cases}
\]

Since this paper is chiefly concerned with the blackbox setting where no particular structure of \(V\) is assumed, we do not include a rigorous analysis of the convergence properties of the standard Itoh--Abe discrete gradient method for functions of the above form. However, we point out that for nonsmooth, nonconvex optimisation problems where Clarke stationarity is equivalent to Clarke directional stationarity along the standard coordinates, one can adapt \thmref{thm:main} in a straightforward manner to prove that the iterates converge to a set of Clarke stationary points when the directions \((d^k)_{k \in \NN}\) are drawn from the standard coordinates \((e^i)_{i=1}^n\).

Furthermore, one could drop the requirement that \(V\) is locally Lipschitz continuous, and replace \(\|x\|_1\) with \(\|x\|_p^p\), where \(p \in (0,1)\), and \(\|x\|_p^p = |x_1|^p + \ldots + |x_n|^p\). This too is beyond the scope of this paper.

\section{Rotated Itoh--Abe discrete gradients}
\label{sec:rotated}

We briefly discuss a randomised Itoh--Abe method that retains the Itoh--Abe discrete gradient structure, by ensuring that the directions \((d^{kn+1}, d^{kn+2}, \ldots, d^{k(n+1)})\) are orthonormal. Equivalently, we consider each block of \(n\) directions to be independently drawn from a random distribution on the set of orthogonal transformations on \(\RR^n\), denoted by \(\mbox{O}(n)\).
\begin{defn}
The \emph{orthogonal group of dimension \(n\)}, \(\mbox{O}(n)\), is the set of orthogonal matrices in \(\RR^n\), i.e. matrices \(R\) which satisfy \(R^{-1} = R^T\). Equivalently, \(R\) maps one orthonormal basis of \(\RR^n\) to another.
\end{defn}
Each element of \(\mbox{O}(n)\) corresponds to a \emph{rotated} Itoh--Abe discrete gradient.
\begin{defn}[Rotated Itoh--Abe discrete gradient]
Suppose \(R \in \mbox{O}(n)\) maps the basis \((e^i)_{i =1}^n\) to another orthonormal basis \((f^i)_{i =1}^n\), i.e. \(Rf^i = e^i\). For continuously differentiable functions \(V\), the \emph{rotated Itoh--Abe discrete gradient}, denoted by \(\overline{\nabla}_R V\), is given by
\[
\overline{\nabla}_R V(x, y) = R^T \hat \nabla_R V(x,y),
\]
where
\[
\left(\hat \nabla_R V(x, y)\right)_i := \dfrac{V\del{x + \sum_{j=1}^i \inner{y-x, f^j}f^j} - V\del{x + \sum_{j=1}^{i-1} \inner{y-x, f^j}f^j}}{\inner{y-x, f^i}}.
\]
\end{defn}

It is straightforward to check that it is a discrete gradient, as defined for continuously differentiable functions \(V\).
\begin{prop}
\(\overline{\nabla}_R V\) is a discrete gradient.
\end{prop}

\begin{proof}
For any \(x, y \in \RR^n\), \(x \neq y\),
\begin{align*}
\inner{\overline{\nabla}_R V(x,y&), y - x} \\
&= \inner{R^T \hat \nabla_R V(x,y), y - x} \\
&= \inner{\hat \nabla_R V(x,y), R(y - x)} \\
&= \sum_{i = 1}^n  \dfrac{V\del{x + \sum_{j=1}^i \inner{y-x, f^j}f^j} - V\del{x + \sum_{j=1}^{i-1} \inner{y-x, f^j}f^j}}{\inner{y-x, f^i}} \cdot \inner{y-x, f^i} \\
&= \sum_{i = 1}^n  V\del{x + \sum_{j=1}^i \inner{y-x, f^j}f^j} - V\del{x + \sum_{j=1}^{i-1} \inner{y-x, f^j}f^j} \\
&= V(y) - V(x).
\end{align*}
The convergence property \(\lim_{y \to x} \overline{\nabla}_R V(x,y) = \nabla V(x)\) is immediate, providing \(V\) is continously differentiable.
\end{proof}

Thus, we can implement schemes that are formally discrete gradient methods, and also fulfill the convergence theorem in \secref{sec:theoretical}.

\section{Numerical implementation}

\label{sec:implementation}

We consider three ways of choosing \((d^k)_{k \in \NN}\).
\begin{enumerate}
	\item \emph{Standard Itoh--Abe method.} The directions cycle through the standard coordinates, with the rule \(d^k = e^{[(k - 1) \Mod n] + 1}\). Performing \(n\) steps of this method is equivalent to one step with the standard Itoh--Abe discrete gradient method. 
	\item \emph{Random pursuit.} The directions are independently drawn from a random distribution \(\Xi\) on \(S^{n-1}\). We assume that the support of the density of \(\Xi\) is dense in \(S^{n-1}\).
	\item \emph{Rotated Itoh--Abe method.} For each \(k \in \NN\), the block of \(n\) directions \((d^{kn+1}, d^{kn+2}, \ldots, d^{(k+1)n})\) is drawn from a random distribution on \(\mbox{O}(n)\), the orthogonal group of dimension \(n\). In other words, the directions form an orthonormal basis. This retains the discrete gradient structure of the standard Itoh--Abe discrete gradient method. We assume that each draw from \(\mbox{O}(n)\) is independent, and, for notational continuity, we denote by \(\Xi\) the marginal distribution of \(d^{kn+1}\) on \(S^{n-1}\), and again assume that the support of the density is dense in \(S^{n-1}\).
\end{enumerate}

We formalise an implementation of randomised Itoh--Abe methods with two algorithms, an inner and an outer one. \algoref{algo:inner} is the inner algorithm and solves \eqref{eq:itoh_abe_method} for \(x^{k+1}\), given \(x^k\), \(d^k\) and time step bounds \(\tau_{\min}, \tau_{\max}\).  \algoref{algo:outer} is the outer algorithm, which calls the inner algorithm for each iterate \(x^{k}\), and provides a stopping rule for the methods. The stopping rule in \algoref{algo:outer}  takes two positive integers \(K\) and \(M\) as parameters, such that the algorithm stops either after \(K\) iterations, or when the iterates have not sufficiently decreased \(V\) in the last \(M\) iterations. We typically set \(M \approx n\), \(n\) being the dimension of the domain, unless the function \(V\) is expected to be highly irregular or nonsmooth, in which case we choose a larger \(M\), as directions are generally prone to yield insufficient decrease. This stopping rule can be replaced by any other heuristic.

\algoref{algo:inner} is a tailormade scalar solver for \eqref{eq:itoh_abe_method} that balances the tradeoff between optimally decreasing \(V\) given constraints \(\tau_{\min}, \tau_{\max}\) and using minimal function evaluations. Rather than solving for a given \(\tau_k\), it ensures that there exists \(\tau_k \in [\tau_{\min}, \tau_{\max}]\) that matches the output \(x^{k+1}\). It requires a preliminary \(\tau \in [\tau_{\min}, \tau_{\max}]\), which we heuristically chose as \(\tau = \sqrt{\tau_{\min} \tau_{\max}}\). This method is particularly suitable when \(\tau_{\min} \ll \tau_{\max}\), and can be replaced by any other scalar root finder algorithm. 

\textcolor{red}{The randomised Itoh--Abe methods have been implemented on Python and will be made available on GITHUB upon acceptance of this manuscript.}

\begin{algorithm}[!ht]
\caption{Randomised Itoh--Abe method with solver and stopping criterion}
    \textbf{Input:}
        starting point \(x^0\),
        directions $(d^k)_{k \in \NN}$,
        time step bounds $(\tau_{\min}, \tau_{\max})$,
        tolerance for function reduction $\eta$,
        maximal number of iterations $K$,
        maximal number of consecutive directions without descent before stopping $M$,
        internal solver described by \algoref{algo:inner}.
     \textbf{Initialise:} counter \(m = 0\).
\begin{algorithmic}
\algrule
    \For{$k=0,\ldots,K-1$}
		\State Update \(x^{k+1} \mapsfrom (x^k, d^k, \tau_{\min}, \tau_{\max})\) via \algoref{algo:inner}
		\If{\(V(x^{k}) - V(x^{k+1}) \leq \eta\)}
			\State \(m = m+1\)
		\Else
			\State \(m = 0\)
		\EndIf
		\If{\(m \geq M\)}
			\State Terminate
		\EndIf
  \EndFor
\end{algorithmic}
\label{algo:outer}
\end{algorithm}

\begin{algorithm}[!ht]
\caption{Solver for Itoh--Abe step \eqref{eq:itoh_abe_method}}
    \textbf{Input:}
        current point \(x\),
        direction $d$,
        time step upper bound $\tau_{\max}$, 
        time step lower bound $\tau_{\min}$,
        predicted time step \(\tau = \sqrt{\tau_{\min} \tau_{\max}}\),
        tolerance for \(x\), $\eps$,
        scalar \(\sigma \in (0,1)\).
\begin{algorithmic}
\algrule
    \If{\(V(x + \eps d) \geq V(x)\)}
		\State \(d = - d\)
		\If{\(V(x + \eps d) \geq V(x)\)}
			\State \Return \(x\) (stationary along \(d\))
		\EndIf
    \EndIf
    \State 
    Solve for \(\beta\) assuming linear extrapolation of \(V\) and with predicted \(\tau\) (assume for simplicity \(\beta > \eps\)):
    \begin{align*} 
         \beta &= - \frac{V(x + \eps d) - V(x)}{\eps \tau} \\
    		x^0 &= x, \quad x_1 = x + \eps d, \quad x_2 = x + \beta d
    \end{align*}
    \While{\(V\) is concave between \(x^0\), \(x^1\) and \(x^2\) (meaning \(\tfrac{V(x^2) - V(x^1)}{x^2 - x^1} \leq \tfrac{V(x^1) - V(x^0)}{x^1 - x^0}\))}
    		\State \(\beta = \beta/\sigma\), \(x^2 = x + \beta d\).
    	\EndWhile
    	\State Do step of parabolic interpolation (see \citep[Section 6.2.2]{hea02}) between \(x^0\), \(x^1\) and \(x^2\), i.e.
    	\[
    	y = x^1 - \frac{1}{2}\frac{(x^1-x^0)^2 (V(x^1) - V(x^2)) - (x^1 -  x^2)^2 (V(x^1) - V(x^0)) }{ (x^1-  x^0)(V(x^1) - V(x^2)) - (x^1 - x^2)(V(x^1)-V(x^0))}
    	\]
    	\While{Parabolic step has not decreased \(V\)}
		\State Update parabolic interpolation points \(x^i\), \(i = 0, 1, 2\).
    	\EndWhile
    	\State \(y = x^i\) is optimal point from parabolic interpolation step
    	\While{\(\frac{|V(y) - V(x)|}{\|y-x\|^2} \notin \sbr{1/\tau_{\max}, 1/\tau_{\min}}\)}
    		\If{\(\frac{|V(y) - V(x)|}{\|y-x\|^2} > 1/\tau_{\min}\)}
    			\State \(y = y/\sigma\)
    		\Else
			\State \(y = \sigma y\)
    		\EndIf
    	\EndWhile
	\State \Return \(y\)
    	\end{algorithmic}
\label{algo:inner}
\end{algorithm}

\section{Examples}
\label{sec:examples}

In this section, we use the randomised Itoh--Abe methods to solve several nonsmooth, nonconvex problems. In \secref{sec:rosenbrock}, we consider some well known optimisation challenges developed by Rosenbrock and Nesterov. In \secref{sec:bilevel}, we solve bilevel optimisation of parameters in variational regularisation problems.\footnote{Test images are taken from the Berkeley database \citep{mar01}. Available online:  \url{https://www2.eecs.berkeley.edu/Research/Projects/CS/vision/bsds/BSDS300/html/dataset/images.html}.}

We compare our method to state-of-the-art derivative-free optimisation methods Py-BOBYQA \citep{car18, pow09} and the LT-MADS solver provided by NOMAD \citep{aud06, led09, led11}. For purposes of comparing results across solvers for these problems, we do not measure objective function value against iterates, but objective function value against function evaluations.

\subsection{Rosenbrock functions}
\label{sec:rosenbrock}

We consider the well-known Rosenbrock function \citep{ros60}
\begin{equation}
\label{eq:rosenbrock}
V(x,y) = (1-x)^2 + 100(y-x^2)^2.
\end{equation}
Its global minimiser \([1,1]^T\) is located in a narrow, curved valley, which is challenging for the iterates to navigate. We compare the three variants of the Itoh--Abe method, for which we set the algorithm parameters \(\eps = 10^{-5}\), \(\tau_{\min} = 10^{-4}\), \(\tau_{\max} = 10^2\), \(\eta = 10^{-9}\), and \(M = 30\). See \figref{fig:rosenbrock} for the numerical results. All three methods converge to the global minimiser, which shows that the Itoh--Abe methods are robust. Unsurprisingly, the random pursuit method and the rotated Itoh--Abe method, which descend in varying directions, perform significantly better than the standard Itoh--Abe method.

\begin{figuretmp}
\begin{center}
\begin{tikzpicture}
\draw (0,0) node[inner sep = 0] {\includegraphics[height=5cm]{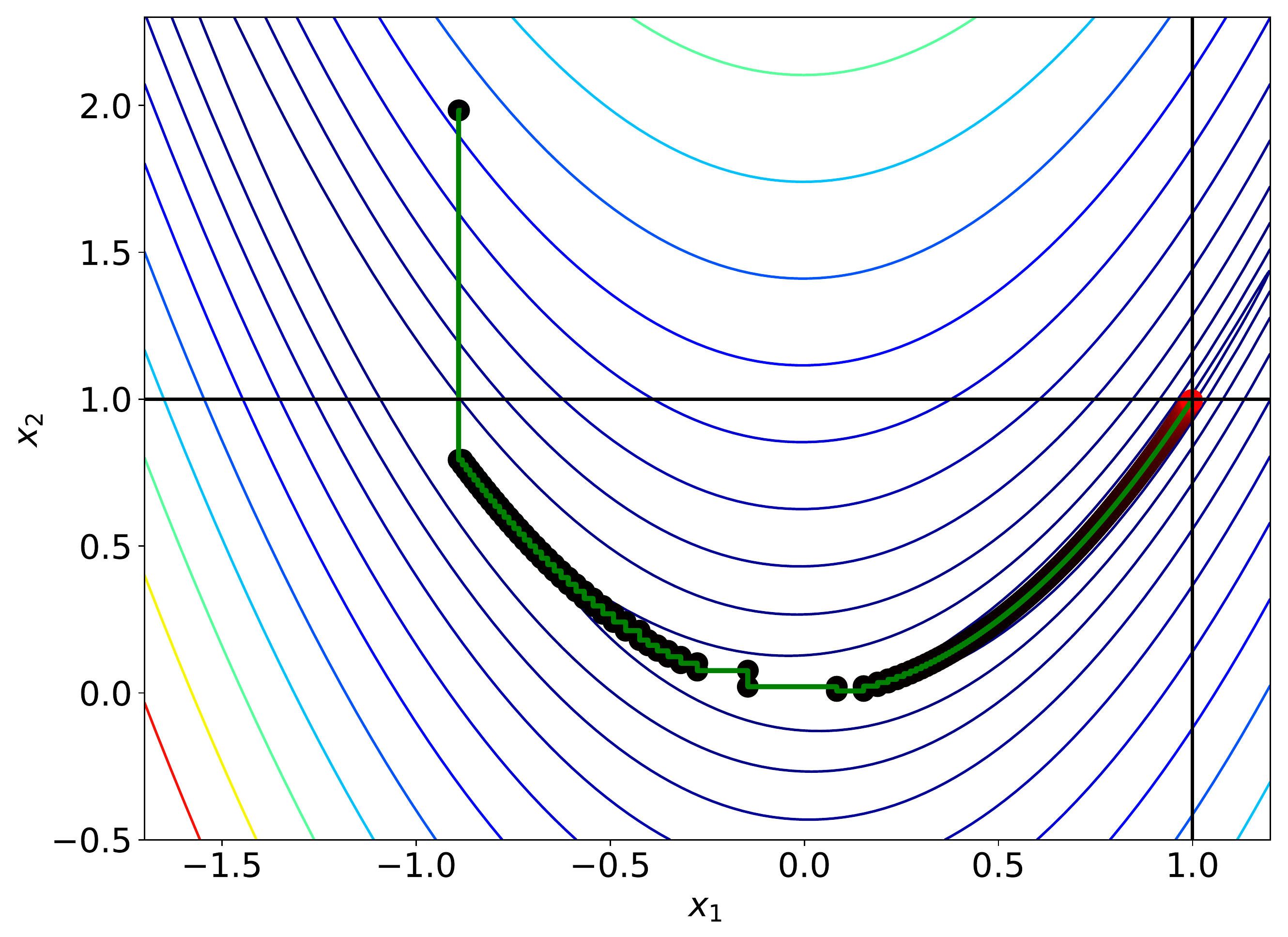}};
\node[draw, fill = white] at (-1.5,-1.7) {\scriptsize Standard Itoh--Abe} ;
\end{tikzpicture}
\begin{tikzpicture}
\draw (0,0) node[inner sep = 0] {\includegraphics[height=5cm]{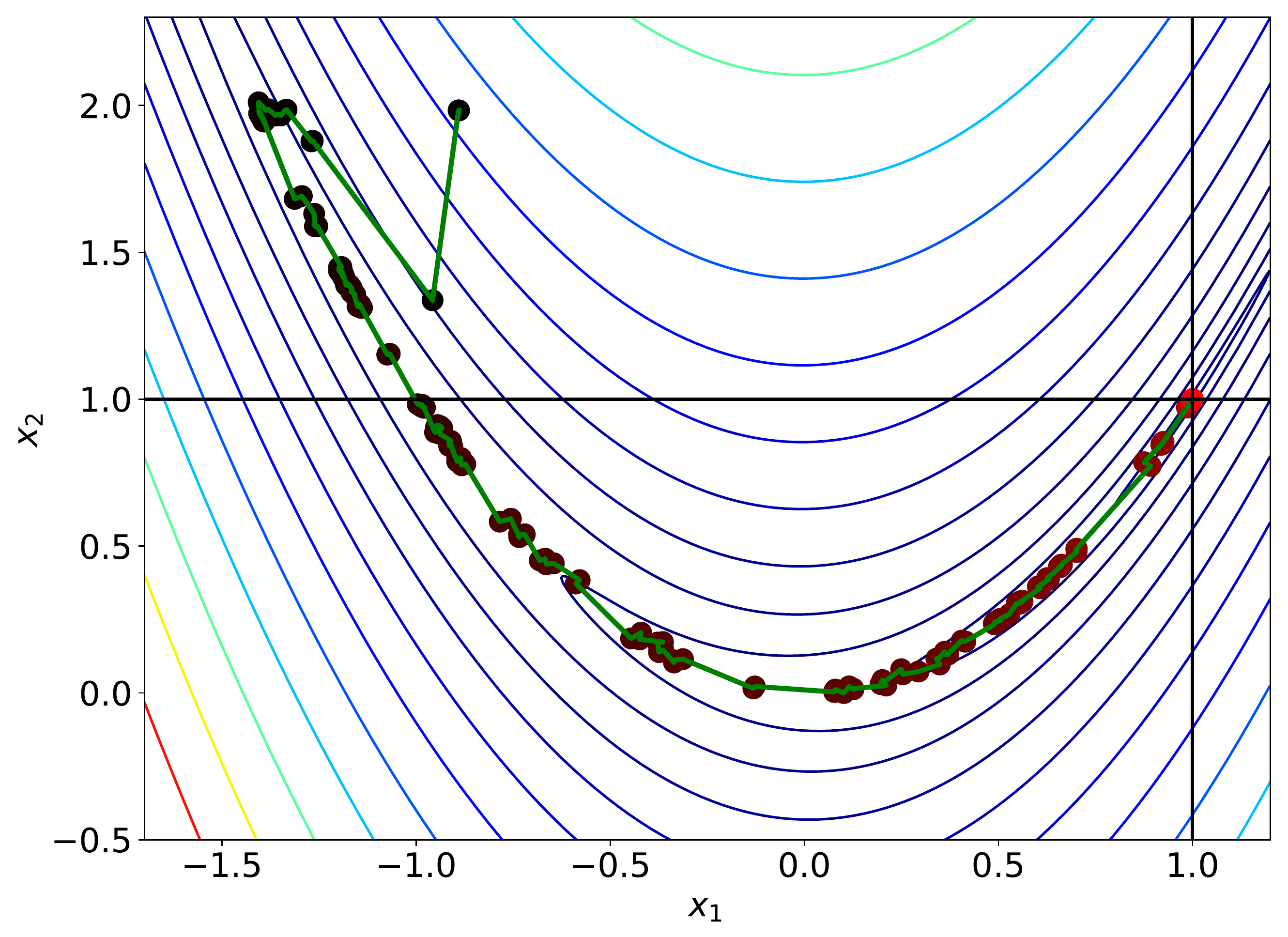}};
\node[draw, fill = white] at (-1.5,-1.7) {\scriptsize Rotated Itoh--Abe} ;
\end{tikzpicture}
\begin{tikzpicture}
\draw (0,0) node[inner sep = 0] {\includegraphics[height=5cm]{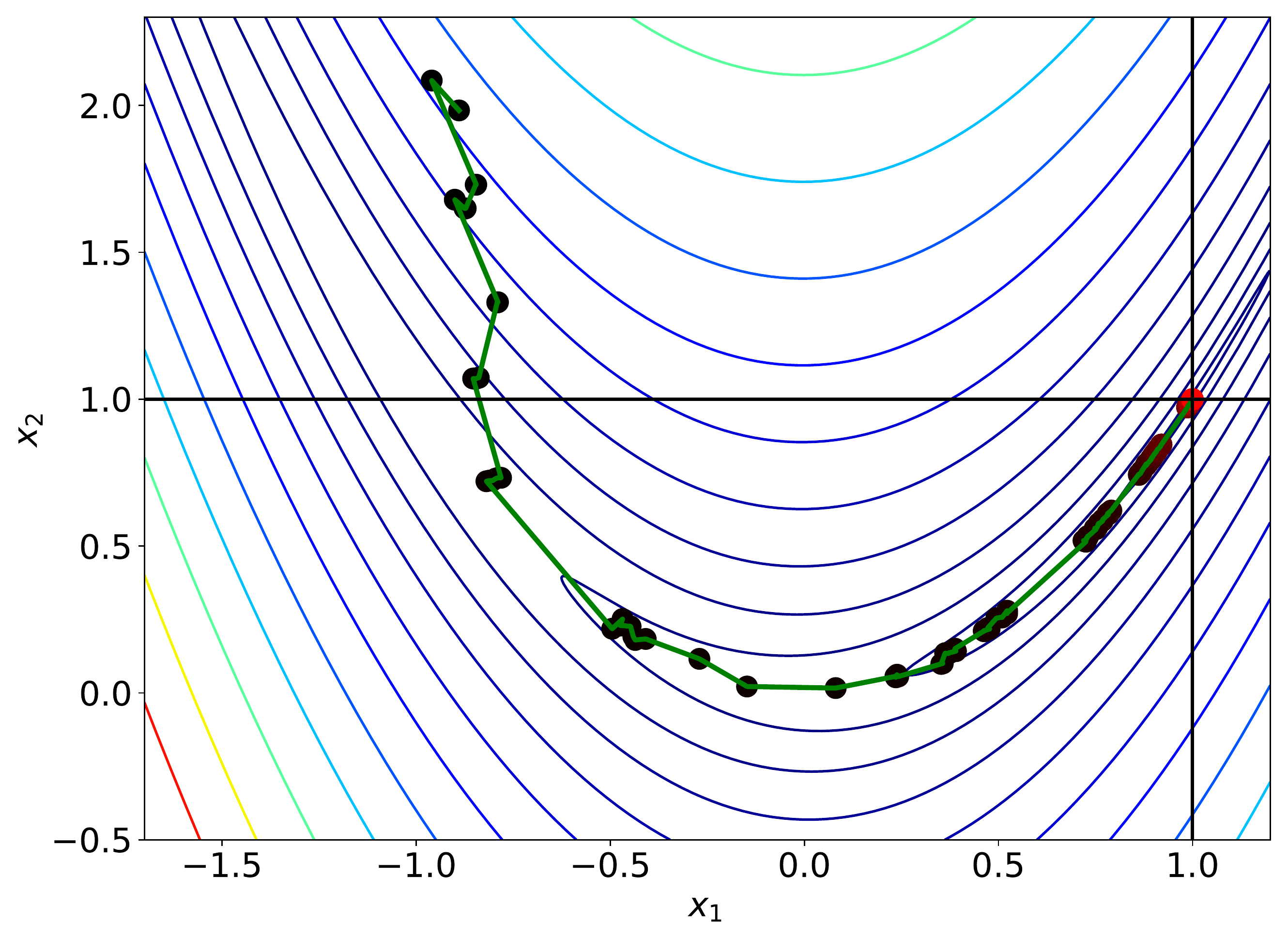}};
\node[draw, fill = white] at (-1.12,-1.7) {{\scriptsize Random pursuit Itoh--Abe}} ;
\end{tikzpicture}
\includegraphics[height=5cm]{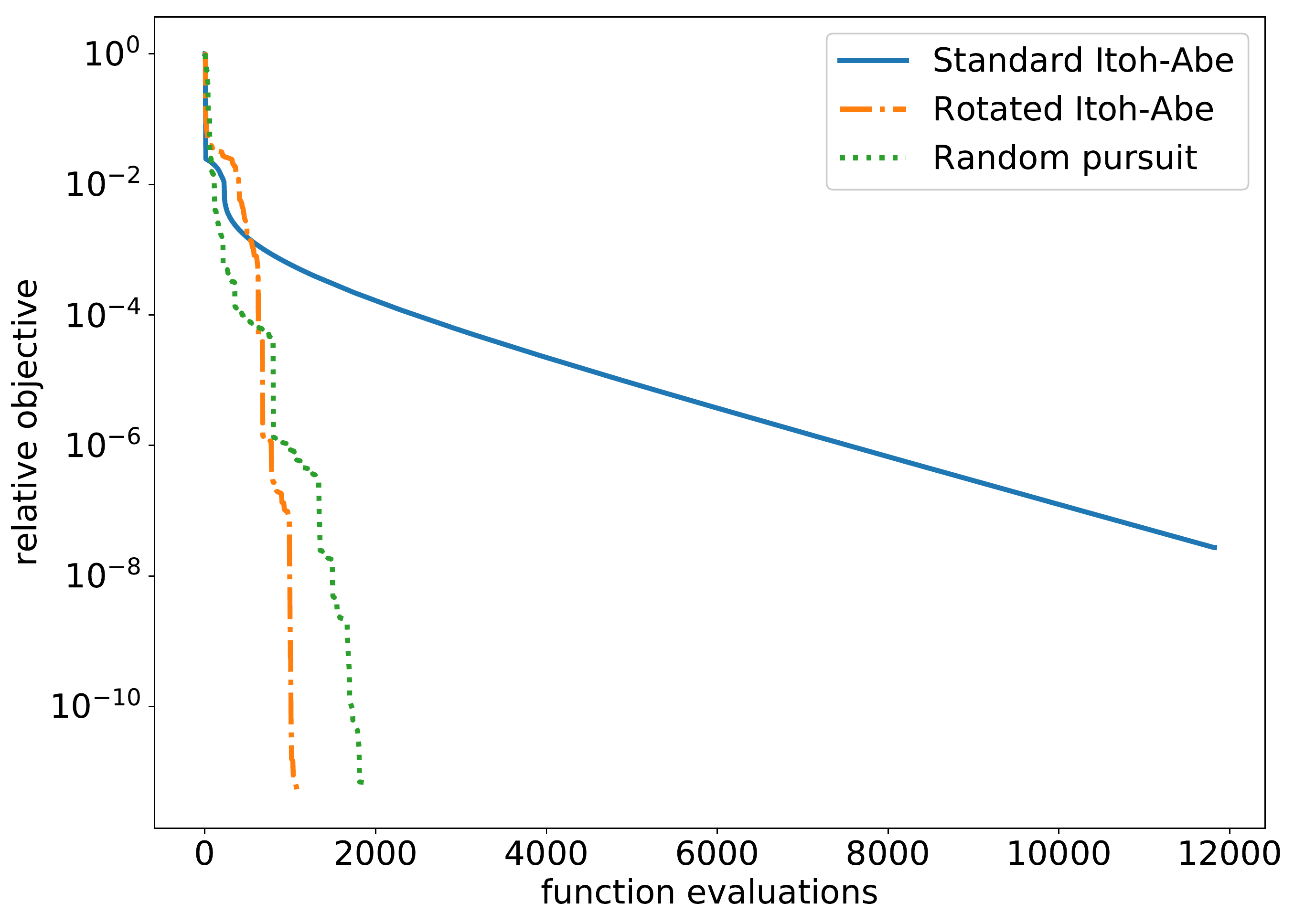}
\end{center}
\caption{Comparison of three variants of the Itoh--Abe method applied to the Rosenbrock function. Top left: Itoh--Abe method with standard frame. Top right: Rotated Itoh--Abe method. Bottom left: Itoh--Abe method with random pursuit. Bottom right: Convergence rates of the relative objective \(\frac{V(x^k) - V^*}{V(x^0) - V^*}\) for the three variants, displayed with a log--log plot.}
\label{fig:rosenbrock}
\end{figuretmp}

We also consider a nonsmooth variant of \eqref{eq:rosenbrock}, termed Nesterov's (second) nonsmooth Chebyshev--Rosenbrock function \citep{gur12},
\begin{equation}
\label{eq:nesterov}
V(x,y) = \frac{1}{4} |x-1| + \envert{y - 2|x| + 1}.
\end{equation}
In this case too, the global minimiser \([1,1]^T\) is located along a narrow path. Furthermore, there is a nonminimising, stationary point at \([0, -1]^T\), which is nonregular---i.e. it has negative directional derivatives).

We also compare the three Itoh--Abe methods for this example, and set the algorithm parameters \(\eps = 10^{-10}\), \(\tau_{\min} = 10^{-4}\), \(\tau_{\max} = 10^2\), \(\eta = 10^{-16}\), and \(M = 100\). See \figref{fig:nesterov_dg} for the results from this. As can be seen, the standard Itoh--Abe discrete gradient method is not suitable for the irregular paths and nonsmooth kinks of the objective function, and stagnates early on. The two randomised Itoh--Abe methods perform better, as they descend in varying directions. For the remaining 2D problems in this paper, we will consider the rotated Itoh--Abe method, although we could just as well have used the random pursuit method. For higher-dimensional problems, we recommend the random pursuit method.

\begin{figuretmp}
\begin{center}
\begin{tikzpicture}
\draw (0,0) node[inner sep = 0] {\includegraphics[height=5cm]{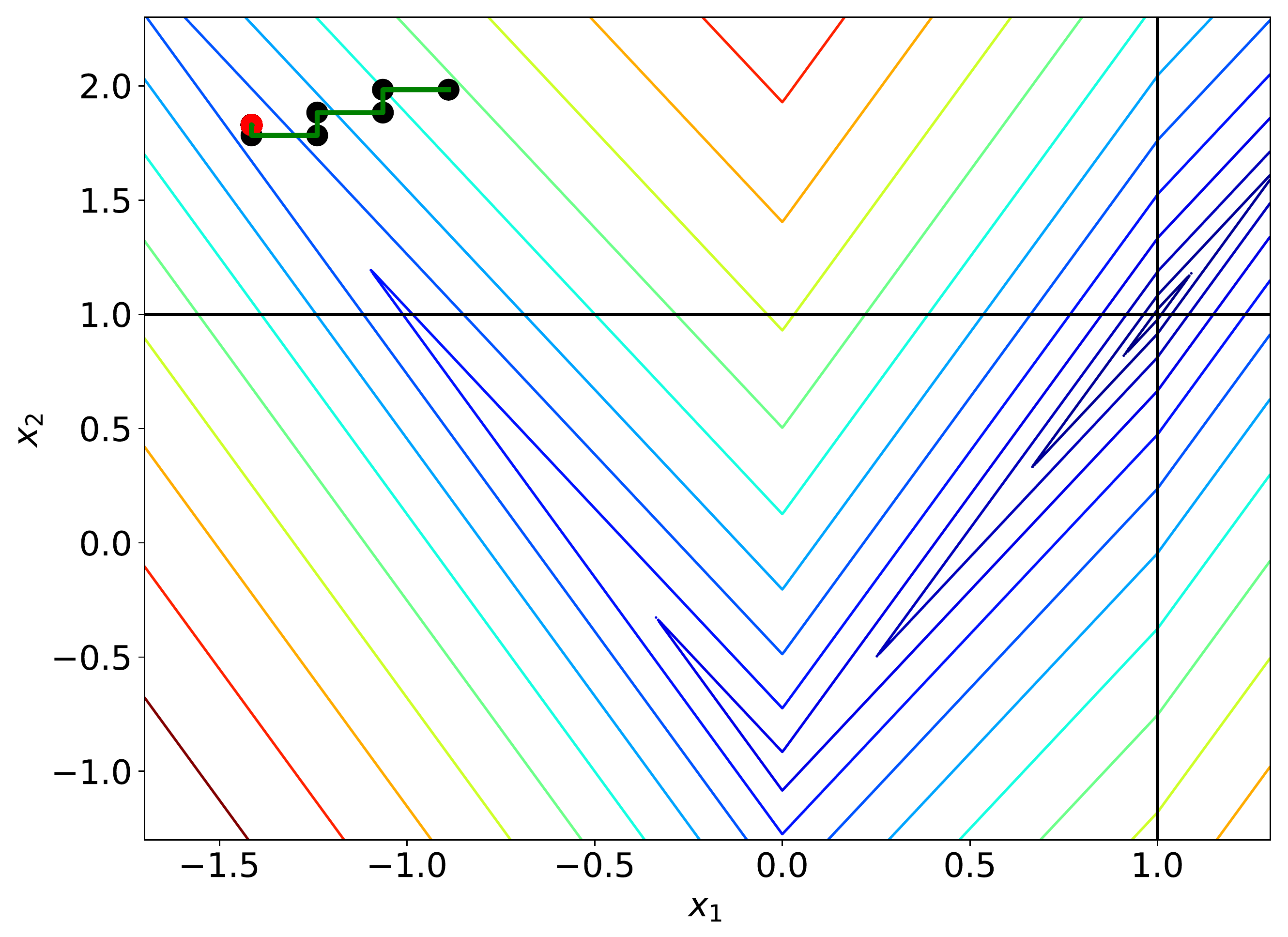}};
\node[draw, fill = white] at (2.2,2.1) {\scriptsize Standard Itoh--Abe} ;
\end{tikzpicture}
\begin{tikzpicture}
\draw (0,0) node[inner sep = 0] {\includegraphics[height=5cm]{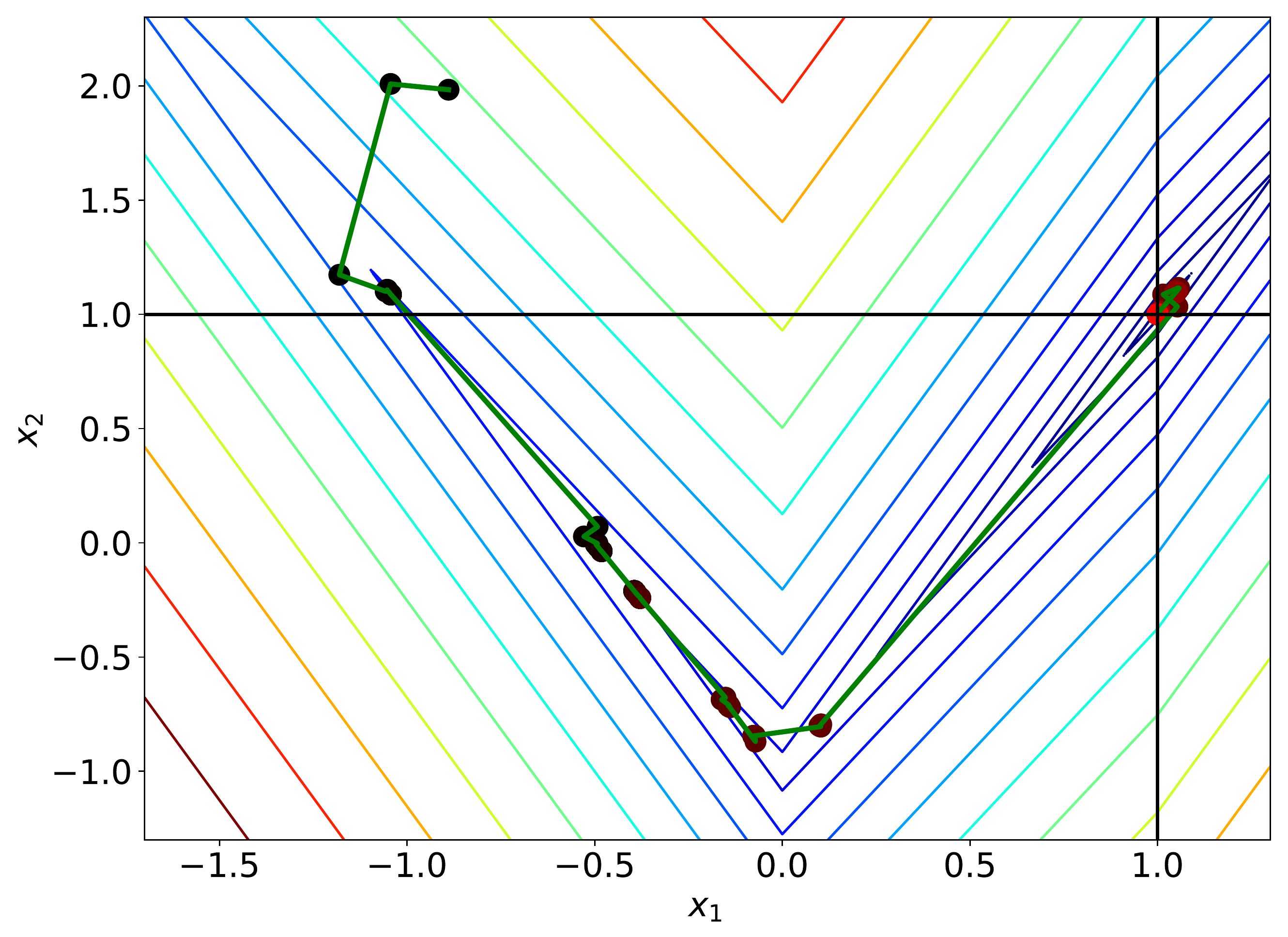}};
\node[draw, fill = white] at (2.25,2.1) {\scriptsize Rotated Itoh--Abe} ;
\end{tikzpicture}
\begin{tikzpicture}
\draw (0,0) node[inner sep = 0] {\includegraphics[height=5cm]{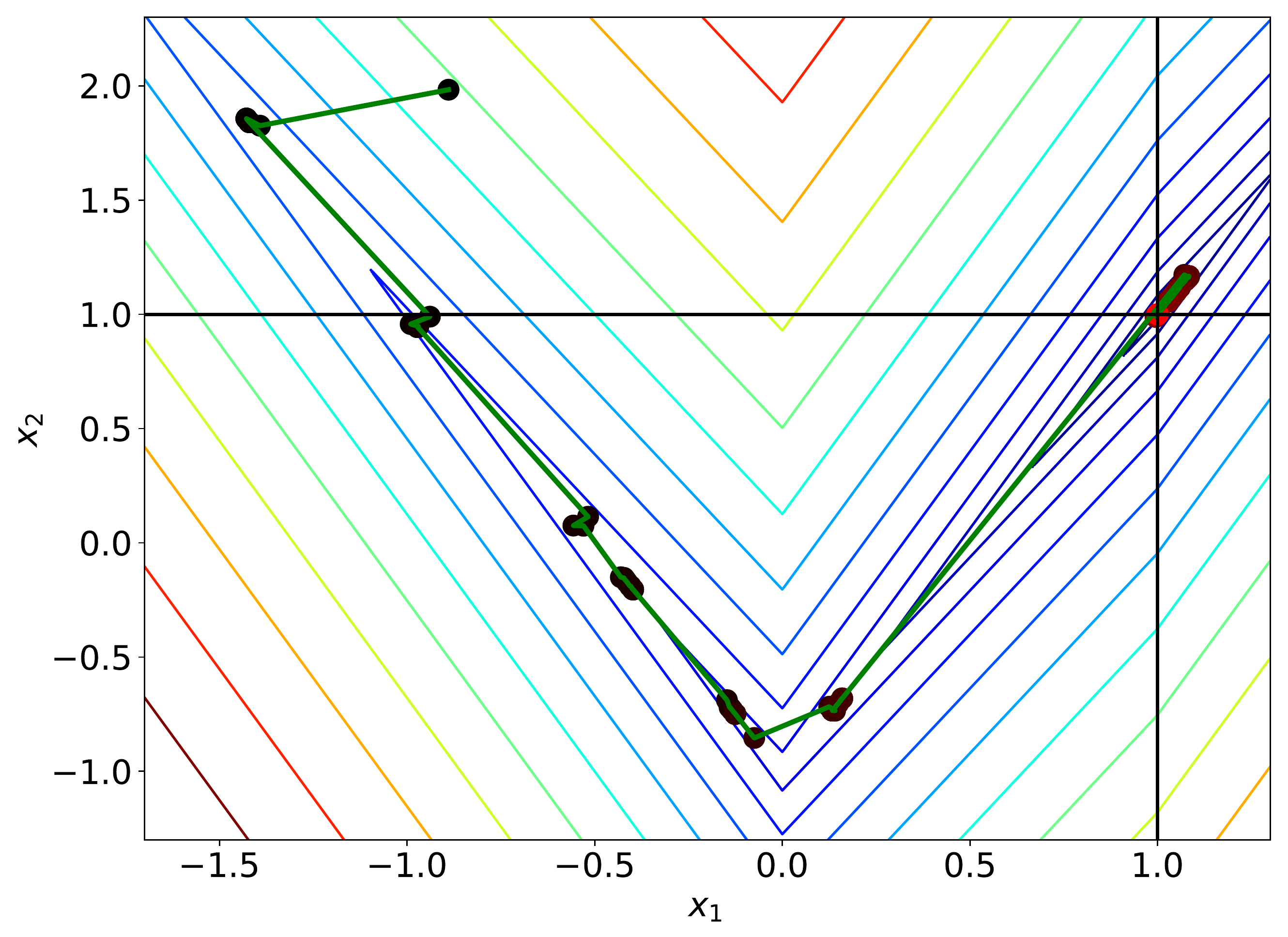}};
\node[draw, fill = white] at (1.85,2.1) {\scriptsize Random pursuit Itoh--Abe} ;
\end{tikzpicture}
\includegraphics[height=5cm]{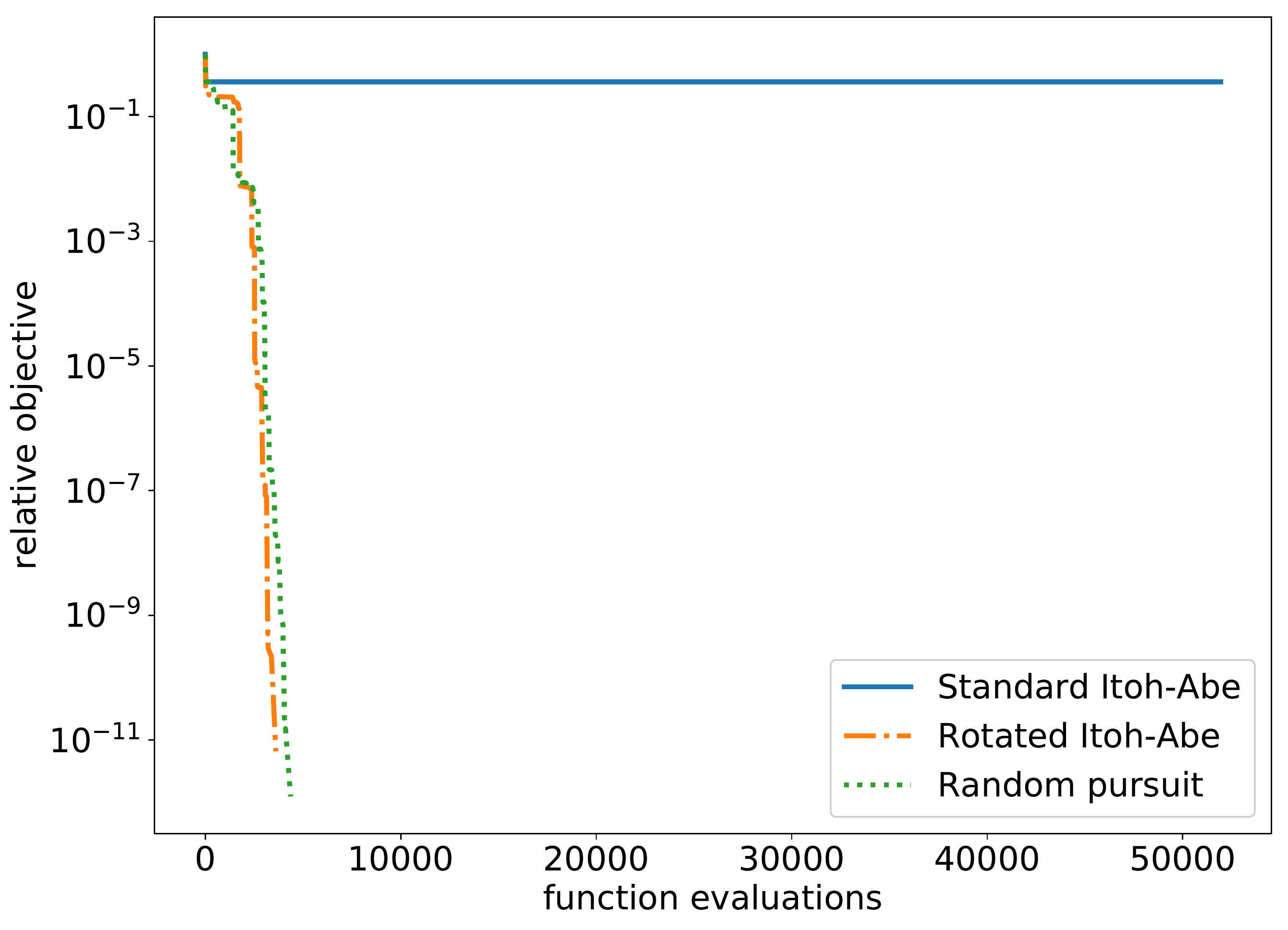}
\end{center}
\caption{Comparison of three variants of the Itoh--Abe method applied to Nesterov's nonsmooth Chebyshev--Rosenbrock function. Top left: Itoh--Abe method with standard frame. Top right: Rotated Itoh--Abe. Bottom left: Itoh--Abe with random pursuit. Bottom right: Convergence rates of the relative objective \(\frac{V(x^k) - V^*}{V(x^0) - V^*}\) for the three variants, displayed with a log--log plot.}
\label{fig:nesterov_dg}
\end{figuretmp}

We also compare the performance of the randomised Itoh--Abe (RIA) method to Py-BOBYQA and LT-MADS for Nesterov's nonsmooth Chebyshev--Rosenbrock function. We set the parameters of the Itoh--Abe method to \(\eps = 10^{-10}\), \(\tau_{\min} = 10^{-4}\), \(\tau_{\max} = 10^2\), \(\eta = 10^{-16}\), and \(M = 100\), the parameters of Py-BOBYQA to \(\mbox{rhobeg} = 2\), \(\mbox{rhoend} = 10^{-16}\) and \(\mbox{npt} = (n+1)(n+2)/2\), and the parameters of LT-MADS to \(\mbox{DIRECTION\_TYPE} = \mbox{LT } 2N\) and \(\mbox{MIN\_MESH\_SIZE } = 10^{-13}\). See \figref{fig:nesterov} and \ref{fig:nesterov2} for the numerical results for two different starting points. In the first case, the Itoh--Abe method successfully converges to the global minimiser, the LT-MADS method locates the nonminimising stationary point at \([0, -1]^T\), while the Py-BOBYQA iterates stagnate at a kink, reflecting the fact that the method is not designed for nonsmooth functions. In the second case, both the Itoh--Abe method and LT-MADS locate the minimiser, while the Py-BOBYQA iterates stagnate at a kink.

\begin{figuretmp}
\begin{center}
\begin{tikzpicture}
\draw (0,0) node[inner sep = 0] {\includegraphics[height=5cm]{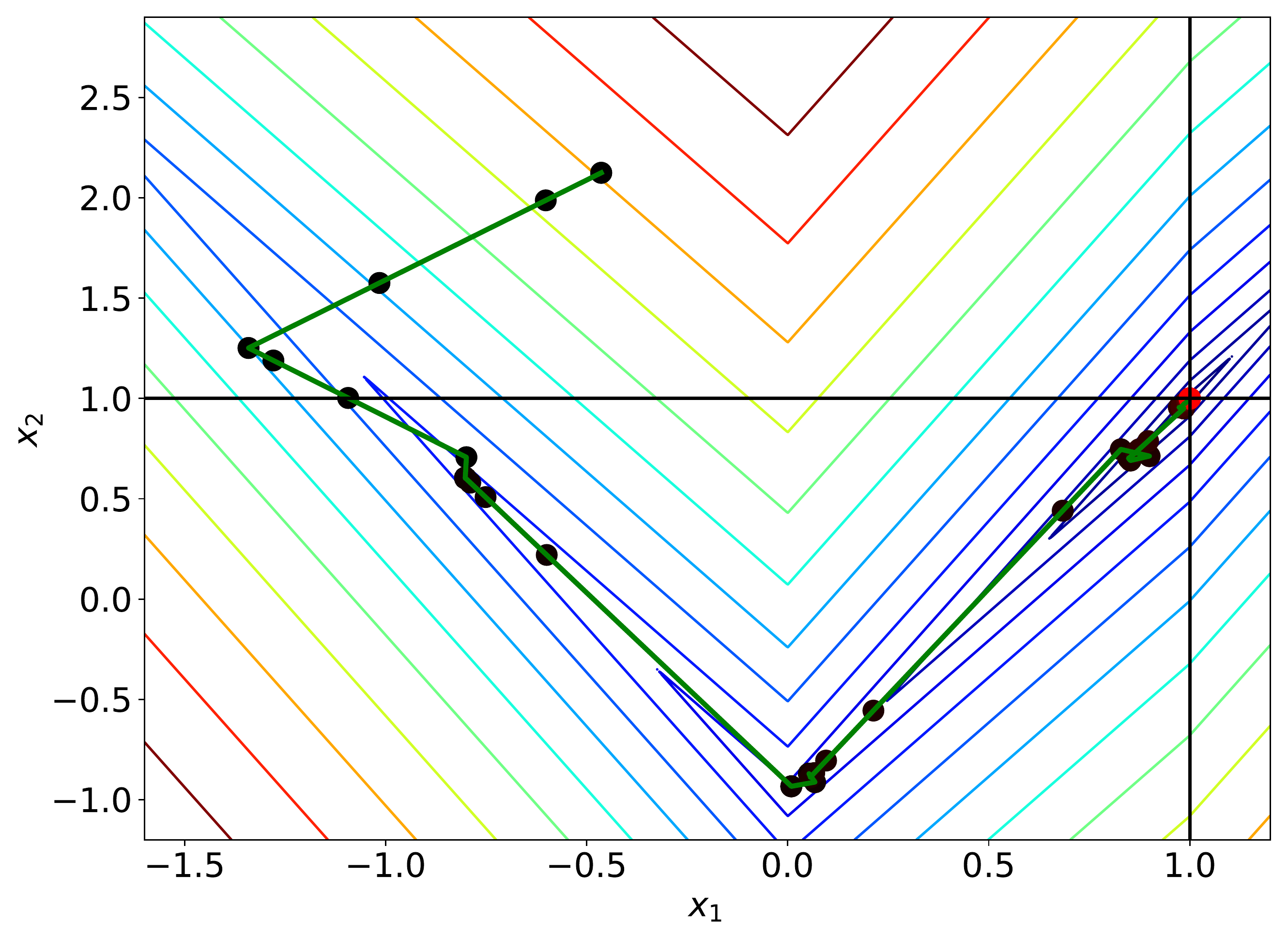}};
\node[draw, fill = white] at (-2.2,-1.7) {\scriptsize RIA} ;
\end{tikzpicture}
\begin{tikzpicture}
\draw (0,0) node[inner sep = 0] {\includegraphics[height=5cm]{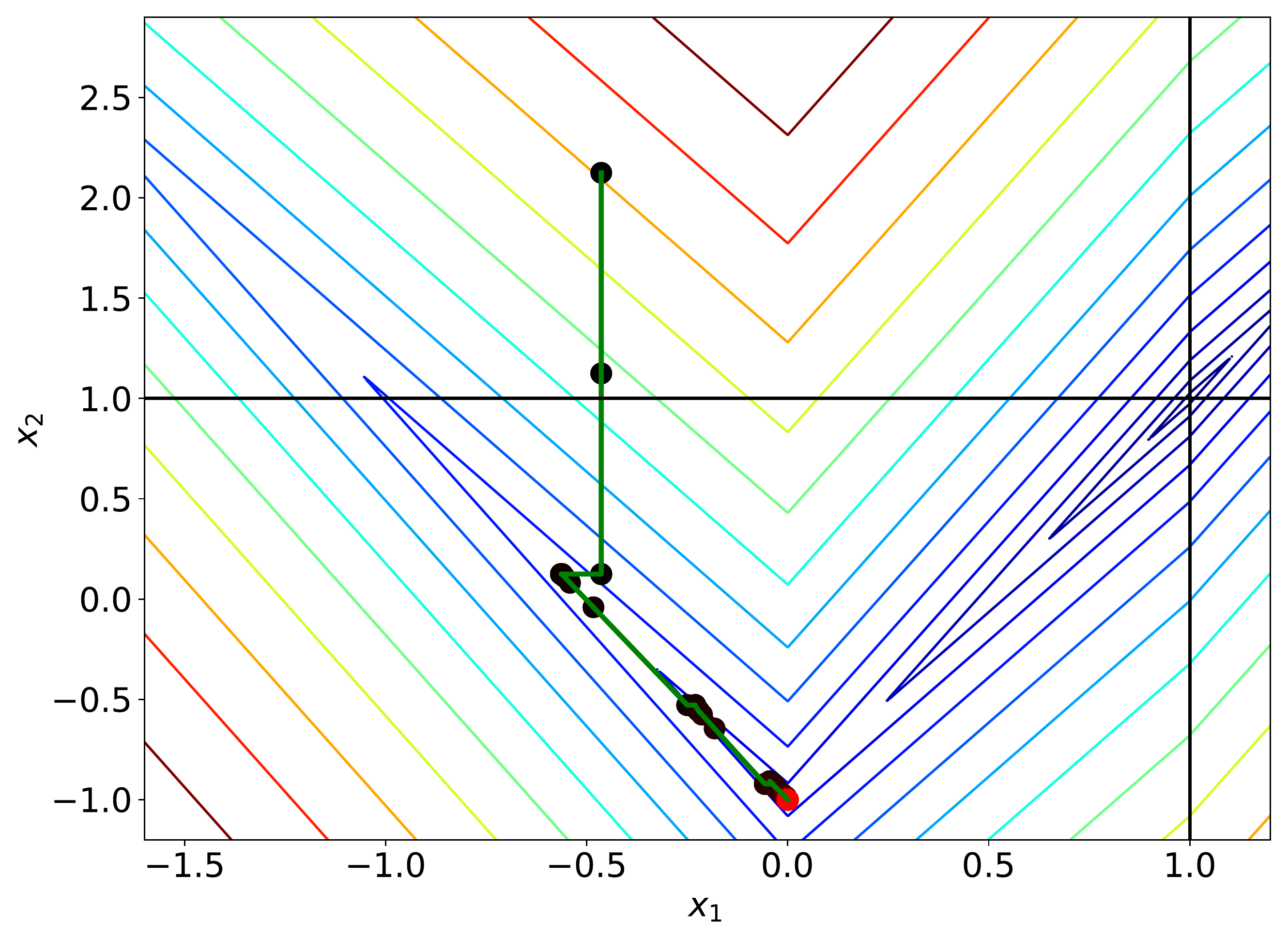}};
\node[draw, fill = white] at (-1.9,-1.7) {\scriptsize LT-MADS} ;
\end{tikzpicture}
\begin{tikzpicture}
\draw (0,0) node[inner sep = 0] {\includegraphics[height=5cm]{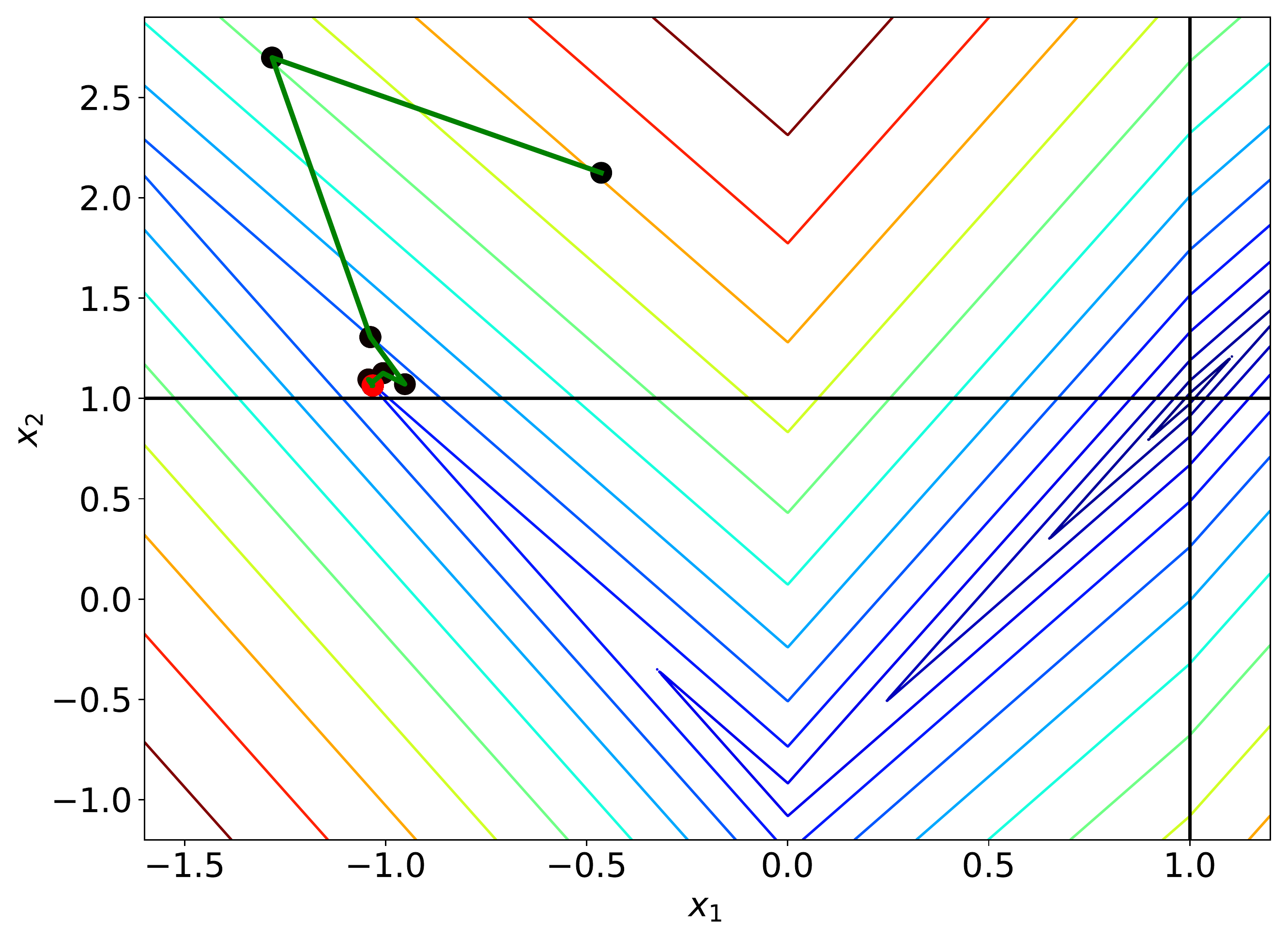}};
\node[draw, fill = white] at (-1.73,-1.7) {\scriptsize Py-BOBYQA} ;
\end{tikzpicture}
\includegraphics[height=5cm]{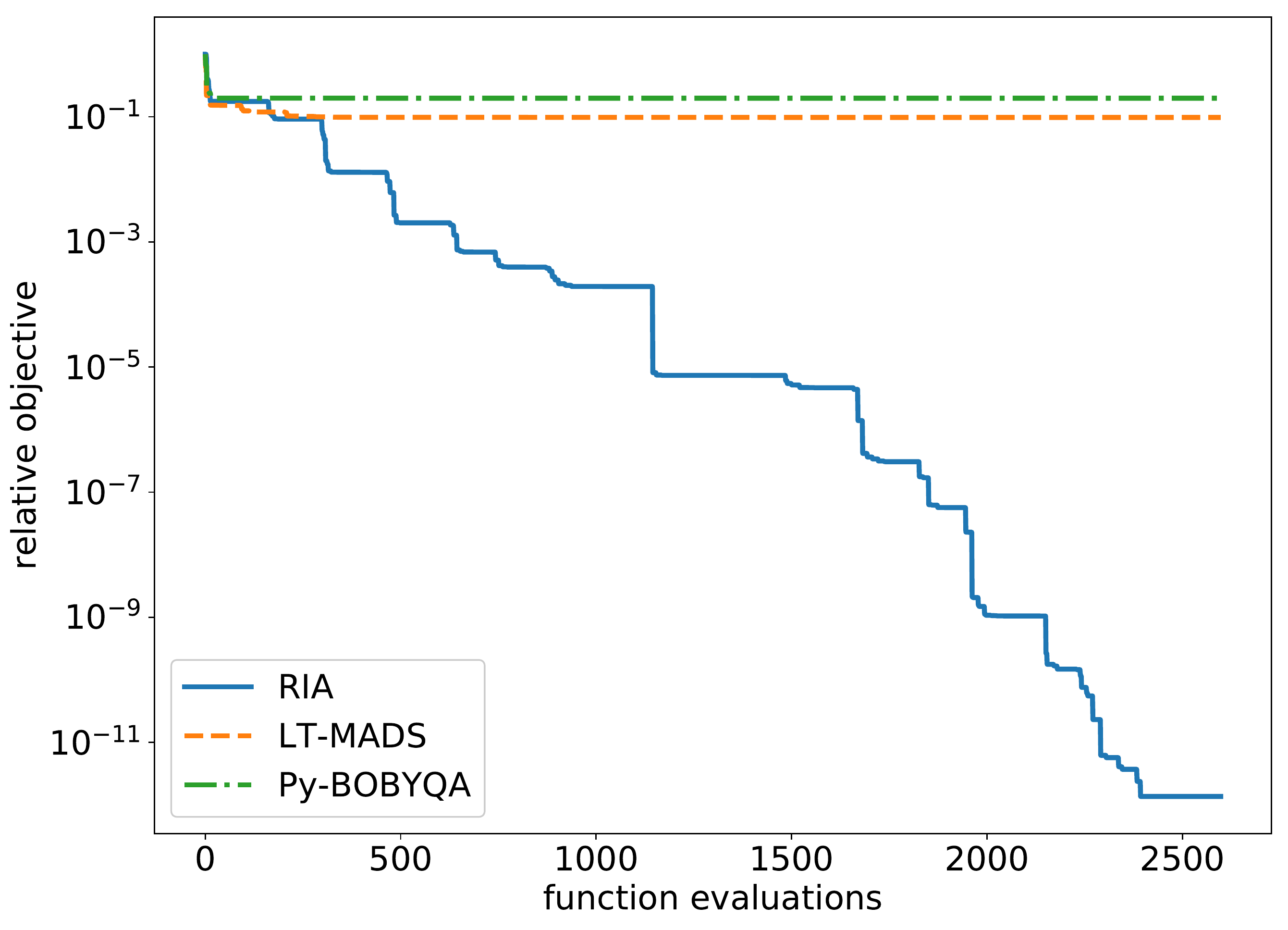}
\end{center}
\caption{Comparison of rotated Itoh--Abe method, LT-MADS and Py-BOBYQA applied to Nesterov's nonsmooth Chebyshev--Rosenbrock function. Top left: The iterates from the Itoh--Abe method locate the unique minimiser to an order of accuracy of about \(10^{-11}\). Top right: The iterates from the LT-MADS method locate the nonminimising stationary point. Bottom left: The iterates from the Py-BOBYQA method stagnate due to nonsmoothness. Bottom right: A plot of the relative objective \(\frac{V(x^k) - V^*}{V(x^0) - V^*}\) with respect to function evaluations, for each method.}
\label{fig:nesterov}
\end{figuretmp}

\begin{figuretmp}
\begin{center}
\begin{tikzpicture}
\draw (0,0) node[inner sep = 0] {\includegraphics[height=5cm]{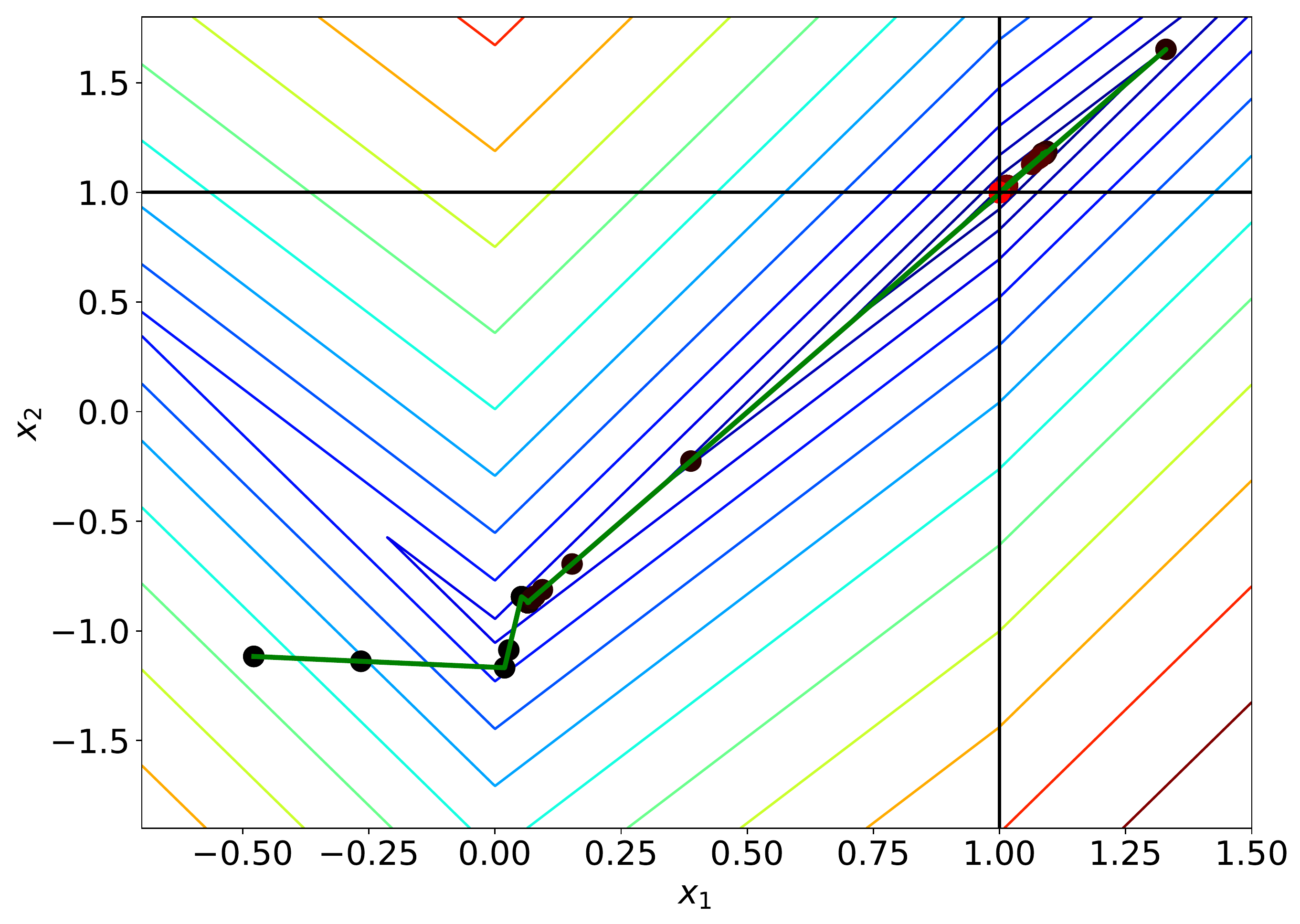}};
\node[draw, fill = white] at (-2.3,-1.7) {\scriptsize RIA} ;
\end{tikzpicture}
\begin{tikzpicture}
\draw (0,0) node[inner sep = 0] {\includegraphics[height=5cm]{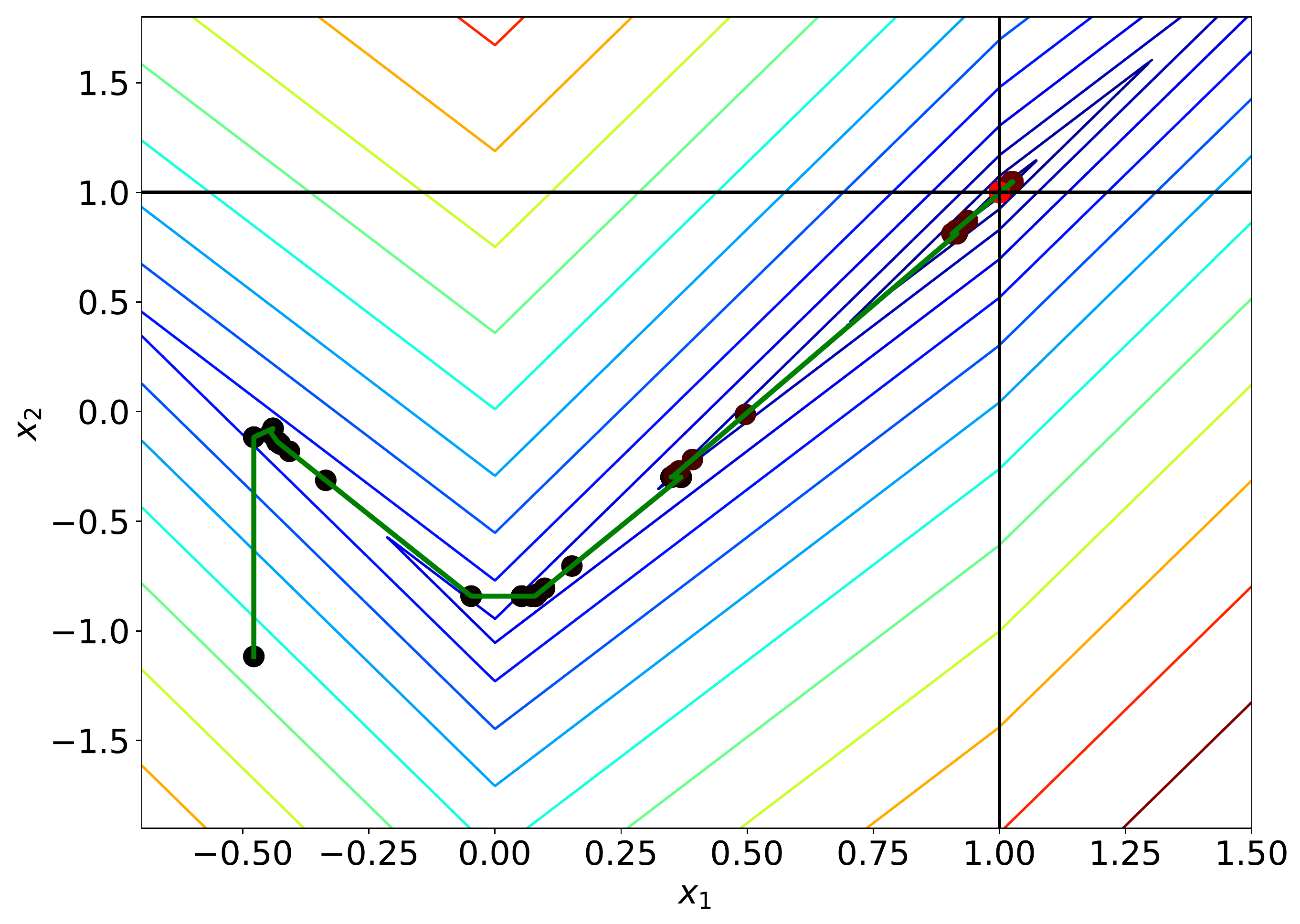}};
\node[draw, fill = white] at (-2.02,-1.7) {\scriptsize LT-MADS} ;
\end{tikzpicture}
\begin{tikzpicture}
\draw (0,0) node[inner sep = 0] {\includegraphics[height=5cm]{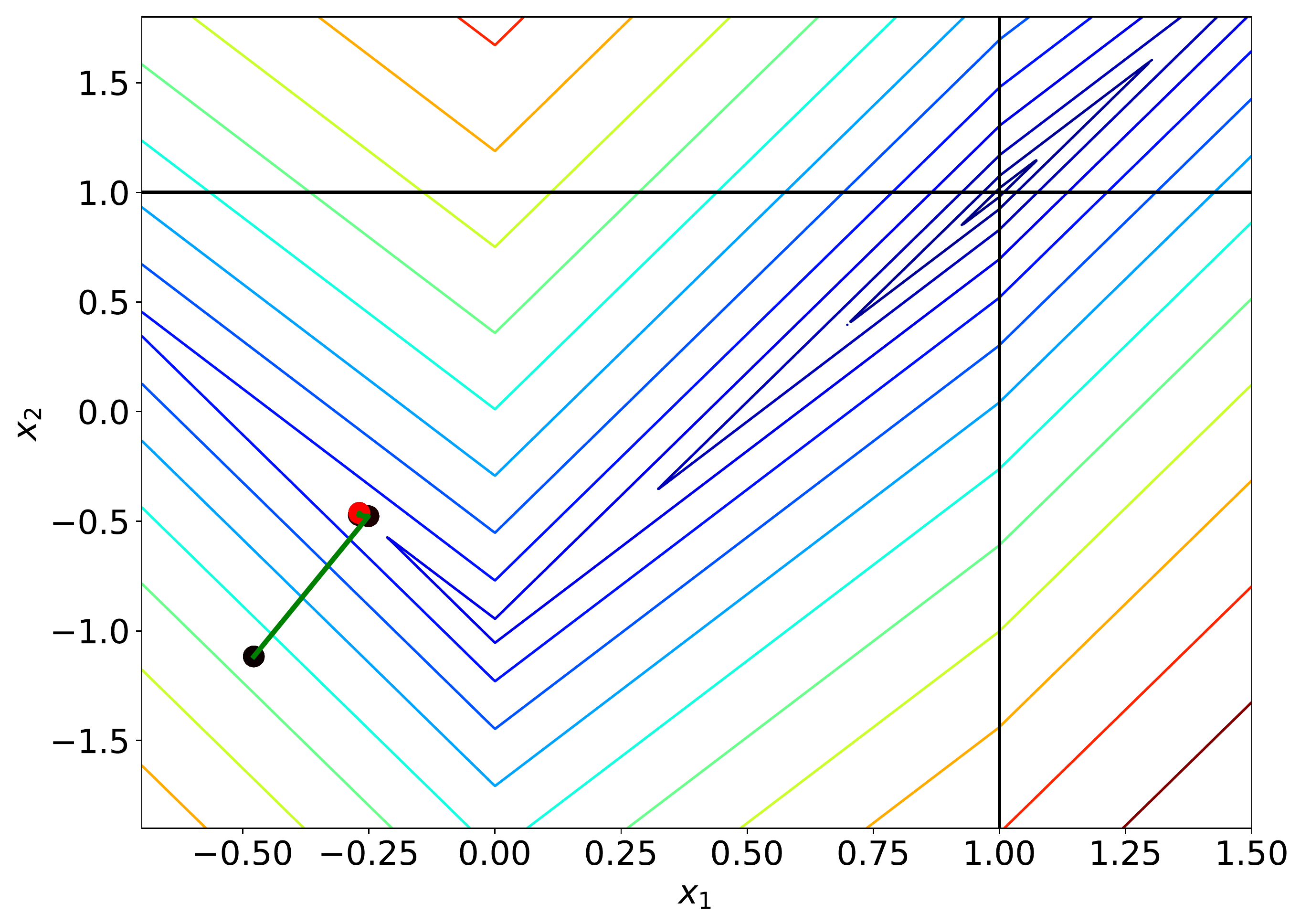}};
\node[draw, fill = white] at (-1.82,-1.7)  {\scriptsize Py-BOBYQA} ;
\end{tikzpicture}
\includegraphics[height=5cm]{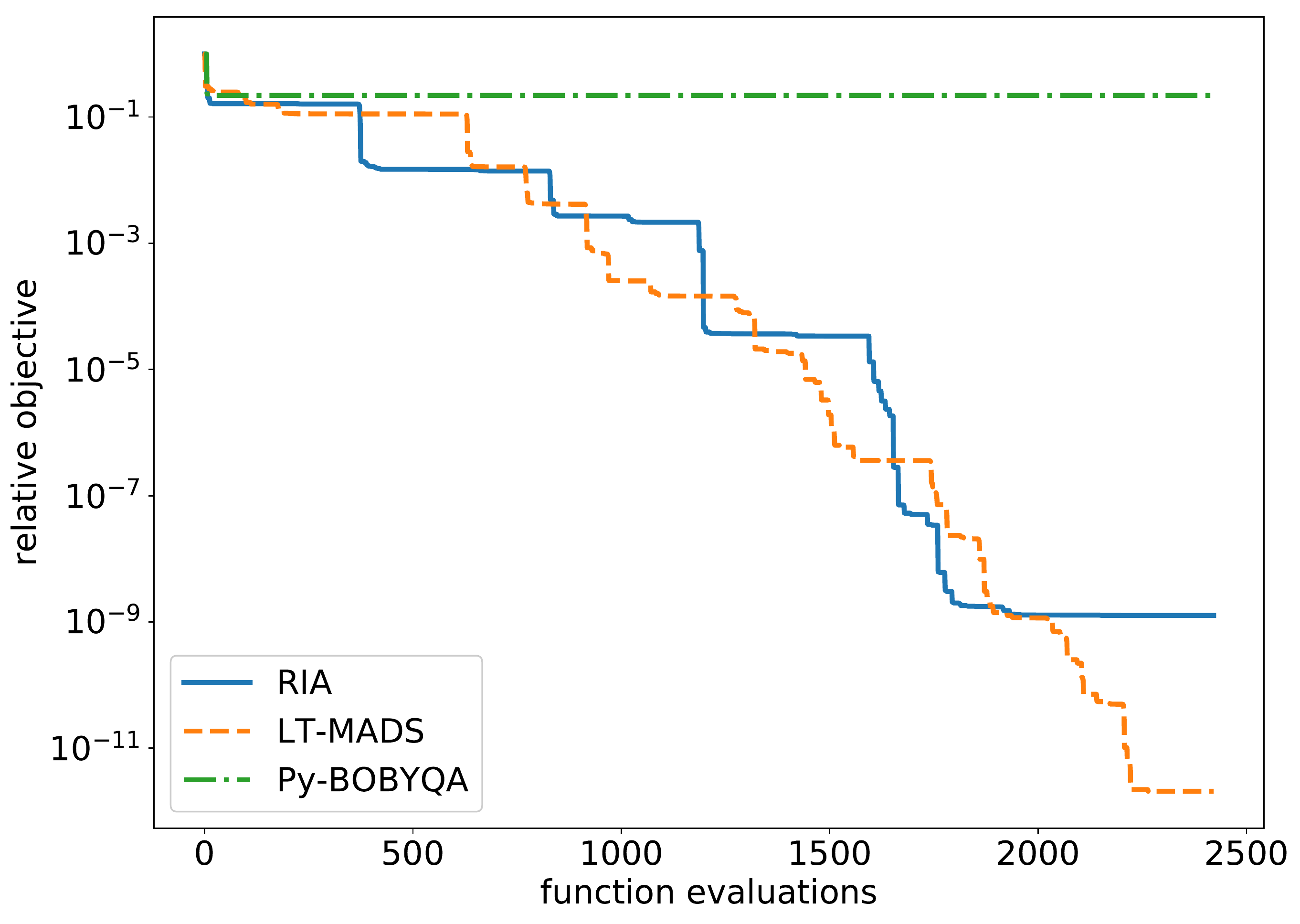}
\end{center}
\caption{Comparison of rotated Itoh--Abe method, LT-MADS and Py-BOBYQA applied to Nesterov's nonsmooth Chebyshev--Rosenbrock function with a different starting point. Top left: The iterates from the Itoh--Abe method locate the unique minimiser to an order of accuracy of about \(10^{-11}\). Top right: The iterates from the Py-BOBYQA method stagnate due to nonsmoothness. Bottom left: The iterates from the LT-MADS method locate the nonminimising stationary point. Bottom right: A plot of the relative objective \(\frac{V(x^k) - V^*}{V(x^0) - V^*}\) with respect to function evaluations, for each method.}
\label{fig:nesterov2}
\end{figuretmp}

\subsection{Bilevel parameter learning in image analysis}
\label{sec:bilevel}

\begin{figuretmp}
\begin{center}
\begin{subfigure}{0.45\textwidth}
\hspace{0.2cm} \includegraphics[height=5cm]{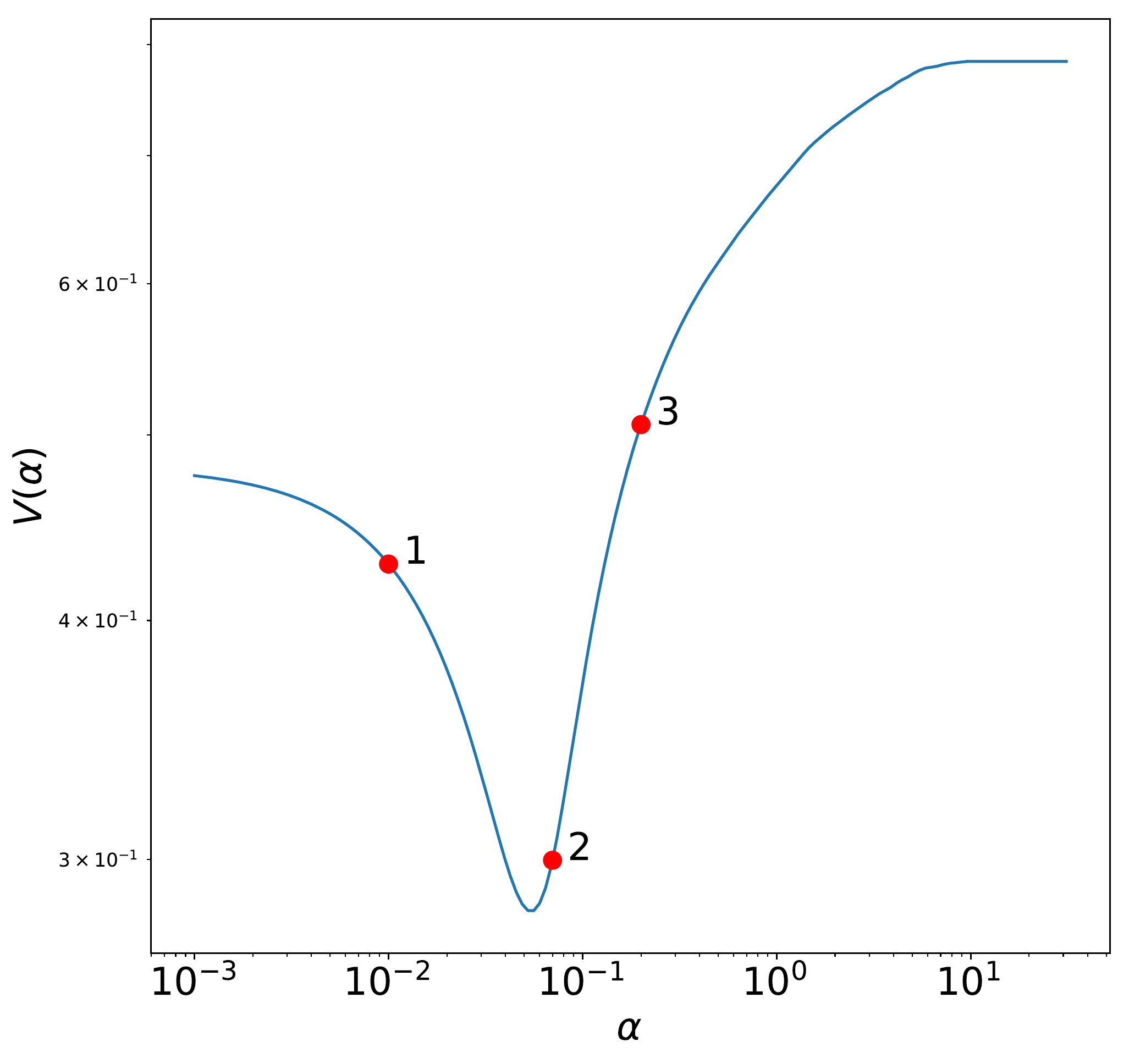}
\end{subfigure}
\begin{subfigure}{0.45\textwidth}
\begin{tikzpicture}
\draw (0,0) node[inner sep = 0] {\includegraphics[height=5cm]{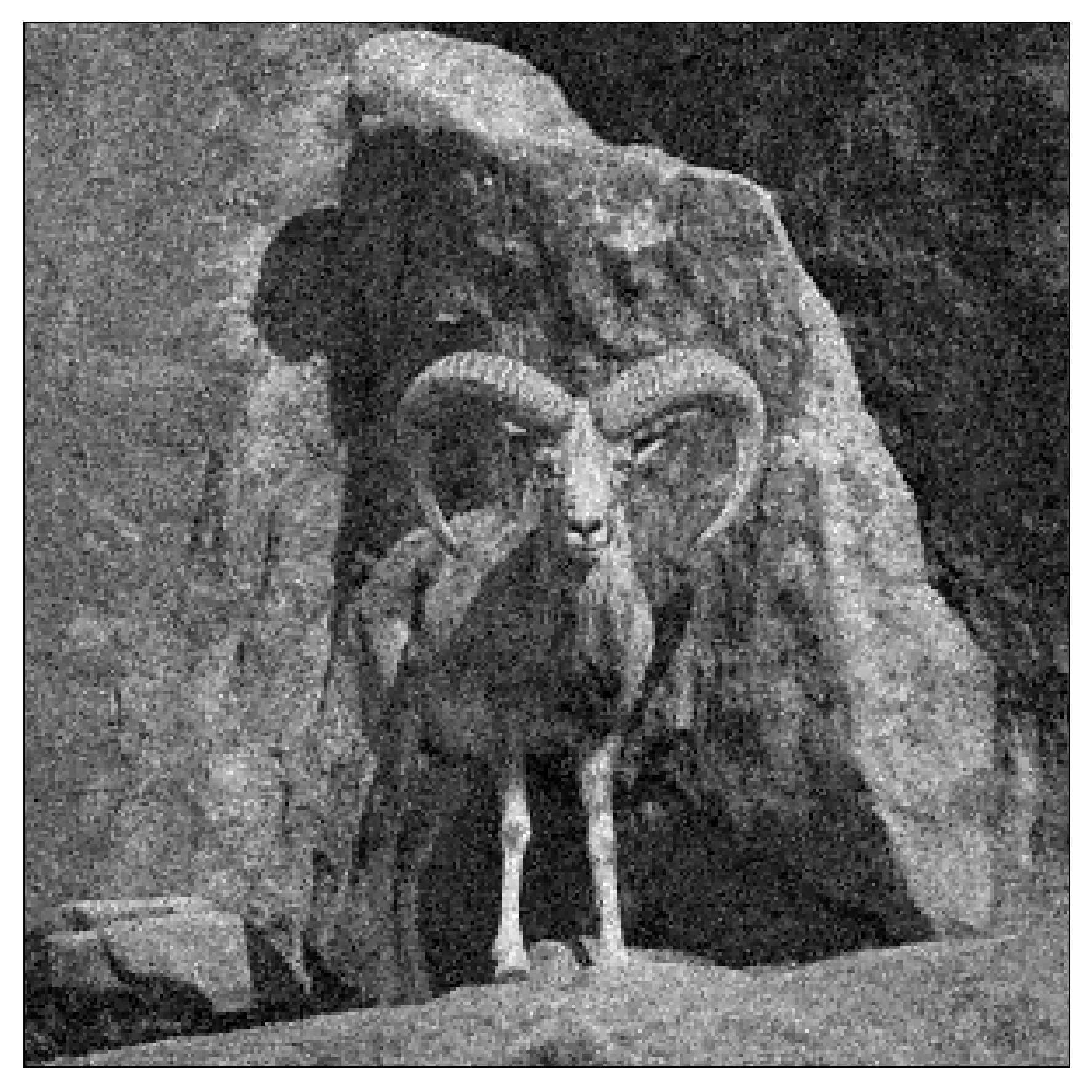}};
\node[draw, fill = white] at (-1.5,-1.6) {\(\alpha_1\)} ;
\end{tikzpicture}
\end{subfigure}
\begin{subfigure}{0.45\textwidth}
\hspace{0.7cm}
\begin{tikzpicture}
\draw (0,0) node[inner sep = 0] {\includegraphics[height=5cm]{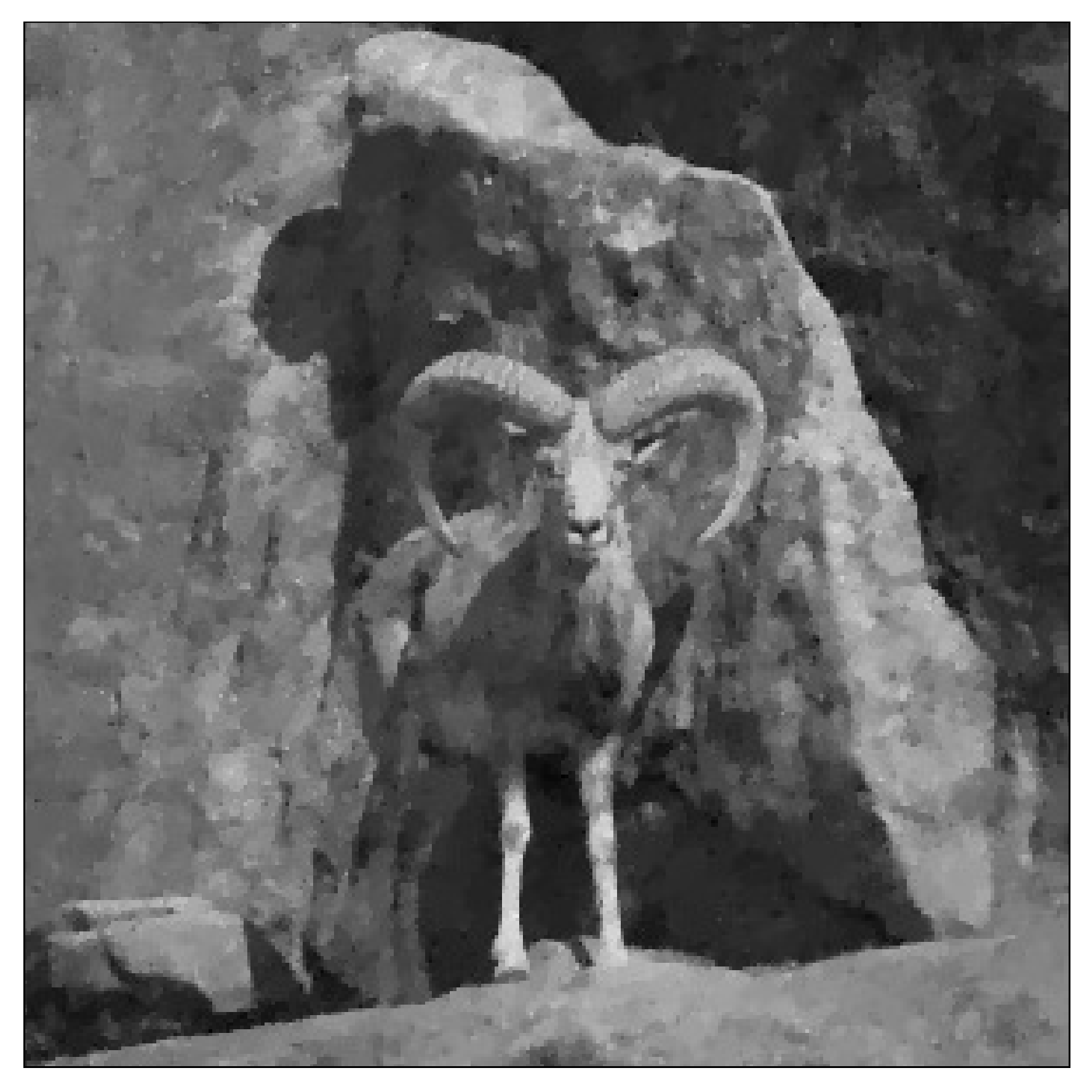}};
\node[draw, fill = white] at (-1.5,-1.6)  {\(\alpha_2\)} ;
\end{tikzpicture}
\end{subfigure}
\begin{subfigure}{0.45\textwidth}
\begin{tikzpicture}
\draw (0,0) node[inner sep = 0] {\includegraphics[height=5cm]{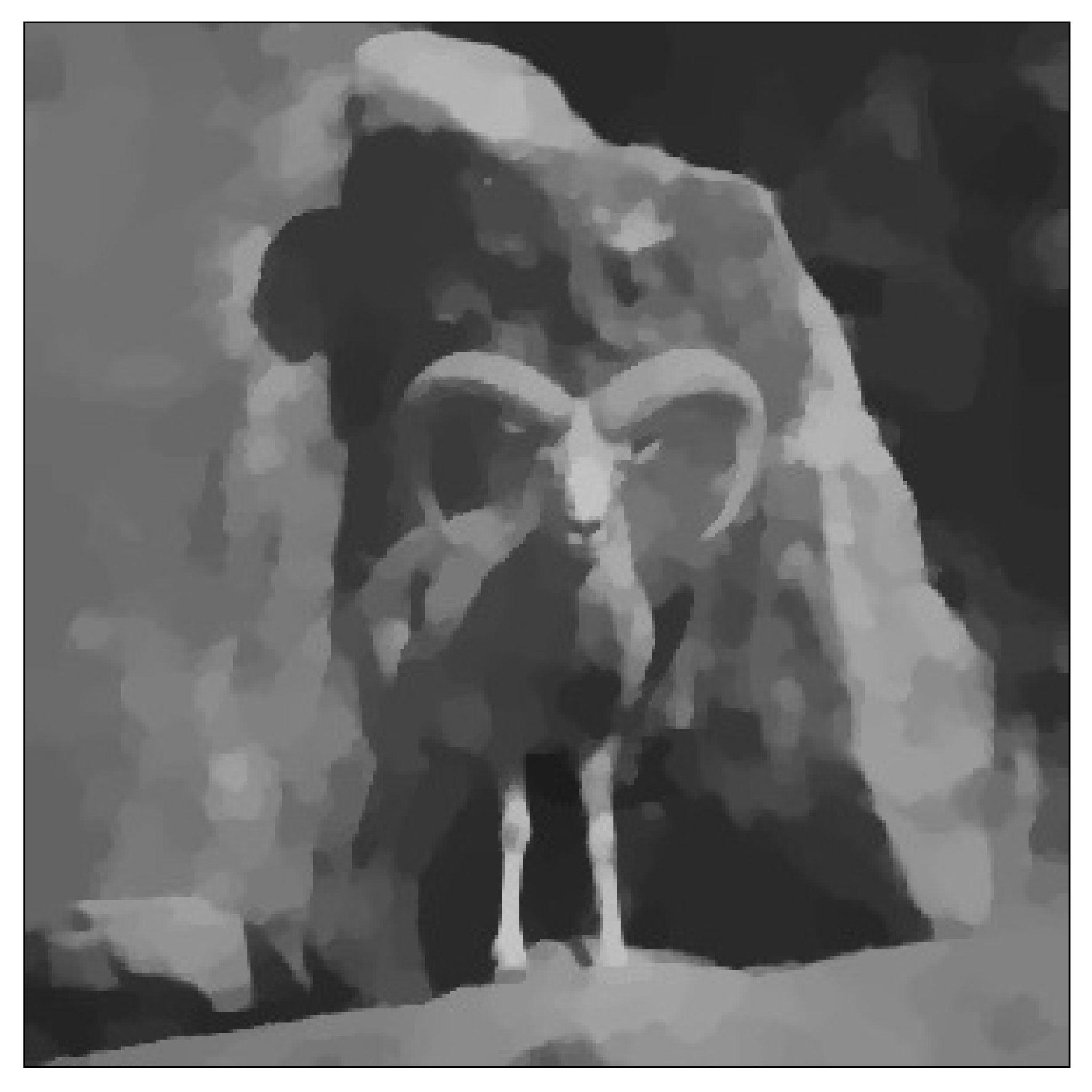}};
\node[draw, fill = white] at (-1.5,-1.6)  {\(\alpha_3\)} ;
\end{tikzpicture}
\end{subfigure}
\end{center}
\caption{TV denoising reconstructions for different regularisation parameters. Top left: Graph of \(V\) in \eqref{eq:V}. Top right: First parameter choice, \(\alpha_1 = 10^{-2}\). Bottom left: Second parameter choice, \(\alpha_2 = 7\times 10^{-2}\). Bottom right: Third parameter choice, \(\alpha_3 = 2\times 10^{-1}\).}
\label{fig:denoise}
\end{figuretmp}

In this subsection, we consider the Itoh--Abe method for solving bilevel optimisation problems for the learning of parameters of variational imaging problems. We restrict our focus to denoising problems, although the same method could be applied to any inverse problem. We first consider one-dimensional bilevel problems with wavelet and TV denoising, and two-dimensional problems with TGV denoising. In the TGV case, we compare the randomised Itoh--Abe method to the Py-BOBYQA and LT-MADS methods. Throughout this section, we set \(M = n\), where \(n= 1, 2\).

\subsubsection{Setup for variational regularisation problem}

Consider an image \(u^\dagger \in L^2(\Omega)\), for some domain \(\Omega \subset \RR^2\), and a noisy image
\[
f^\delta = u^\dagger + \mbox{ noise}.
\]
To recover a clean image from the noisy one, we consider a parametrised family of regularisers,
\[
\left\{R_\alpha: L^2(\Omega) \to [0, \infty] \; : \; \alpha \in [0, \infty)^n\right\},
\]
and solve the variational regularisation problem
\begin{equation}
\label{eq:regularisation_problem}
u_\alpha \in \argmin_{u} \frac{1}{2}\|u - f^\delta\|^2 + R_\alpha(u).
\end{equation}
The first term in \eqref{eq:regularisation_problem}, the \emph{data fidelity} term, ensures that the reconstruction approximates \(f^\delta\). The regulariser term serves to denoise the reconstruction, by promoting favourable features such as smooth regions and sharp edges. The parameters \(\alpha\) determine how heavily to regularise, and sometimes adjust other features of the regulariser. See \citep{ben18, ito14, sch08} for an overview of variational regularisation methods.

We list some common regularisers in image analysis. \emph{Total variation} (TV) \citep{bur13, rud92} is given by the function \(R_\alpha(u) := \alpha \TV(u)\), where \(\alpha \in [0, \infty)\), and
\begin{equation*}
\TV(u) := \sup\set{\int_{\Omega} u(x) \Div \phi(x) \dif x \; : \; \phi \in C^1_c(\Omega; \RR^d), \|\phi\|_\infty \leq 1}.
\end{equation*}
This is one of the most common regularisers for image denoising. See \figref{fig:denoise} for an example of denoising with TV regularisation. We also consider its second-order generalisation, \emph{total generalised variation} \citep{bre10, bre15}, \(R_\alpha(u) = \TGV_\alpha^2(u)\), where \(\alpha = [\alpha_1, \alpha_2]^T \in [0, \infty)^2\) and
\begin{align*}
&\qquad  \TGV^2_\alpha(u) \\
&:= \sup \set{ \int_\Omega \! u(x) \Div^2 \phi(x) \dif x \; : \; \phi \in C^2_c(\Omega; \Sym^2(\RR^d)),  \|\Div^l \phi\|_\infty \leq \alpha_{l+1}, \; l = 0, 1 }.
\end{align*}
For a linear operator \(W\) on \(L^2(\Omega)\), the basis pursuit regulariser
\[
R_\alpha(u) := \alpha \| W u\|_1
\]
promotes sparsity of the image \(u\) in the dictionary of \(W\).

As illustrated in \figref{fig:denoise}, the quality of the reconstruction is sensitive to \(\alpha\). If \(\alpha\) is too low, the reconstruction is too noisy, while if \(\alpha\) is too high, too much detail is removed. As it is generally not possible to ascertain the optimal choice of \(\alpha\) a priori, a significant amount of time and effort is spent on parameter tuning. It is therefore of interest to improve our understanding of optimal parameter choices. One approach is to learn suitable parameters from training data. This requires a desired reconstruction \(u^\dagger\), noisy data \(f^\delta\), and a scoring function \(\Phi: L^2(\Omega) \to \RR\) that measures the error between \(u^\dagger\) and the reconstruction \(u_\alpha\). The bilevel optimisation problem is given by
\begin{equation}
\label{eq:bilevel}
\alpha^* \in \argmin_{\alpha \in [0, \infty)^n} \Phi(u_\alpha), \qquad   \mbox{ s.t.} \;\; u_\alpha \mbox{ solves } \eqref{eq:regularisation_problem}.
\end{equation}
In our case, we have strong convexity in the data fidelity term, which implies that \(u_\alpha\) is unique for each \(\alpha \in [0, \infty)^n\). We can therefore define a mapping
\begin{equation}
\label{eq:V}
V(\alpha) := \Phi(u_\alpha).
\end{equation}

The bilevel problem \eqref{eq:bilevel} is difficult to tackle, both analytically as well as numerically. In most cases, the lower level problem \eqref{eq:regularisation_problem} does not have a closed form formulation. Instead, a reconstruction \(u_\alpha\) is approximated numerically with an algorithm. Therefore, one typically does not have access to gradient or subgradient information\footnote{While automatic differentiation \citep{gri08} can be useful for these purposes, it is in many cases not applicable to iterative algorithms that involve nonsmooth terms or for which the number of iterations cannot be predetermined.  See \citep{nes17} for further discussion of why derivative-free optimisation schemes are still needed.} for the mappings \(\alpha \mapsto u_\alpha\). Furthermore, the bilevel mapping \(\alpha \mapsto \Phi(u_\alpha)\) is often nonsmooth and nonconvex. Therefore, numerically solving \eqref{eq:bilevel} amounts to solving a nonsmooth, nonconvex function in a blackbox setting. We consider the application of the Itoh--Abe method for these problems.

For the numerical experiments in this paper, we reparametrise \(V(\alpha)\) as \(V(\exp(\alpha))\), where the exponential operator is applied elementwise on the parameters. There are two reasons for doing so. The first reason is that this paper is concerned with unconstrained optimisation, and this parametrisation allows is to optimise on \(\RR^n\) instead of \([0, \infty)^n\). The second reason is that \(\exp(\alpha)\) has been found to be a preferable scaling for purposes of numerical optimisation.

\subsubsection{Wavelet denoising}

We consider the wavelet denoising problem
\begin{equation*}
u_\alpha = \argmin_{u \in L^2(\Omega)} \frac{1}{2}\|u - f^\delta\|^2 + \alpha \|W u\|_1,
\end{equation*}
where \(W\) is a wavelet transform. In particular, \(W\) is an orthonormal basis, which implies that the regularisation problem has the unique solution
\[
u_\alpha = W^{-1}T_\alpha(Wf^\delta),
\]
where \(T_\alpha\) is the shrinkage operator defined by
\[
[T_\alpha(v)]_i := \sgn(v_i) \max\del{|v_i| - \alpha, 0}.
\]

We first optimise \(\alpha\) for the scoring function
\[
\Phi(u) := \frac{1}{2}\|u - u^\dagger\|^2.
\]
We set the parameters of the Itoh--Abe method to \(\eps = 10^{-4}\), \(\tau_{\min} = 10^{-1}\), \(\tau_{\max} = 10\), and \(\eta = 10^{-1}\). See \figref{fig:wavelet_L2} for the numerical results.

\begin{figuretmp}
\begin{subfigure}{0.33\textwidth}
\begin{center}
\includegraphics[height=4cm]{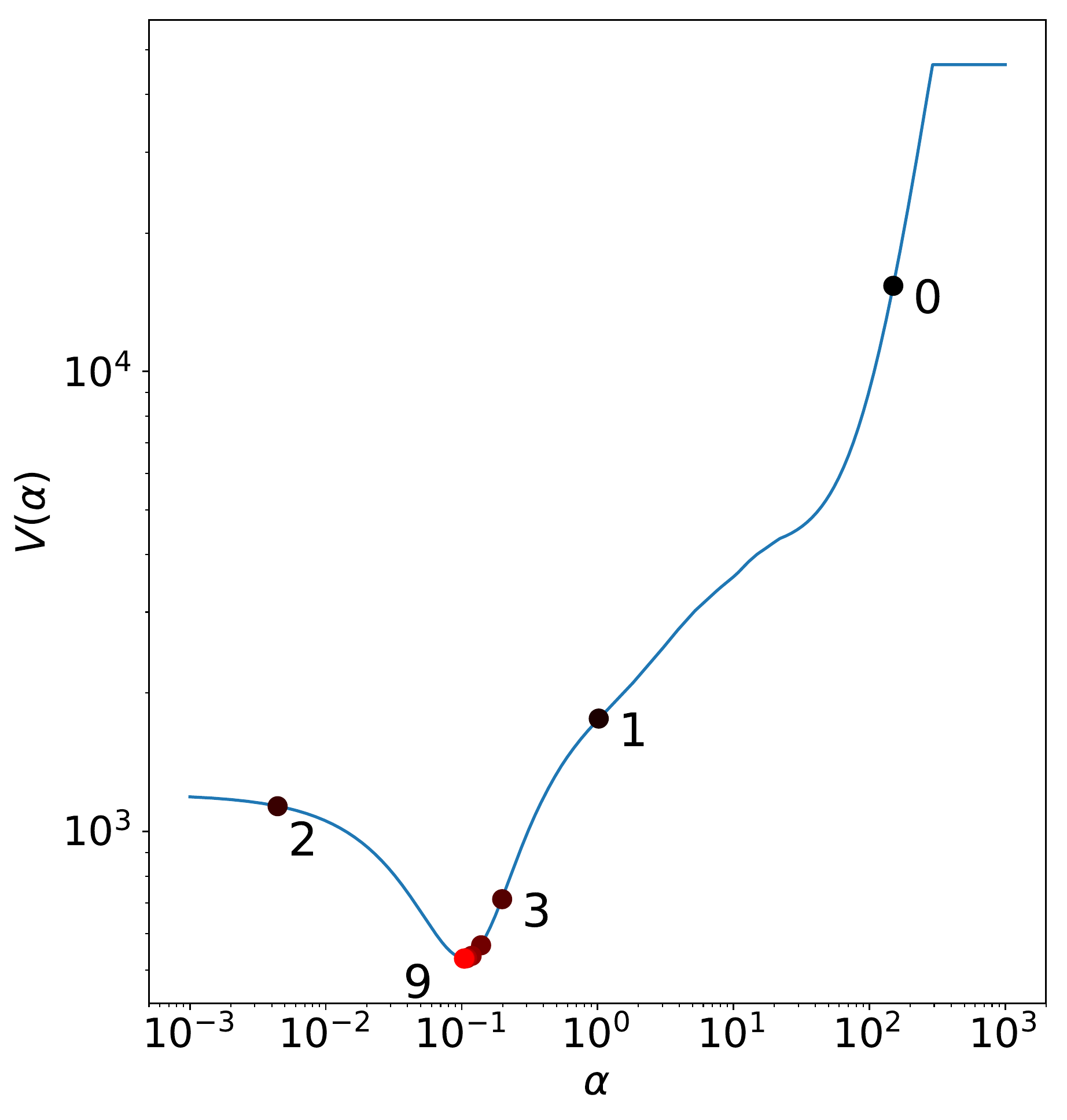}
\end{center}
\vspace{-0.4cm}
\caption{Plot with labels.}
\vspace{0.3cm}
\end{subfigure}
\begin{subfigure}{0.33\textwidth}
\begin{center}
\includegraphics[height=4cm]{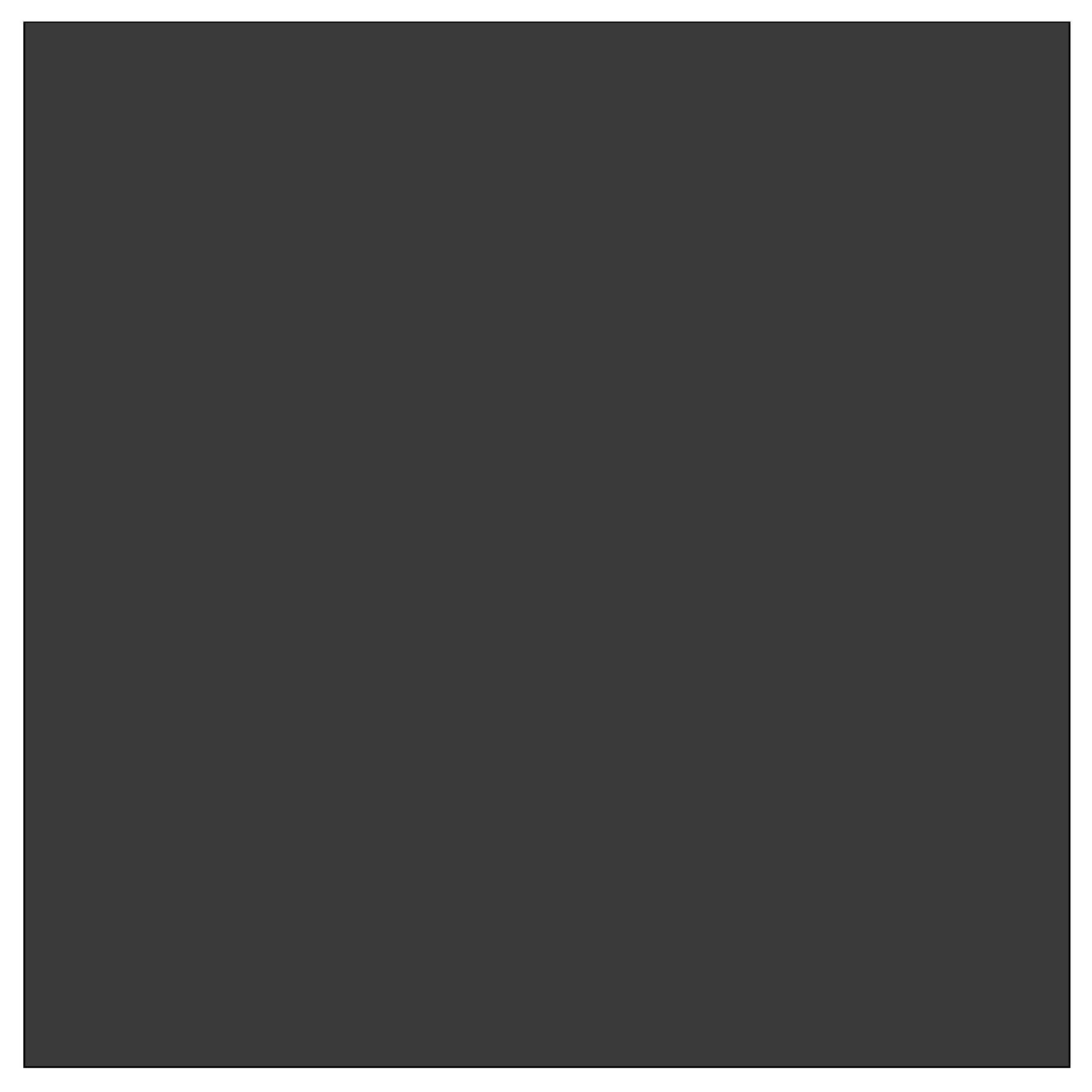}
\end{center}
\vspace{-0.4cm}
\caption{\(k = 0.\) \(\alpha = 1.50 \times 10^2\).}
\vspace{0.3cm}
\end{subfigure}
\begin{subfigure}{0.33\textwidth}
\begin{center}
\includegraphics[height=4cm]{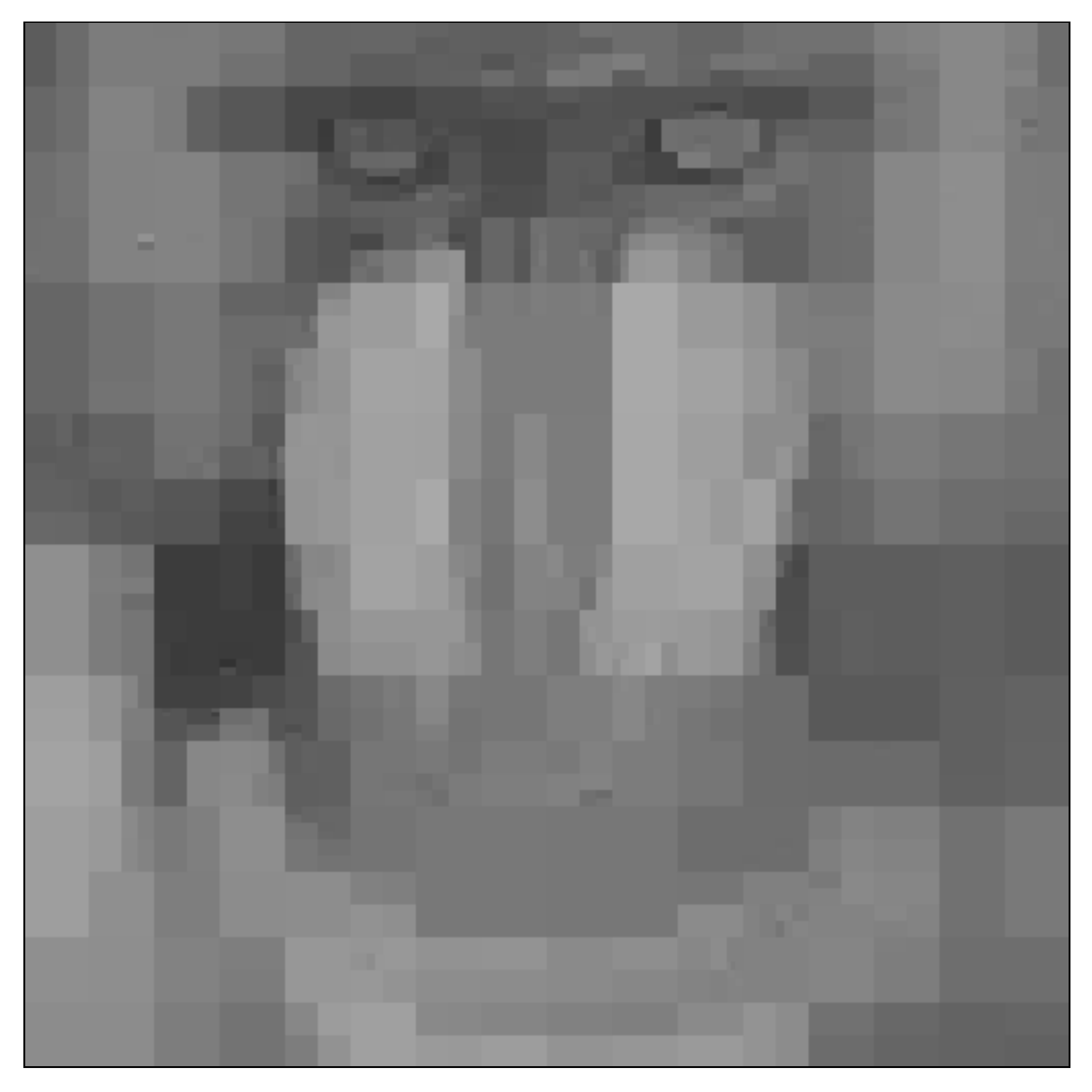}
\end{center}
\vspace{-0.4cm}
\caption{\(k= 1\). \(\alpha = 1.02\).}
\vspace{0.3cm}
\end{subfigure}
\begin{subfigure}{0.33\textwidth}
\begin{center}
\includegraphics[height=4cm]{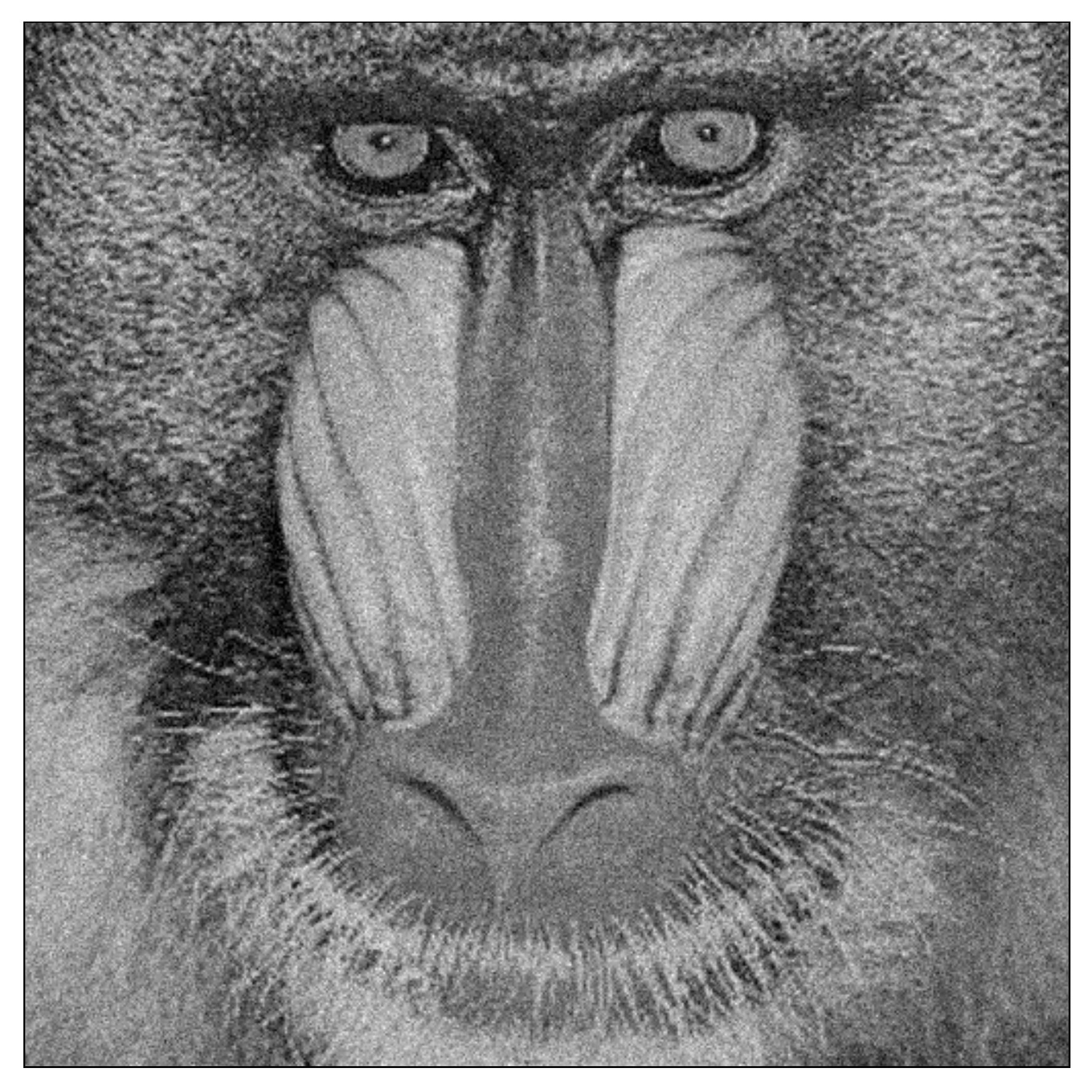}
\end{center}
\vspace{-0.4cm}
\caption{\(k=2\). \(\alpha = 4.42 \times 10^{-3}\).}
\vspace{0.3cm}
\end{subfigure}
\begin{subfigure}{0.33\textwidth}
\begin{center}
\includegraphics[height=4cm]{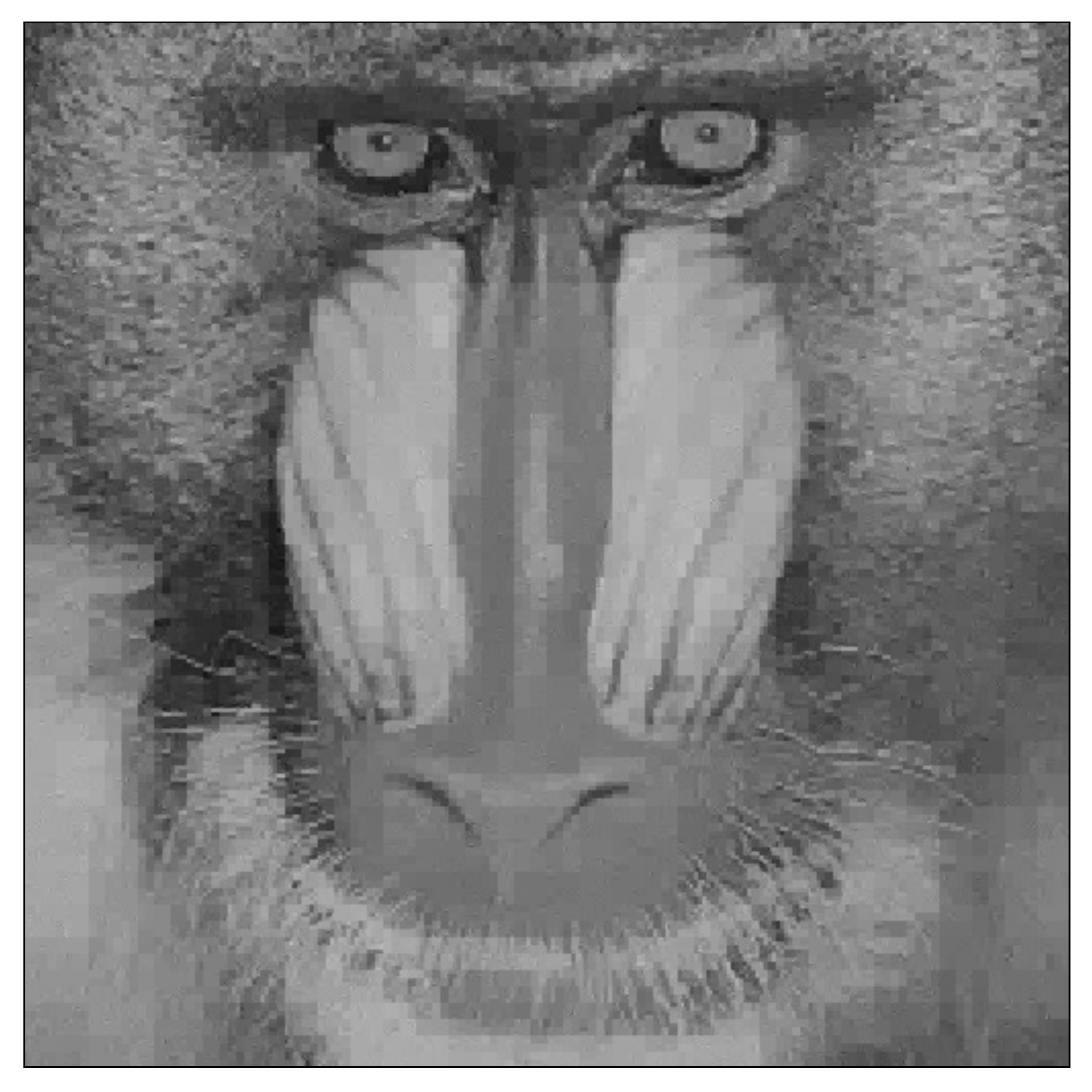}
\end{center}
\vspace{-0.4cm}
\caption{\(k=3\). \(\alpha = 1.99 \times 10^{-1}\).}
\vspace{0.3cm}
\end{subfigure}
\begin{subfigure}{0.33\textwidth}
\begin{center}
\includegraphics[height=4cm]{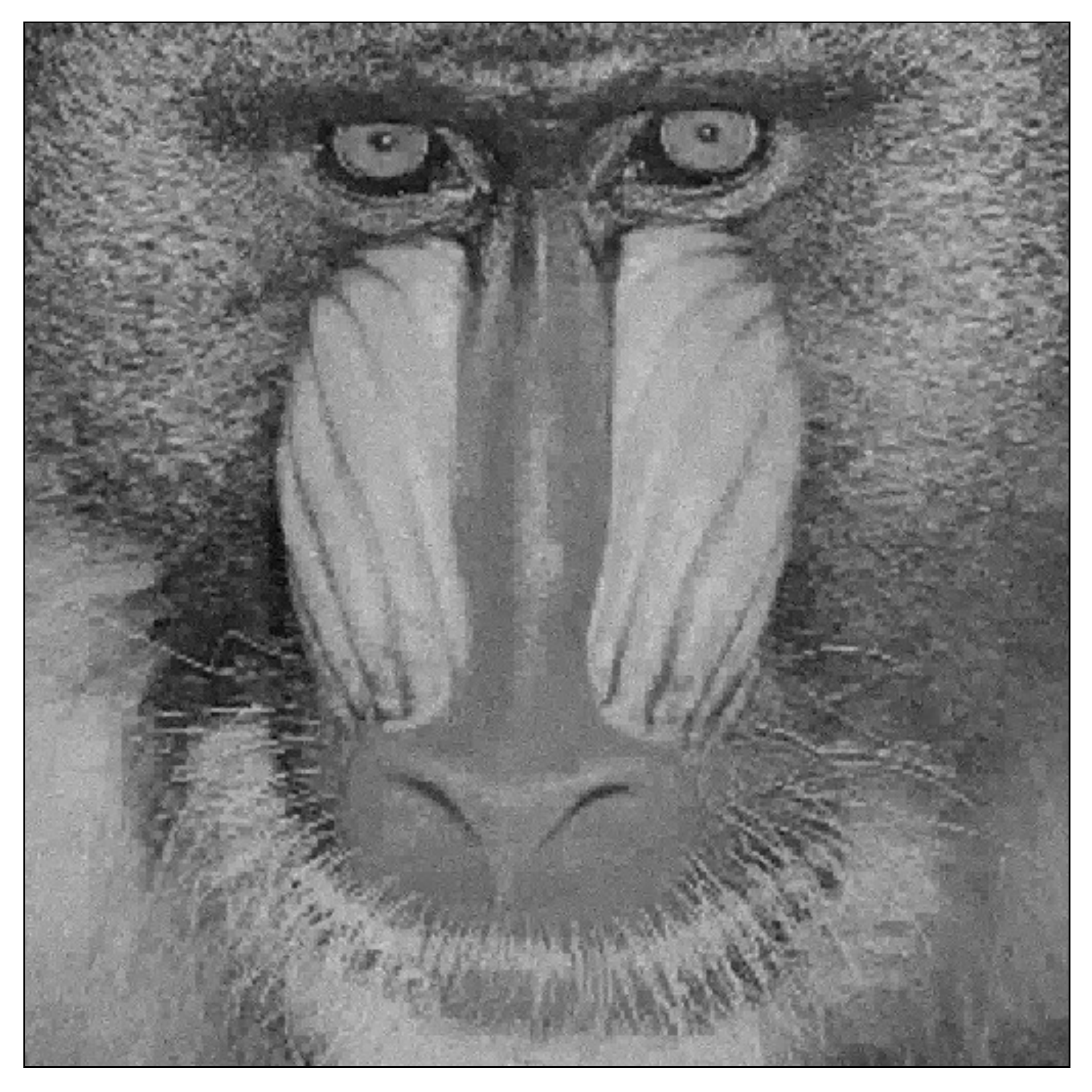}
\end{center}
\vspace{-0.4cm}
\caption{\(k=9\). \(\alpha = 1.04 \times 10^{-1}\).}
\vspace{0.3cm}
\end{subfigure}
\caption{Wavelet denoising with \(L^2\) scoring function and the Itoh--Abe method. Top left: Plot of iterates of the Itoh--Abe method. The rest: Image denoising results at different iterates \(k\).}
\label{fig:wavelet_L2}
\end{figuretmp}

We also optimise \(\alpha\) with respect to the scoring function \(\Phi(u) := 1 - \SSIM(u, u^\dagger)\), where \(\SSIM\) is the \emph{structural similarity} function \citep{wan04}
\[
\SSIM(u, v) := \frac{(2\mu_u \mu_v + c)(2\sigma_{uv} + C)}{(\mu_u^2 + \mu_v^2 + c)(\sigma_u^2 + \sigma_v^2 + C)}.
\]
Here \(\mu_u\) is the mean intensity of \(u\), \(\sigma_u\) is the unbiased estimate of the standard deviation of \(u\), and \(\sigma_{uv}\) is the correlation coefficient between \(u\) and \(v\):
\[
\mu_u := \frac{1}{m} \sum_{i = 1}^m u_i, \; \sigma_u := \del{ \frac{1}{m-1} \sum_{i = 1}^m (u_i-\mu_u)^2}^{\frac{1}{2}}, \; \sigma_{uv} := \frac{1}{m-1} \sum_{i = 1}^m (u_i-\mu_u)(v_i - \mu_v).
\]
We set the parameters of the Itoh--Abe method to \(\eps = 10^{-4}\), \(\tau_{\min} = 10^{-3}\), \(\tau_{\max} = 10^3\), and \(\eta = 10^{-2}\). See \figref{fig:wavelet_SSIM} for the numerical results.

\begin{figuretmp}
\begin{subfigure}{0.33\textwidth}
\begin{center}
\vspace{0.05cm}
\includegraphics[height=4cm]{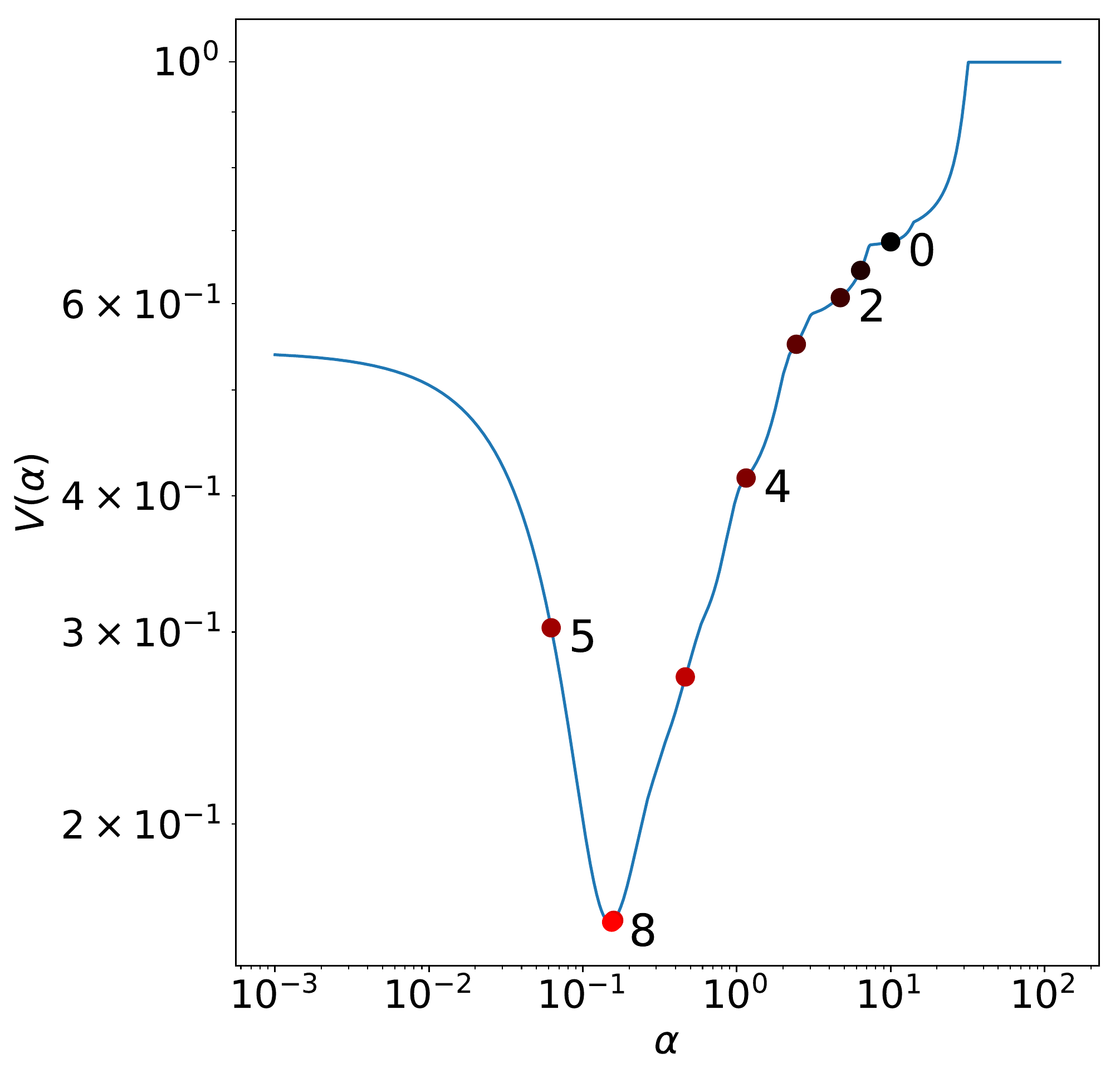}
\end{center}
\vspace{-0.35cm}
\caption{Plot with labels.}
\vspace{0.3cm}
\end{subfigure}
\begin{subfigure}{0.33\textwidth}
\begin{center}
\includegraphics[height=4cm]{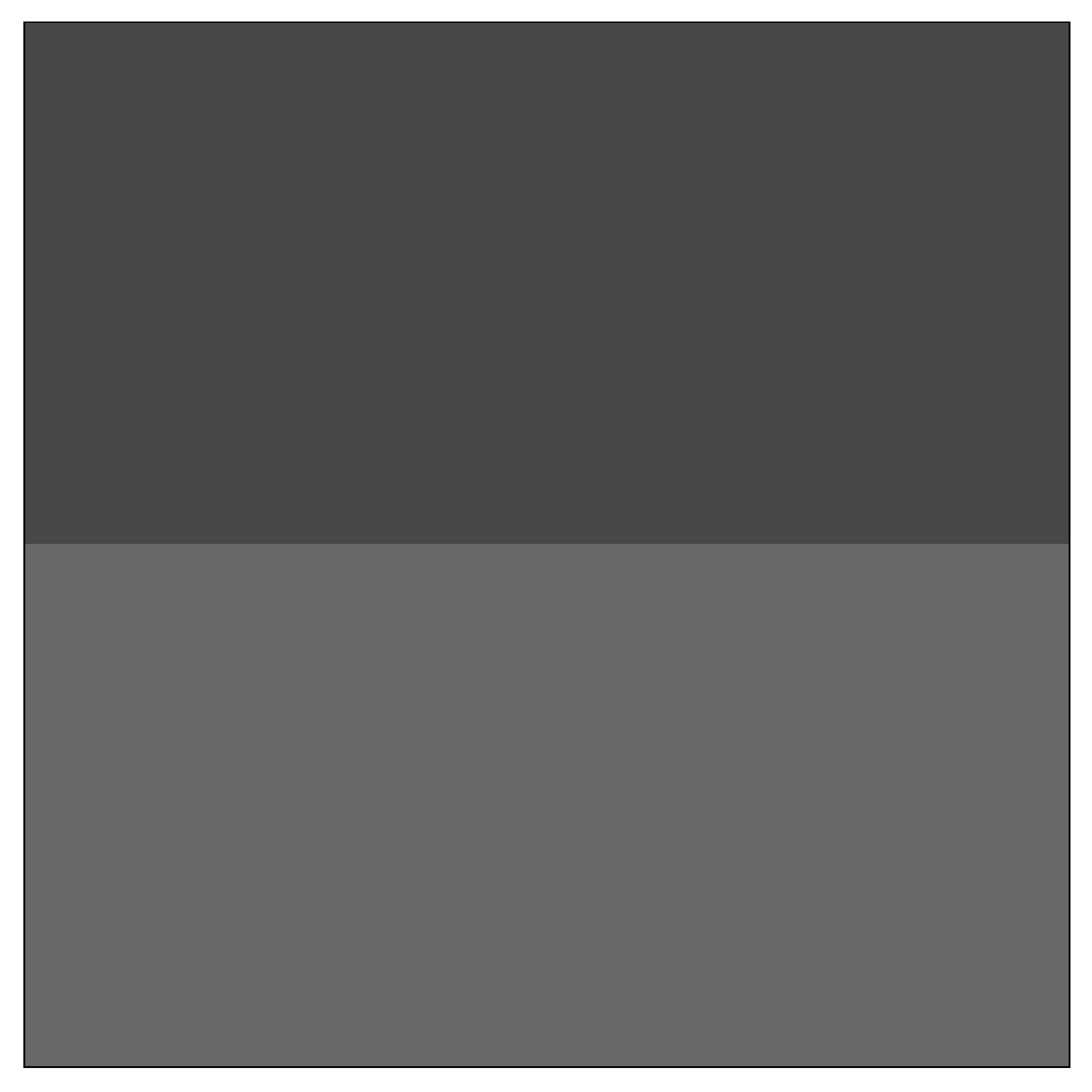}
\end{center}
\vspace{-0.35cm}
\caption{\(k = 0.\) \(\alpha = 10.0\).}
\vspace{0.3cm}
\end{subfigure}
\begin{subfigure}{0.33\textwidth}
\begin{center}
\includegraphics[height=4cm]{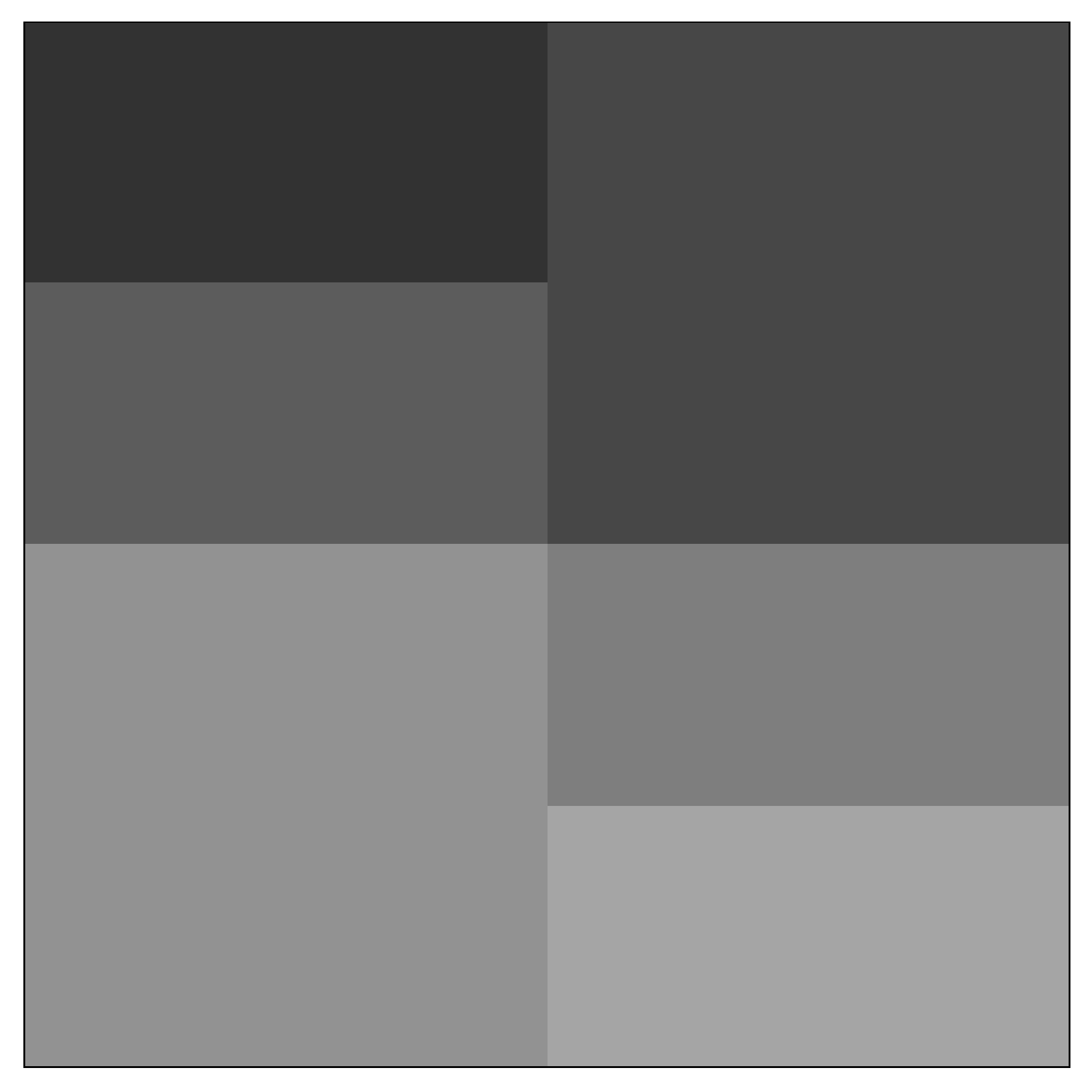}
\end{center}
\vspace{-0.35cm}
\caption{\(k= 2\). \(\alpha = 4.71\).}
\vspace{0.3cm}
\end{subfigure}
\begin{subfigure}{0.33\textwidth}
\begin{center}
\includegraphics[height=4cm]{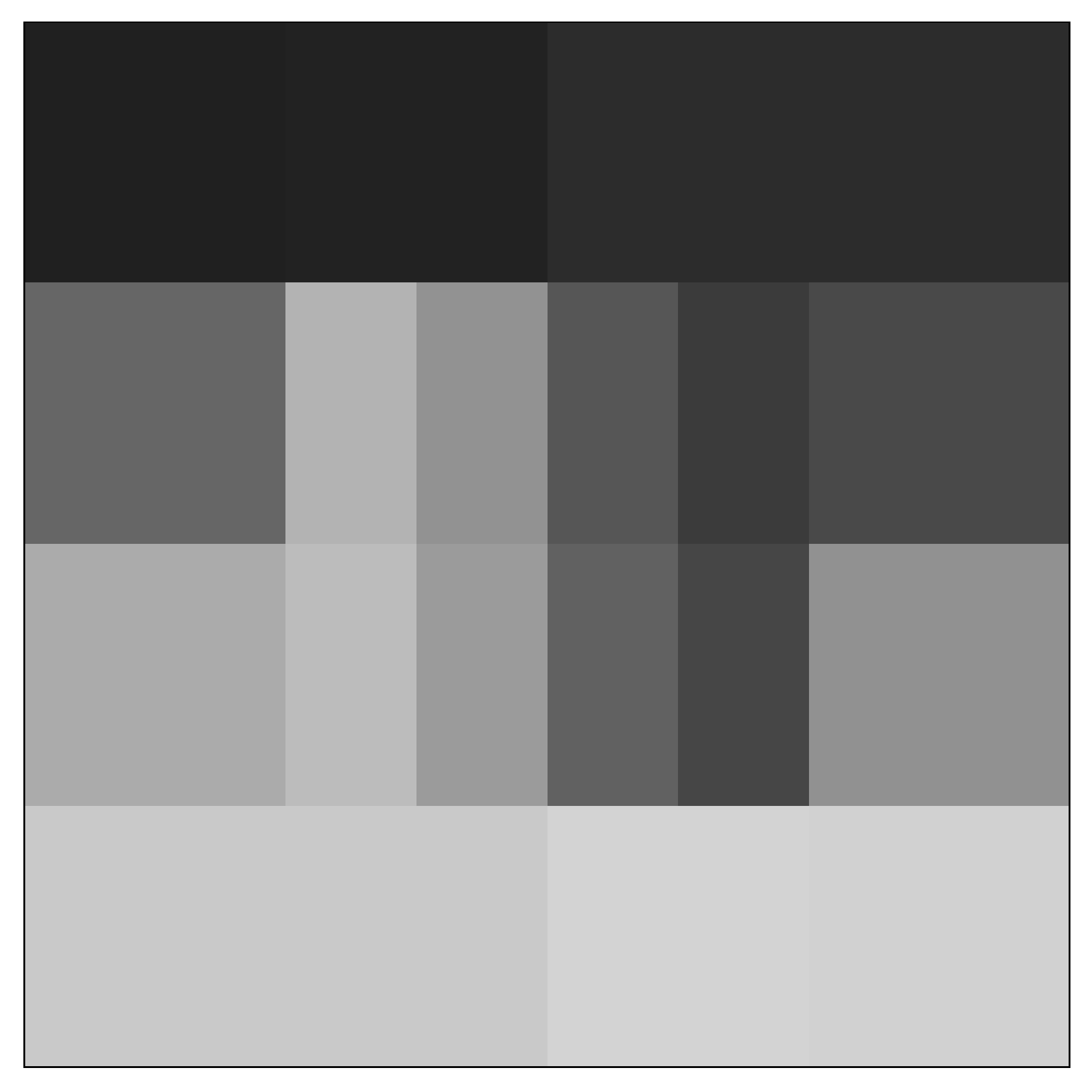}
\end{center}
\vspace{-0.35cm}
\caption{\(k=4\). \(\alpha = 1.15\).}
\end{subfigure}
\begin{subfigure}{0.33\textwidth}
\begin{center}
\includegraphics[height=4cm]{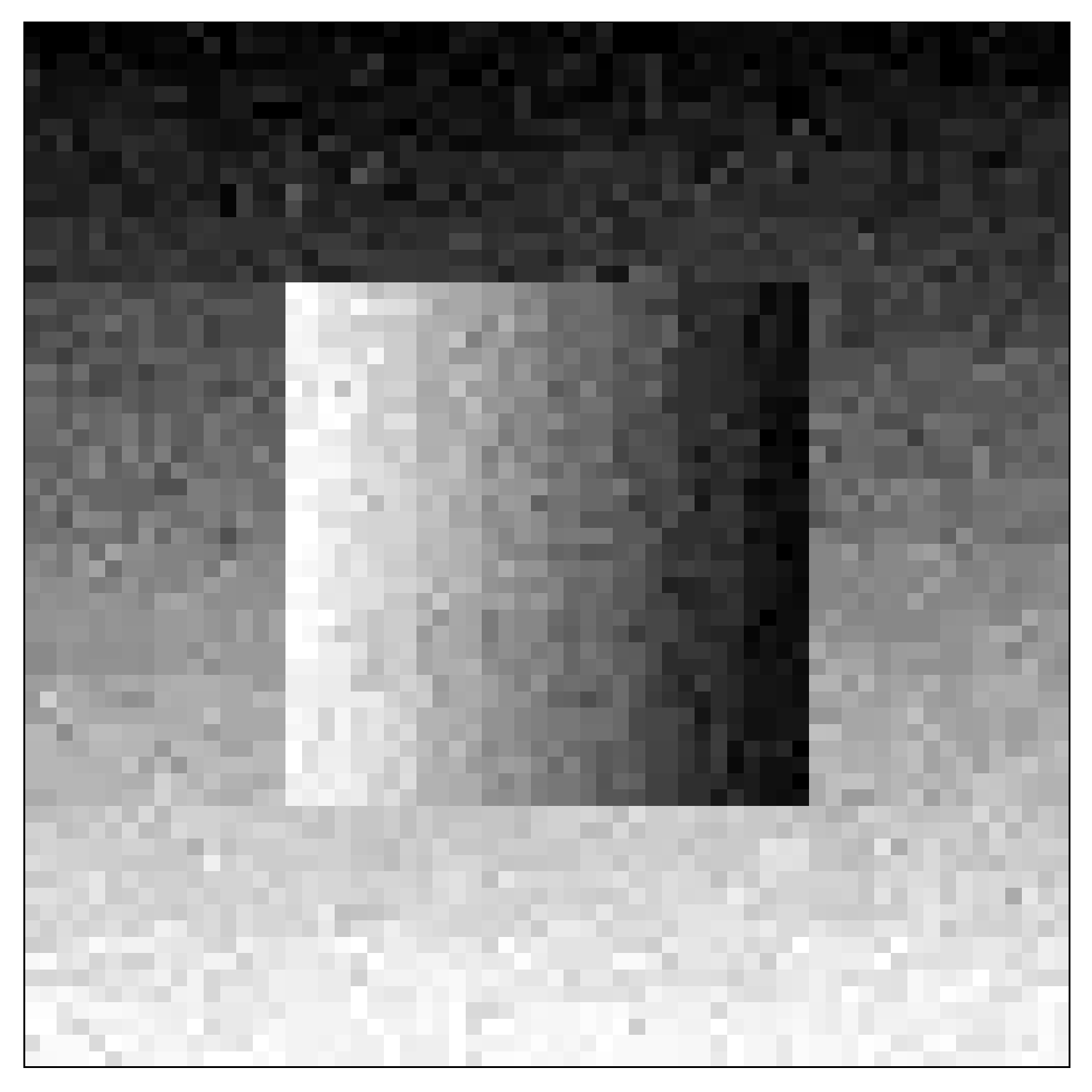}
\end{center}
\vspace{-0.35cm}
\caption{\(k=5\). \(\alpha = 6.23 \times 10^{-2}\).}
\end{subfigure}
\begin{subfigure}{0.33\textwidth}
\begin{center}
\includegraphics[height=4cm]{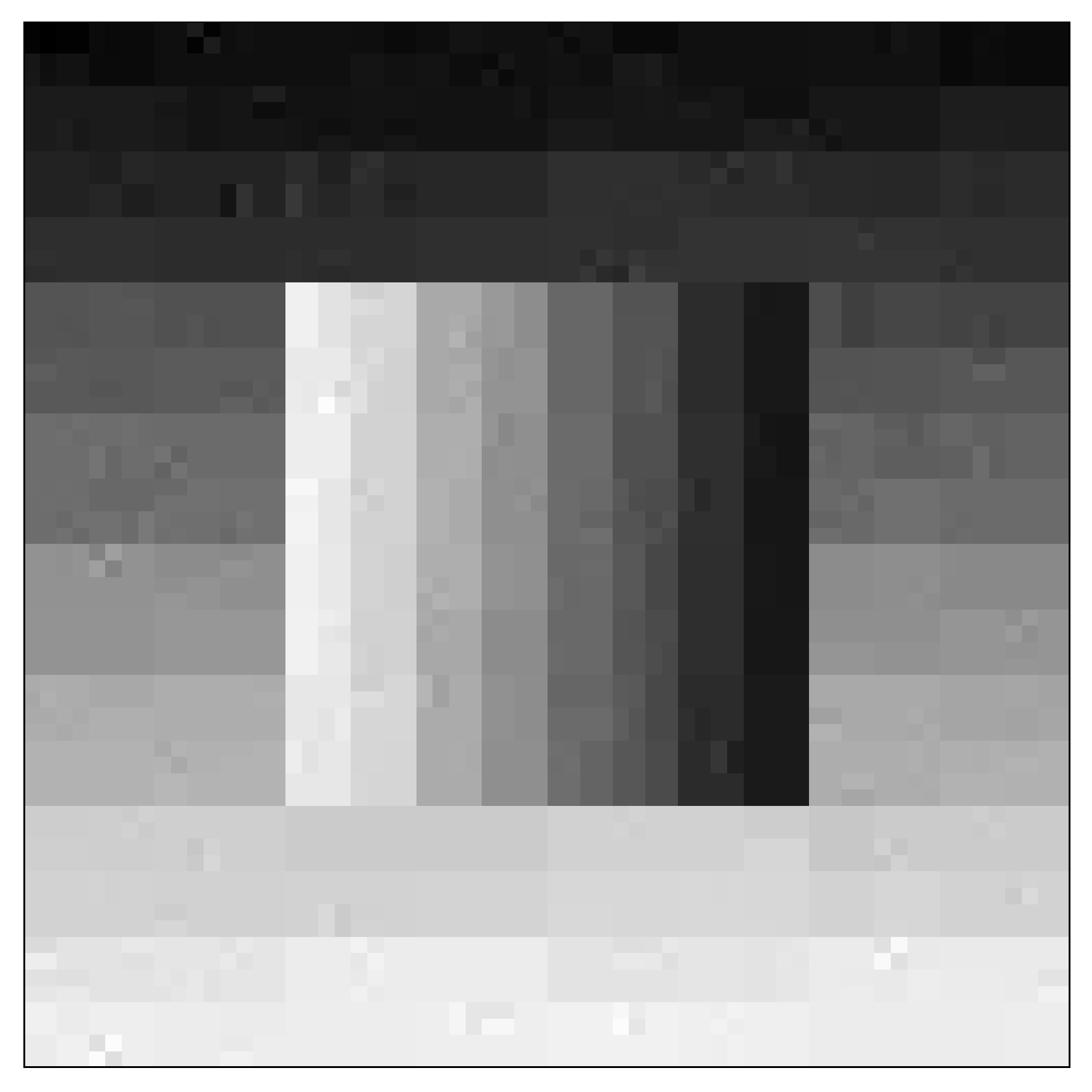}
\end{center}
\vspace{-0.35cm}
\caption{\(k=8\). \(\alpha = 1.54 \times 10^{-1}\).}
\end{subfigure}
\caption{Wavelet denoising with SSIM scoring function and the Itoh--Abe method. Top left: Plot of iterates of the Itoh--Abe method. The rest: Image denoising result at different iterates \(k\).}
\label{fig:wavelet_SSIM}
\end{figuretmp}

\subsubsection{Total variation denoising}

We consider the TV denoising problem
\begin{equation*}
u_\alpha = \argmin_{u \in L^2(\Omega)} \frac{1}{2}\|u - f^\delta\|^2 + \alpha \TV(u),
\end{equation*}
with the SSIM scoring function. We solve the above denoising problem using the PDHG method \citep{cha11}. We set the parameters of the Itoh--Abe method to \(\eps = 10^{-4}\), \(\tau_{\min} = 10^{-5}\), \(\tau_{\max} = 9 \times 10^{-4}\), and \(\eta = 10^{-5}\). See \figref{fig:TV_SSIM} for the numerical results.

\newcommand{\PlotFigZoom}[4]{%
    \begin{tikzpicture}[spy using outlines={every spy on node/.append style={thin}}]
    \draw[white] (0cm, 0cm) node [anchor=south west] {\includegraphics[height=#3]{#2}};%
    \spy [white, draw, width = 2.5cm, height = 2.5cm, magnification = 3,connect spies] on #1 in node [left] at (3.6,-0.3);
    \draw[white] (0.9cm, -.1cm) node [anchor=center] {\footnotesize #4};%
    \end{tikzpicture}}%

\begin{figuretmp}
\begin{subfigure}{0.66\textwidth}
\begin{center}
\includegraphics[height=5.3cm]{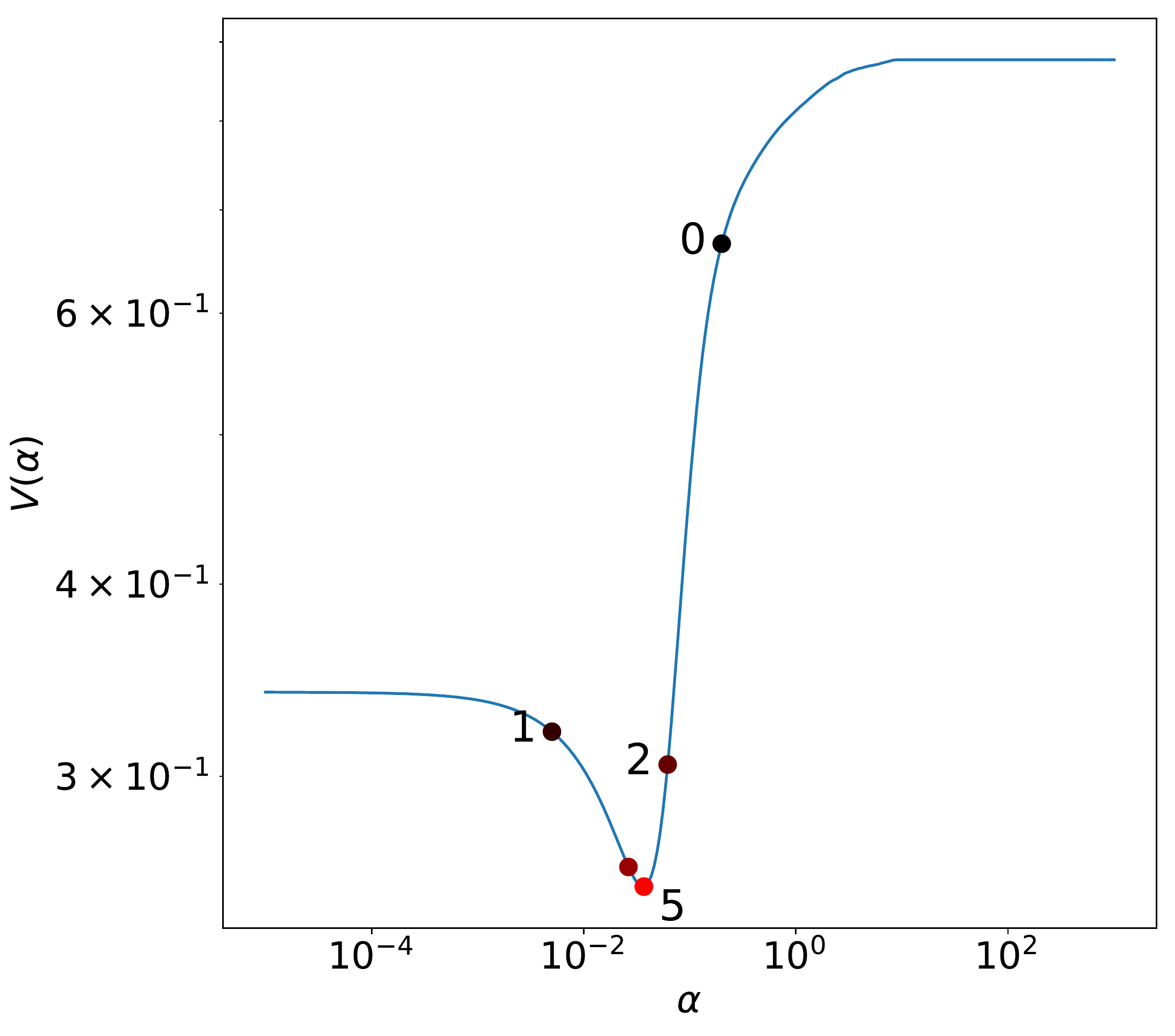}
\end{center}
\vspace{-0.4cm}
\caption{Plot with labels.}
\end{subfigure}
\begin{subfigure}{0.33\textwidth}
\PlotFigZoom{(2.1,2.3)}{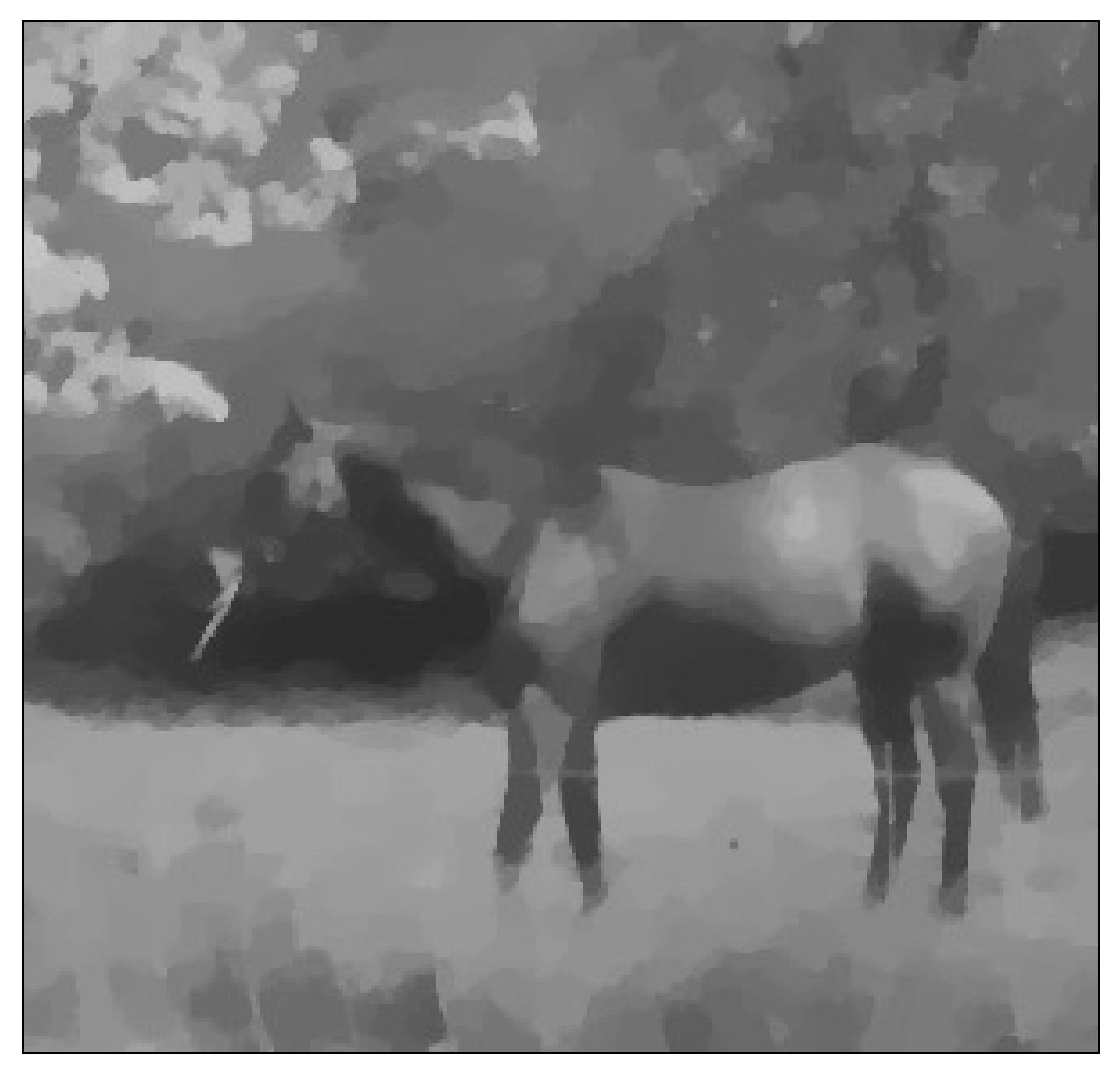}{3.6cm}{}
\caption{\(k = 0.\) \(\alpha = 2.00 \times 10^{-1}\).}
\end{subfigure}
\begin{subfigure}{0.33\textwidth}
\PlotFigZoom{(2.1,2.3)}{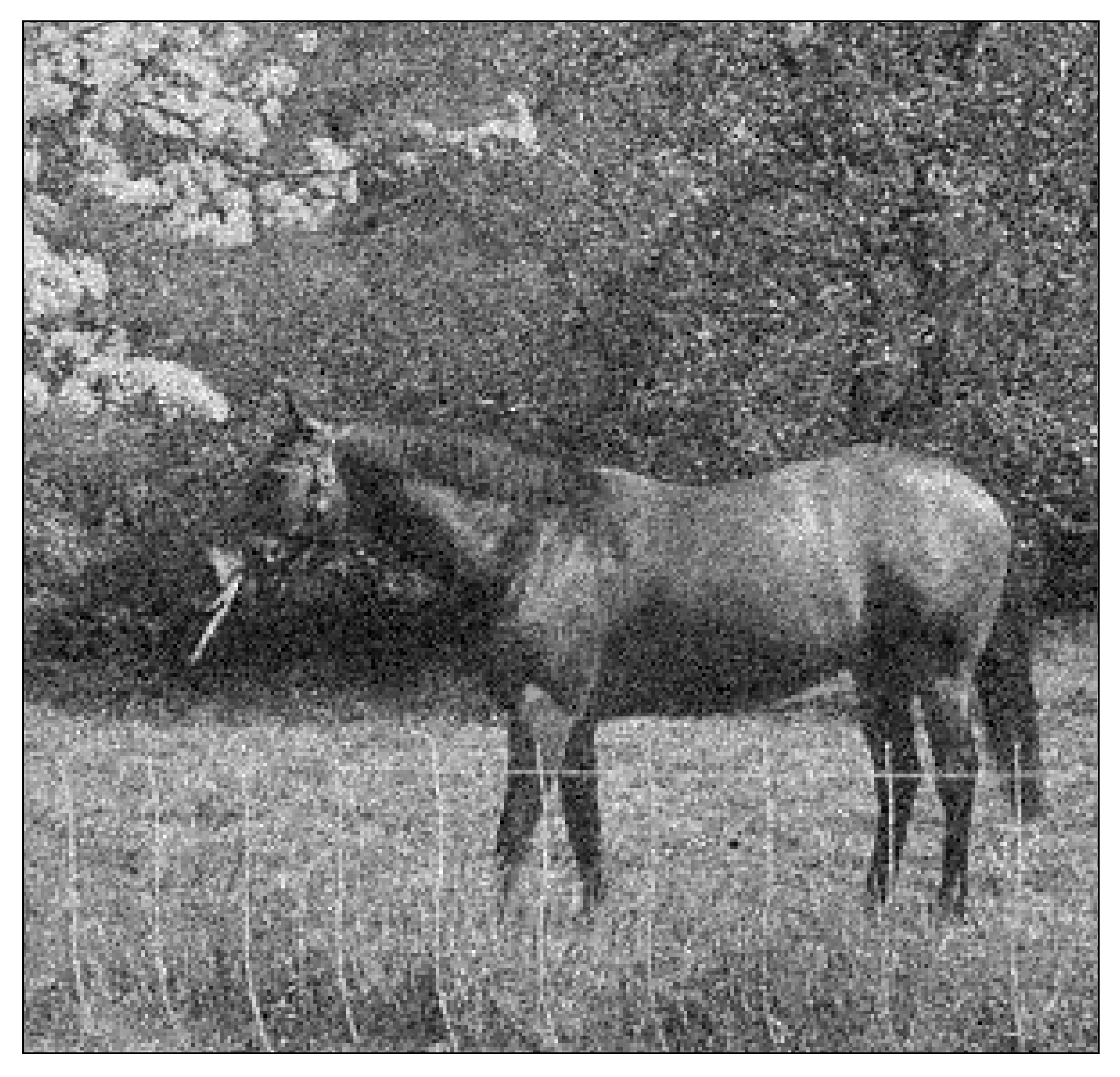}{3.6cm}{}
\caption{\(k = 0.\) \(\alpha = 2.00 \times 10^{-1}\).}
\end{subfigure}
\begin{subfigure}{0.33\textwidth}
\PlotFigZoom{(2.1,2.3)}{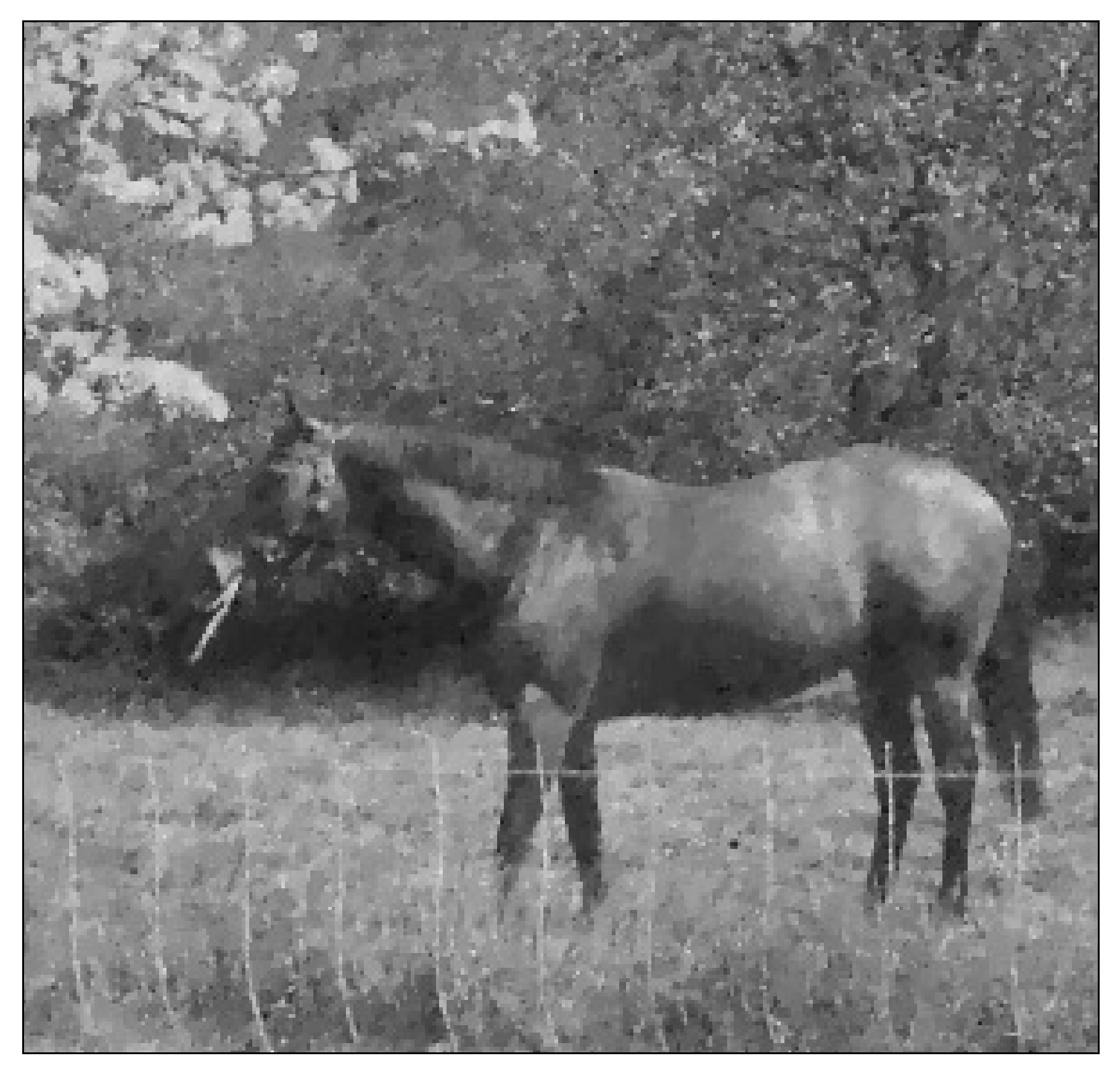}{3.6cm}{}
\caption{\(k = 0.\) \(\alpha = 2.00 \times 10^{-1}\).}
\end{subfigure}
\begin{subfigure}{0.33\textwidth}
\PlotFigZoom{(2.1,2.3)}{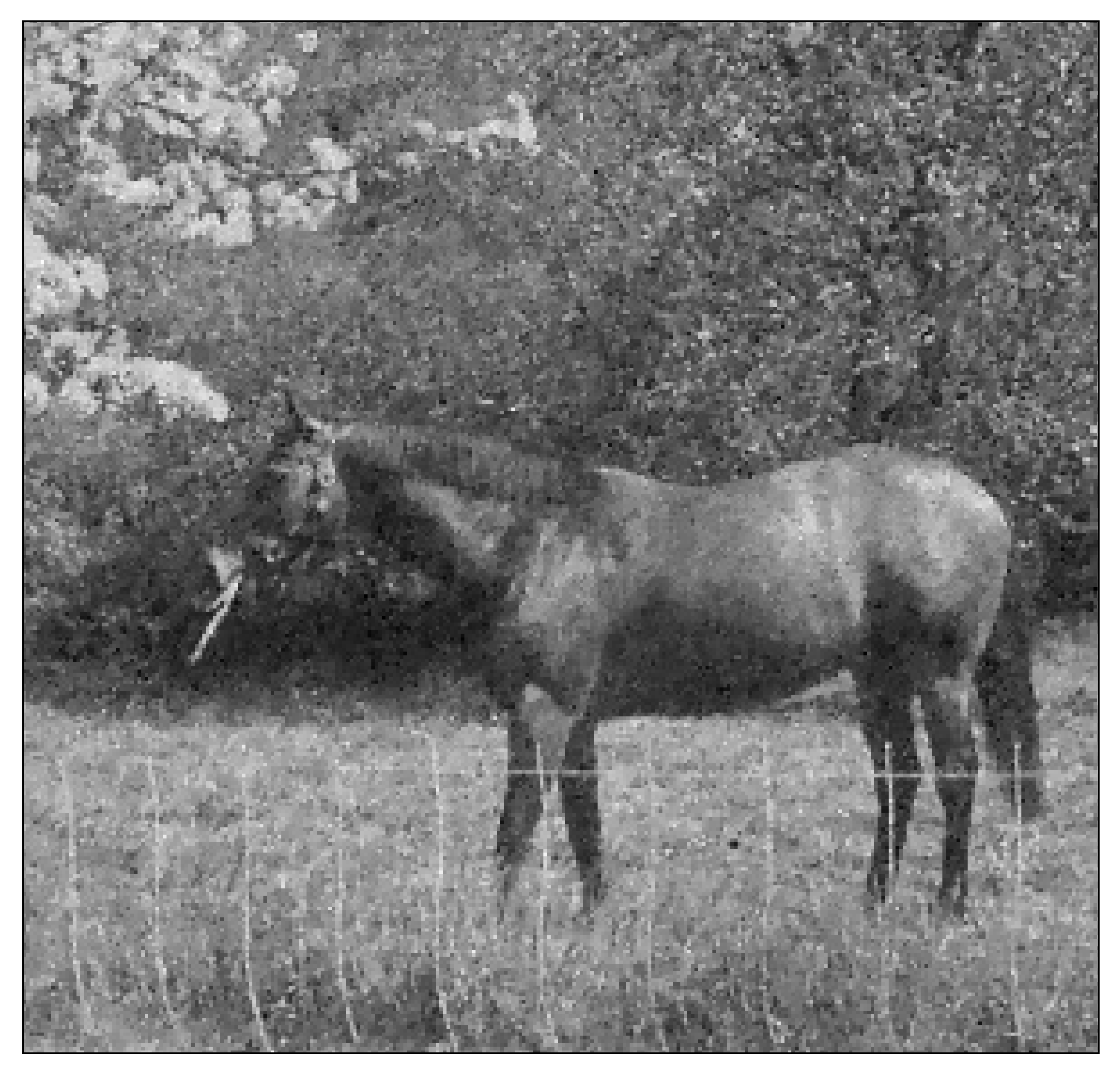}{3.6cm}{}
\caption{\(k = 0.\) \(\alpha = 2.00 \times 10^{-1}\).}
\end{subfigure}
\caption{TV denoising with SSIM scoring function and the Itoh--Abe method. Top left: Plot of iterates of the Itoh--Abe method. The rest: Image denoising result at different iterates \(k\), with a zoom to show the difference.}
\label{fig:TV_SSIM}
\end{figuretmp}

\subsubsection{Total generalised variation regularisation}

We now consider the second-order total generalised variation (TGV) regulariser \(R_{\alpha_1,\alpha_2}(u) = \TGV_{\alpha_1,\alpha_2}^2(u)\) for image denoising, with the scoring function
\[
\Phi(u) := 1 - \SSIM(u, u^\dagger).
\]
Like for TV denoising, we solve the denoising problem using the PDHG method. We set the parameters of the randomised Itoh--Abe (RIA) method to \(\eps = 10^{-1}\), \(\tau_{\min} = 10^{-3}\), \(\tau_{\max} = 10^5\), and \(\eta = 10^{-20}\). See \figref{fig:tgv_dg} for the numerical results.

\begin{figuretmp}
\centering
\begin{subfigure}{0.49\textwidth}
\centering
\includegraphics[height=5.5cm]{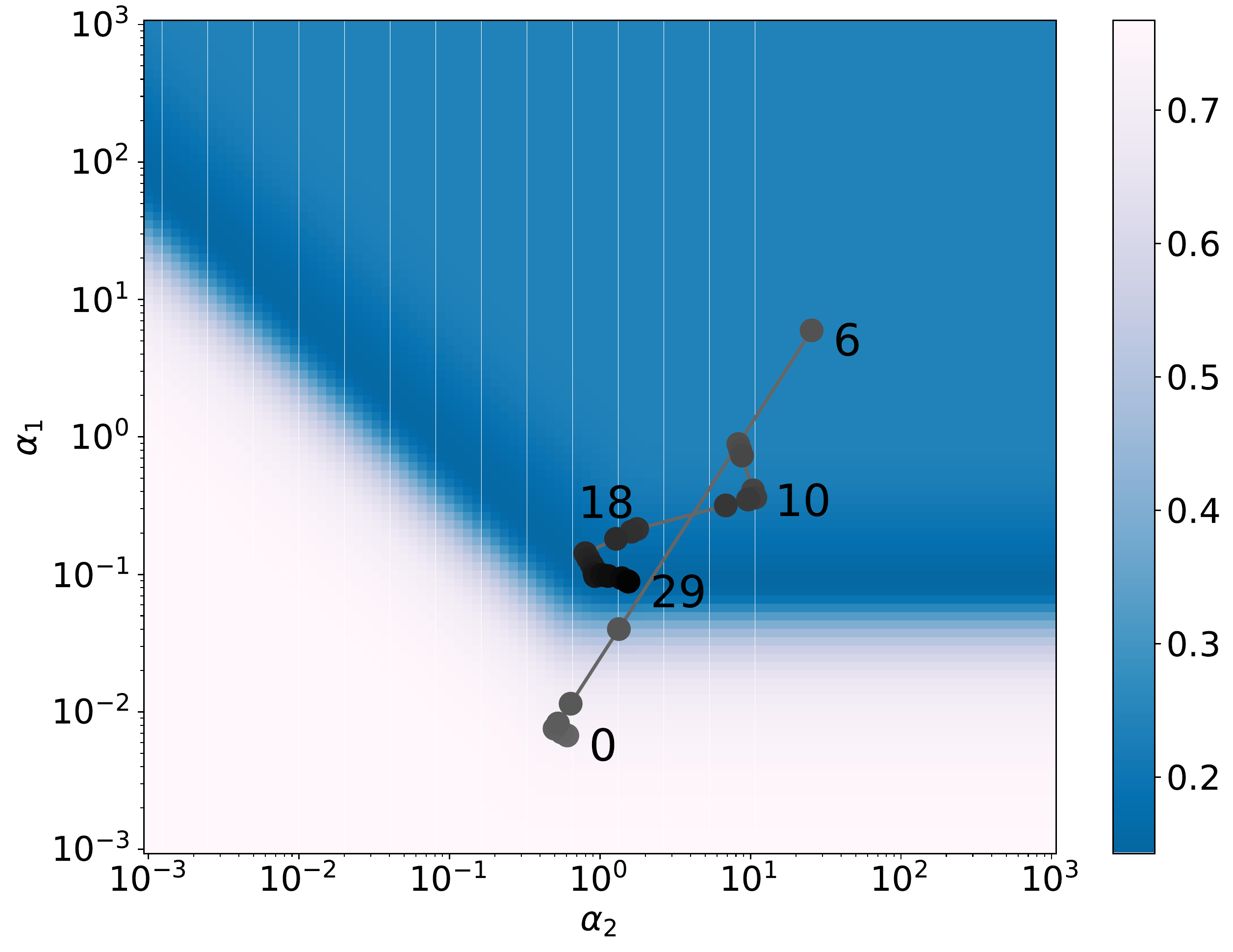}
\caption{}
\end{subfigure}
\begin{subfigure}{0.48\textwidth}
\centering
\includegraphics[height=5.5cm]{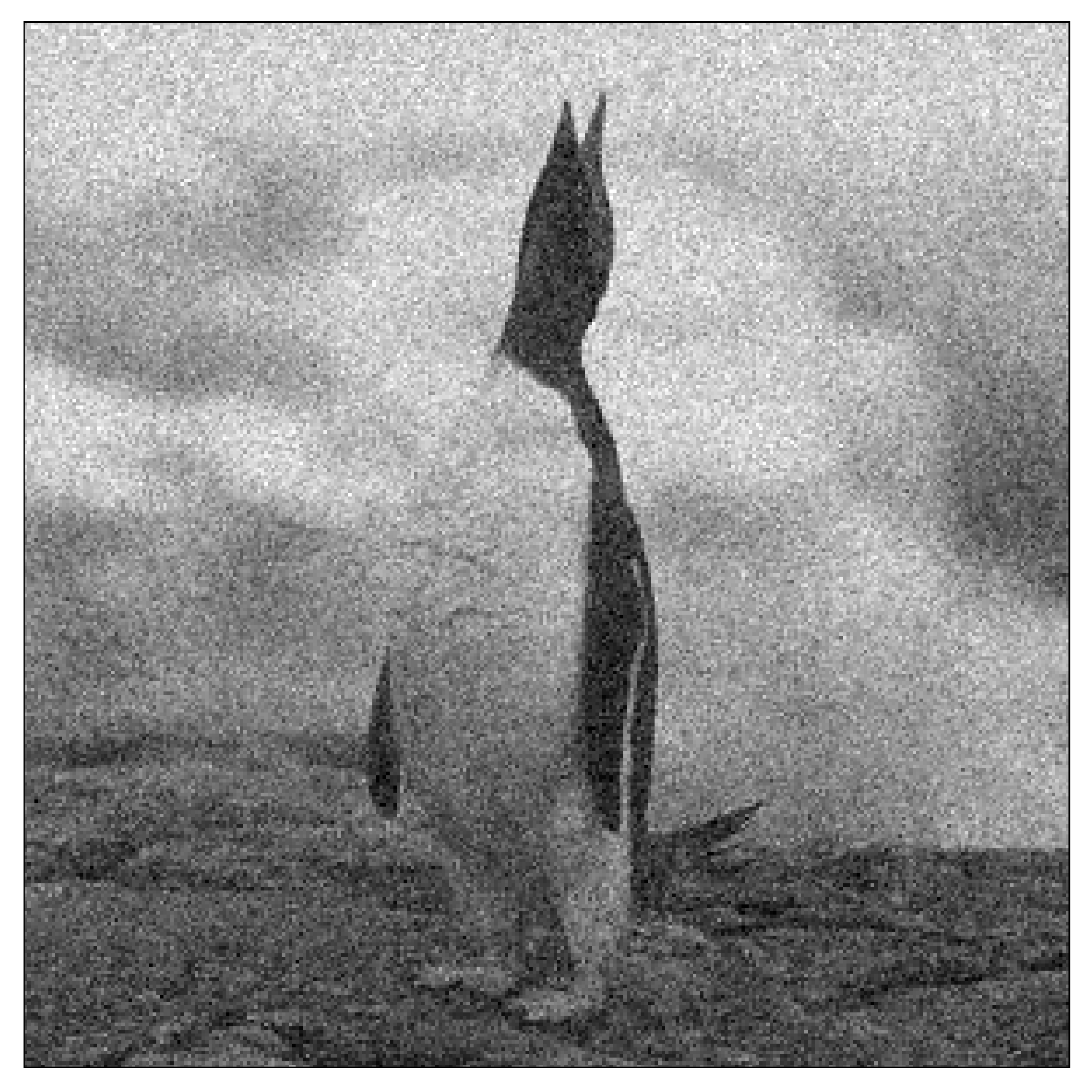}
\caption{\(j= 0\), \(\quad \alpha_1 = 6.74 \times 10^{-3}\), \(\quad \alpha_2 = 6.07 \times 10^{-1}\).}
\end{subfigure}
\begin{subfigure}{0.48\textwidth}
\centering
\includegraphics[height=5.5cm]{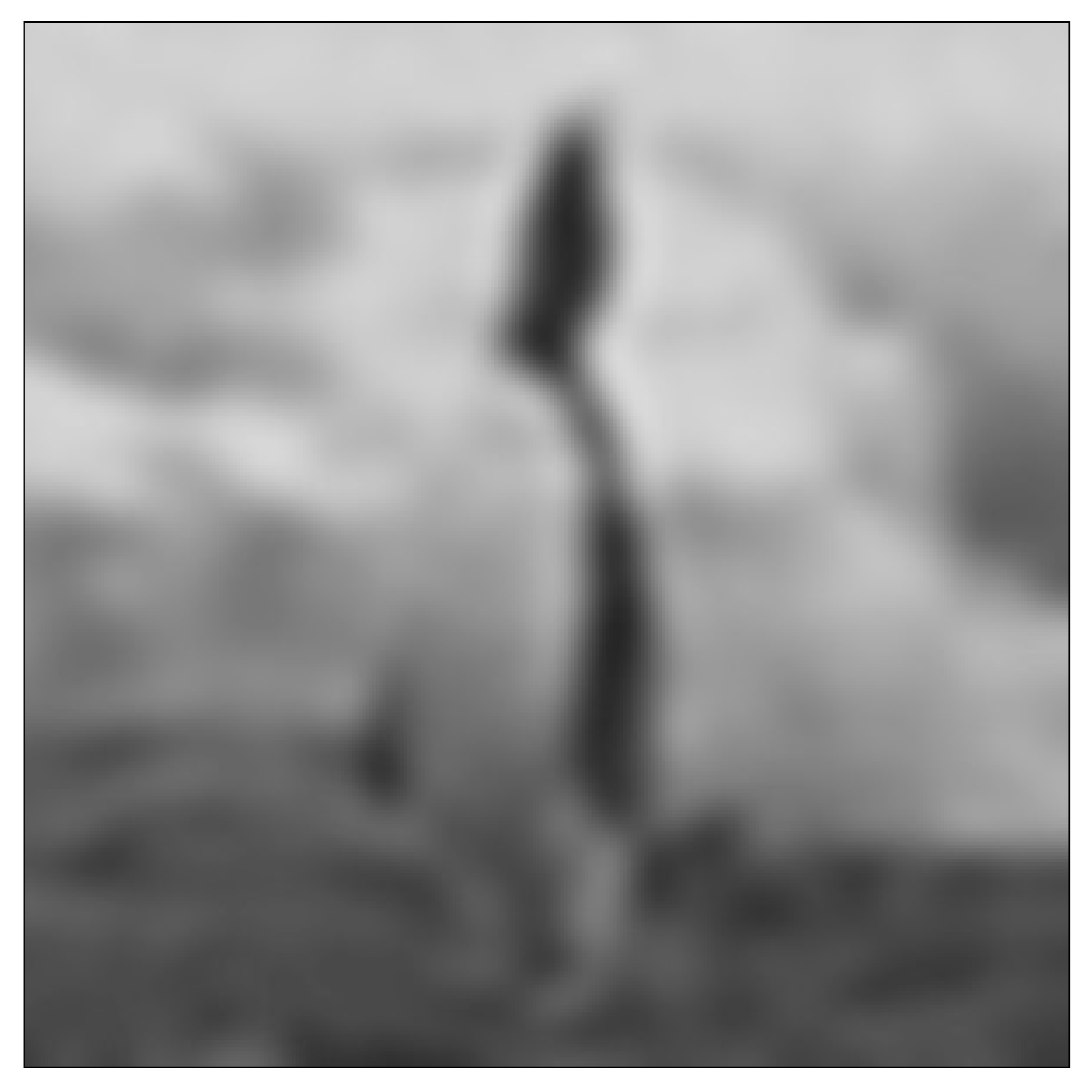}
\caption{\(j= 6,\quad \alpha_1 = 5.96,\quad \alpha_2 = 25.5\).}
\end{subfigure}
\begin{subfigure}{0.48\textwidth}
\centering
\includegraphics[height=5.5cm]{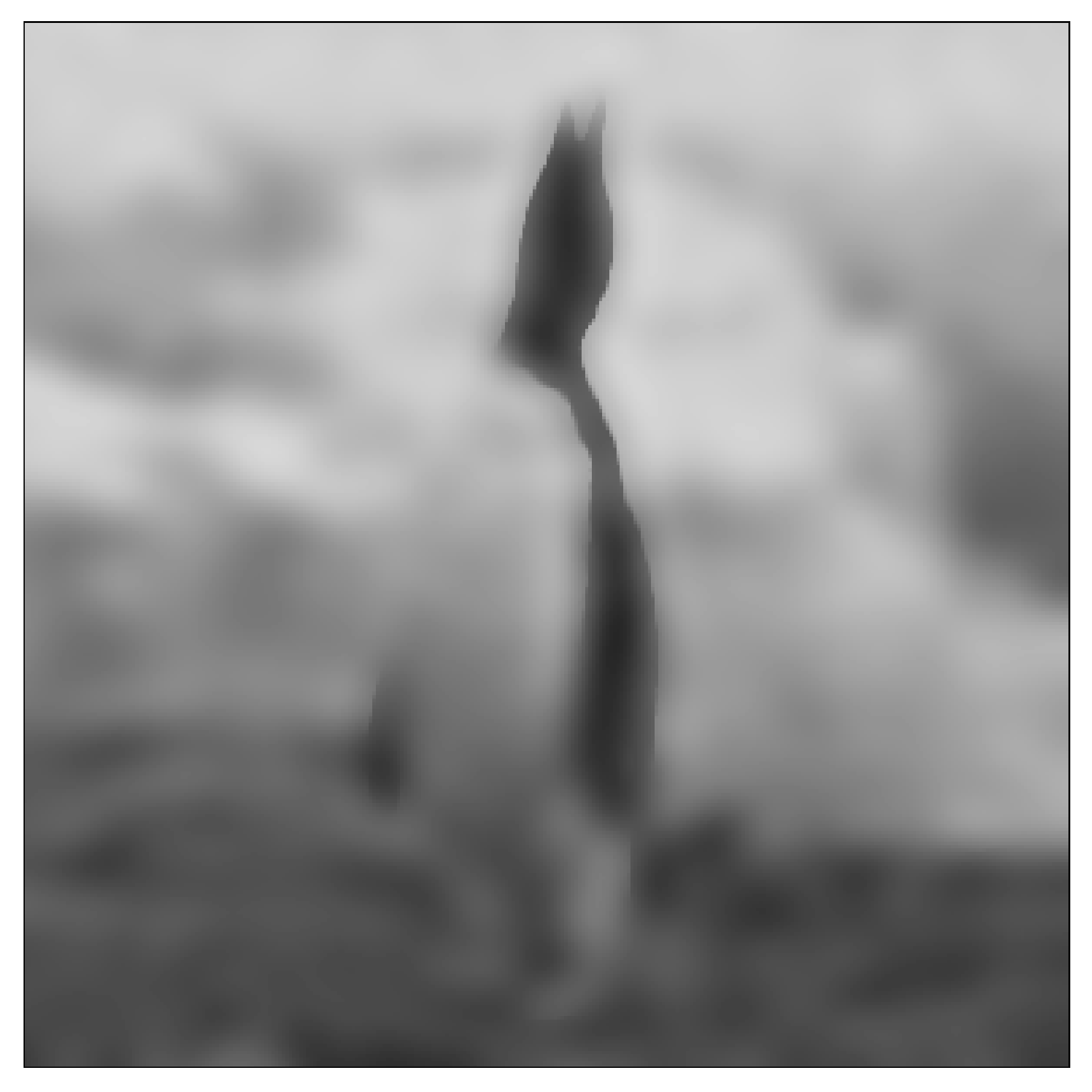}
\caption{\(j= 10, \quad \alpha_1 = 4.09 \times 10^{-1}, \quad \alpha_2 = 10.4\).}
\end{subfigure}
\begin{subfigure}{0.48\textwidth}
\centering
\includegraphics[height=5.5cm]{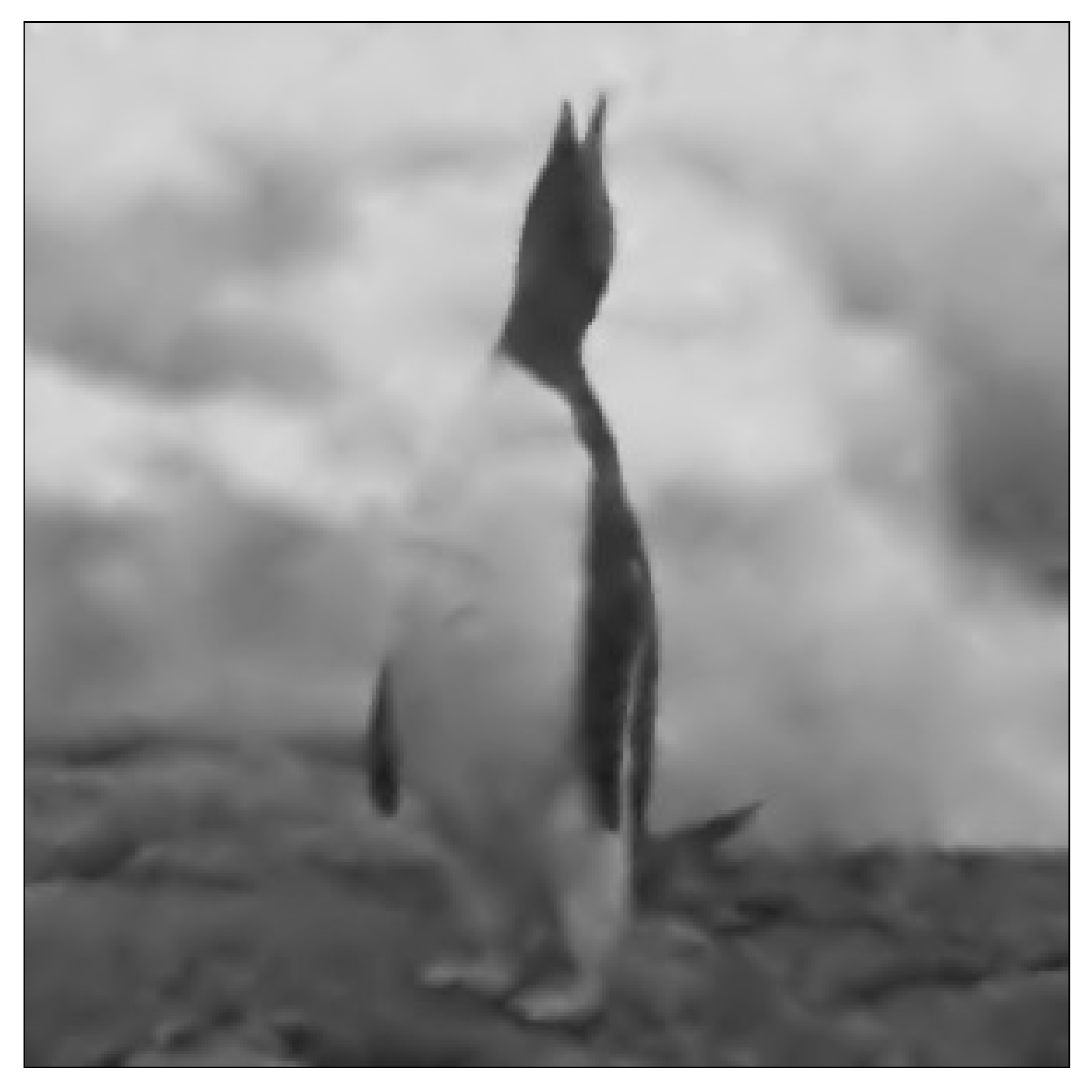}
\caption{\(j= 18, \quad \alpha_1 = 1.43 \times 10^{-1}, \quad \alpha_2 = 7.99 \times 10^{-1}\).}
\end{subfigure}
\begin{subfigure}{0.48\textwidth}
\centering
\includegraphics[height=5.5cm]{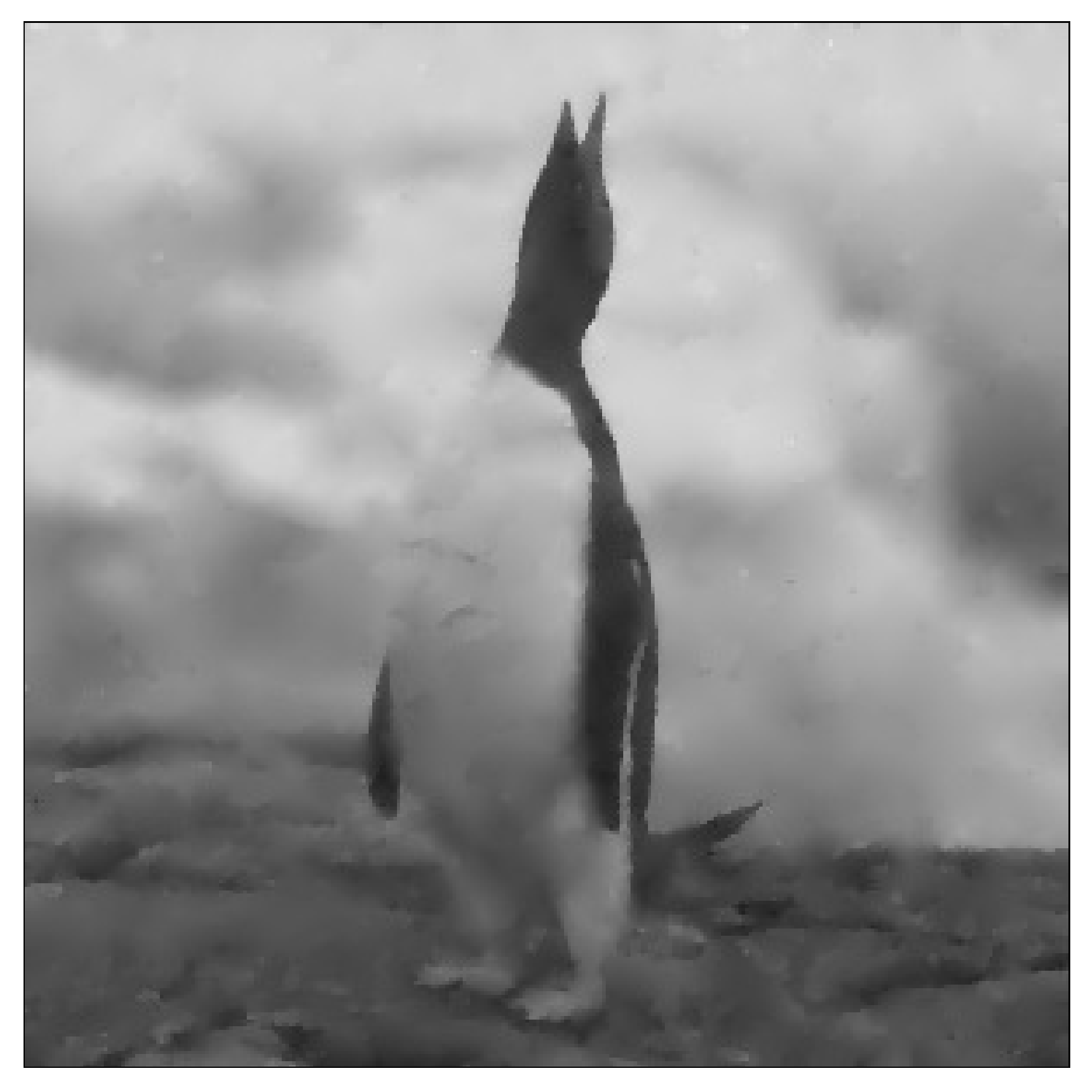}
\caption{\(j= 29\), \(\alpha_1 = 8.87 \times 10^{-2}\), \(\alpha_2 = 1.55\).}
\end{subfigure}
\caption{TGV denoising with SSIM scoring function and the Itoh--Abe method. Top left: Plot of iterates of the method. The rest: Image denoising result at different function evaluations \(j\).}
\label{fig:tgv_dg}
\end{figuretmp}

We compare these results to the results from the Py-BOBYQA and LT-MADS solvers. We set the parameters of Py-BOBYQA to \(\mbox{rhobeg} = 2\), \(\mbox{rhoend} = 10^{-10}\) and \(\mbox{npt} = 2(n+1)\) and the parameters of LT-MADS to \(\mbox{DIRECTION\_TYPE} = \mbox{LT } 2N\). See the results for two different starting points in \figref{fig:tgv_rate} and \ref{fig:tgv_rate_2}. We note that the objective function is approximately stationary across a range of values, which leads to the different points of convergence, and different limiting values of the objective function for different methods. We see that the methods are all of comparable efficiency, although the Itoh--Abe method is slower initially. The Itoh--Abe method seems to be the most efficient, once it is within a neighborhood of the minimiser.

\begin{figuretmp}
\begin{subfigure}{0.45\textwidth}
\begin{tikzpicture}
\draw (0,0) node[inner sep = 0] {\includegraphics[height=5cm]{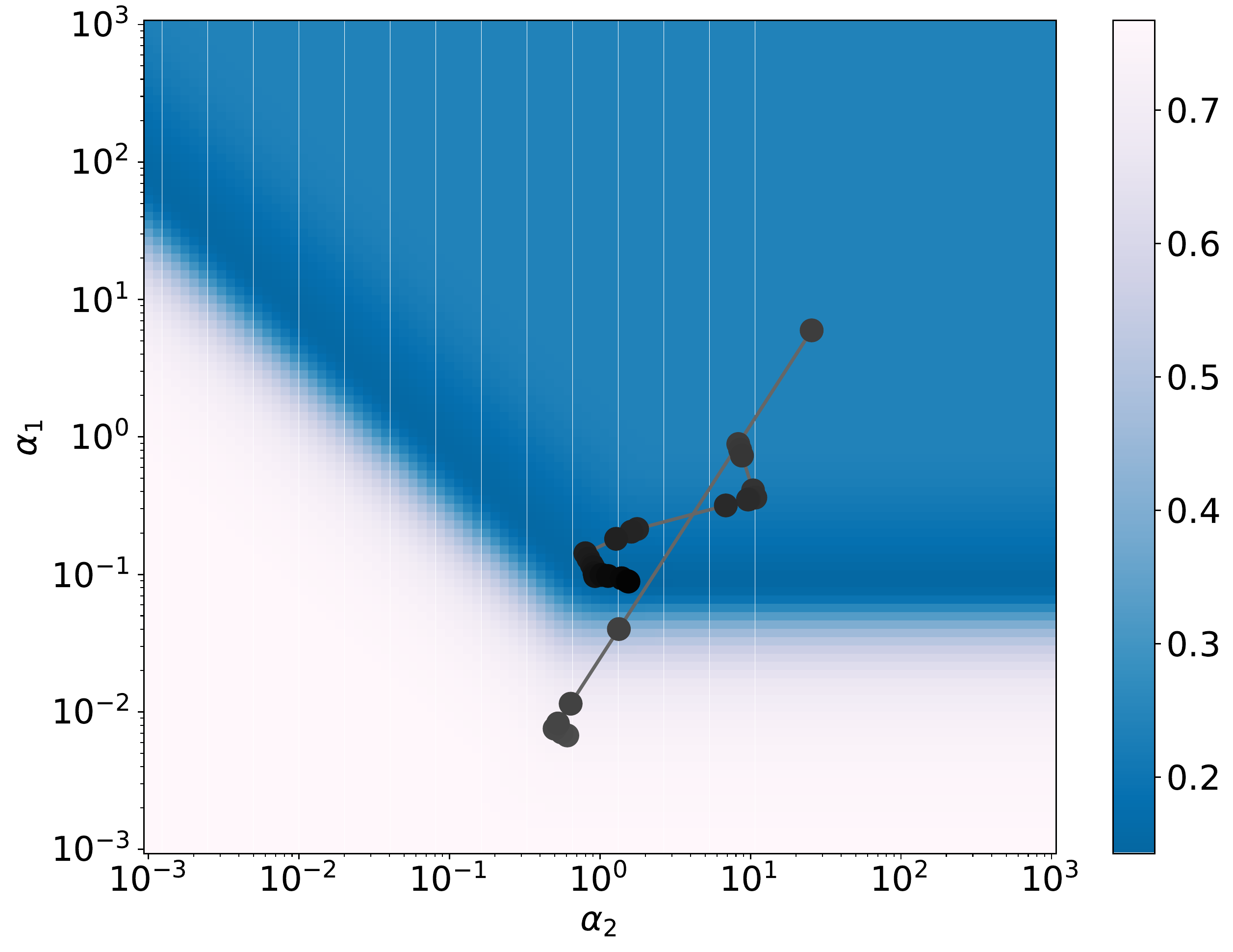}};
\node[draw, fill = white] at (-2,-1.7) {\scriptsize RIA} ;
\end{tikzpicture}
\end{subfigure}
\begin{subfigure}{0.45\textwidth}
\begin{tikzpicture}
\draw (0,0) node[inner sep = 0] {\includegraphics[height=5cm]{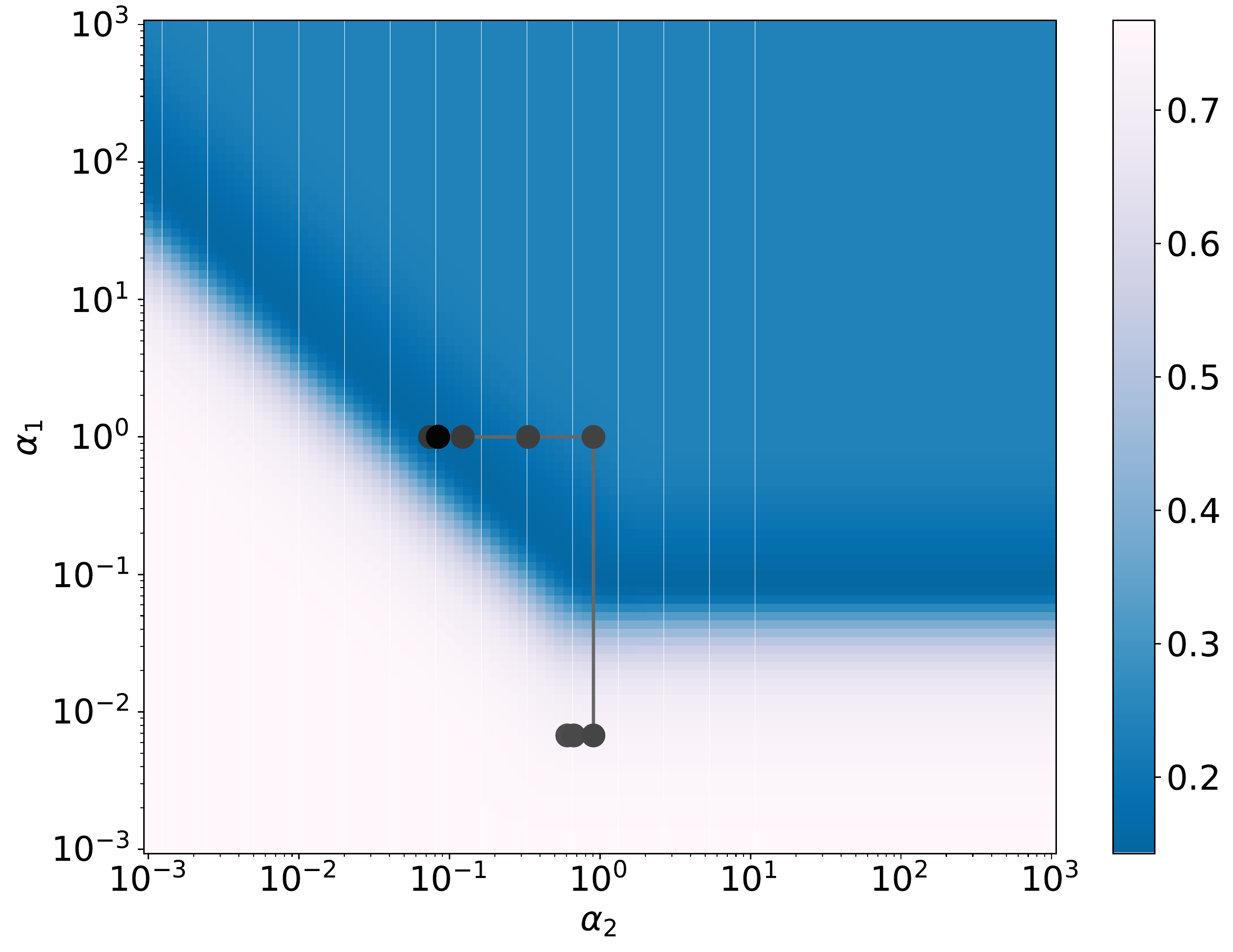}};
\node[draw, fill = white] at (-1.7,-1.7) {\scriptsize LT-MADS} ;
\end{tikzpicture}
\end{subfigure}
\begin{subfigure}{0.45\textwidth}
\begin{tikzpicture}
\draw (0,0) node[inner sep = 0] {\includegraphics[height=5cm]{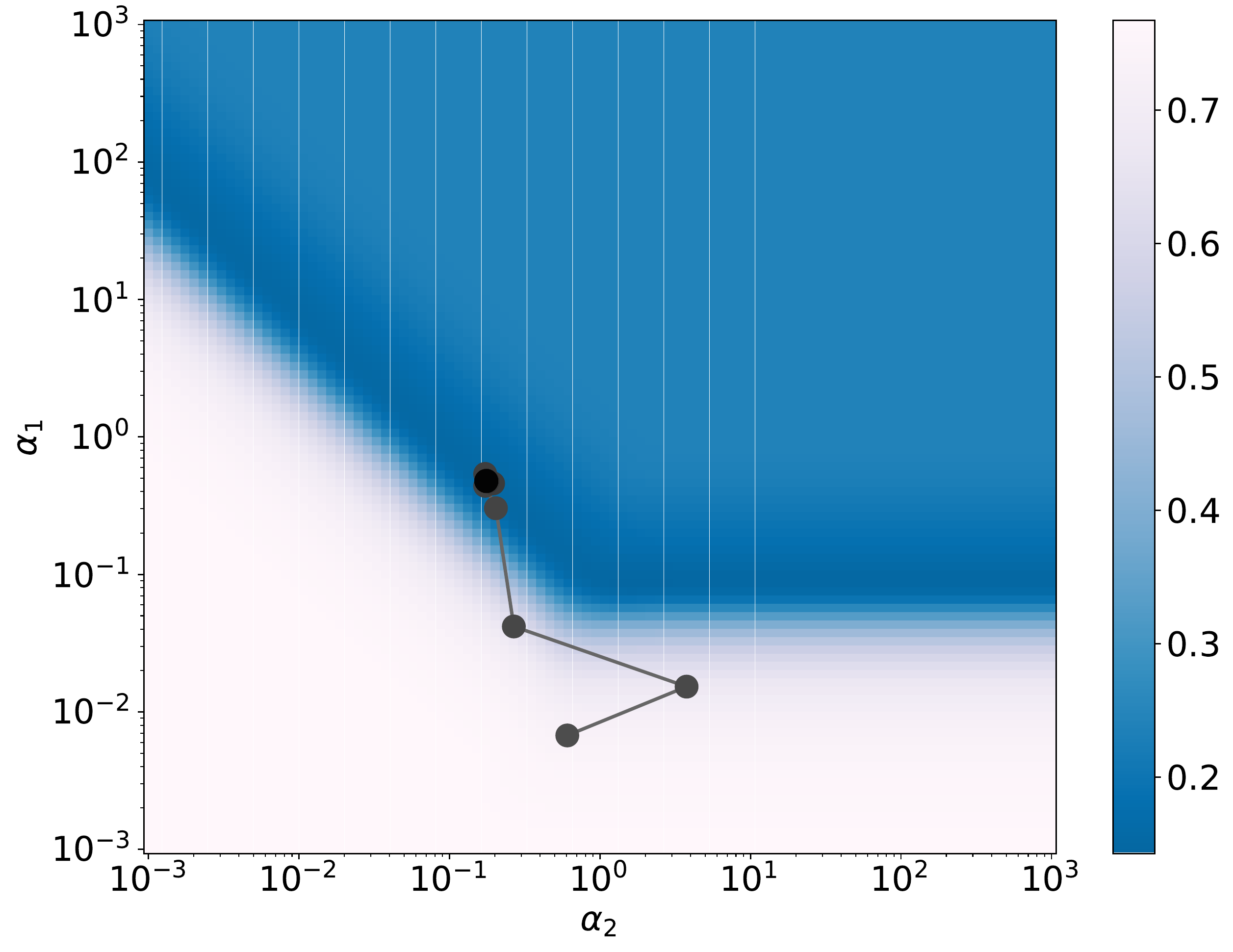}};
\node[draw, fill = white] at (-1.55,-1.7) {\scriptsize Py-BOBYQA} ;
\end{tikzpicture}
\end{subfigure}
\hspace{0.93cm}
\begin{subfigure}{0.45\textwidth}
\includegraphics[height=5cm]{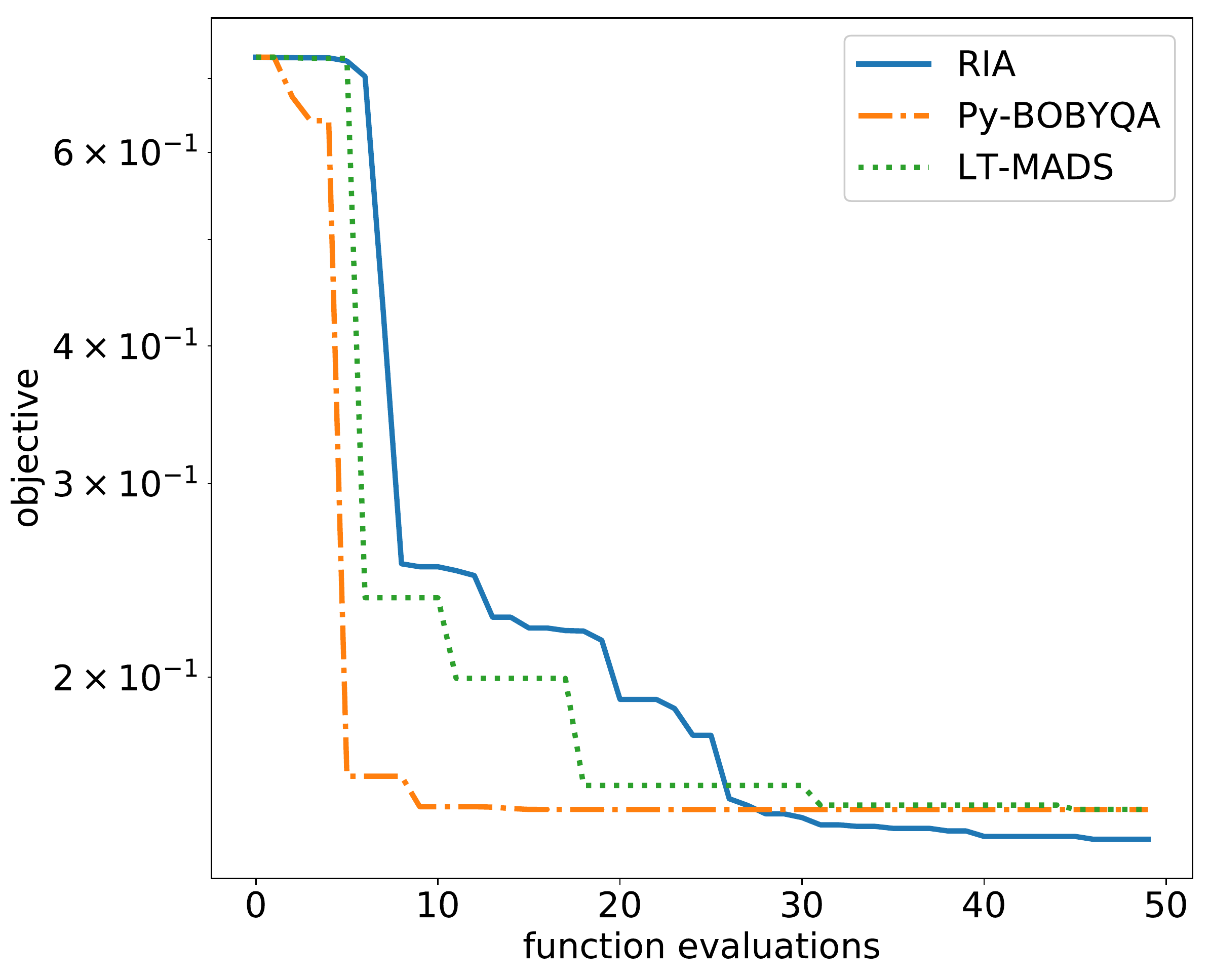}
\end{subfigure}
\caption{Comparison of optimisation methods for TGV denoising with SSIM scoring function. Top left: Plot of iterates of the Itoh--Abe method. Top right: Plot of iterates of the LT-MADS method. Bottom left: Plot of iterates of the Py-BOBYQA method. Bottom right: Comparison of convergence rates for the methods with respect to function evaluations.}
\label{fig:tgv_rate}
\end{figuretmp}

\begin{figuretmp}
\begin{subfigure}{0.45\textwidth}
\begin{tikzpicture}
\draw (0,0) node[inner sep = 0] {\includegraphics[height=5cm]{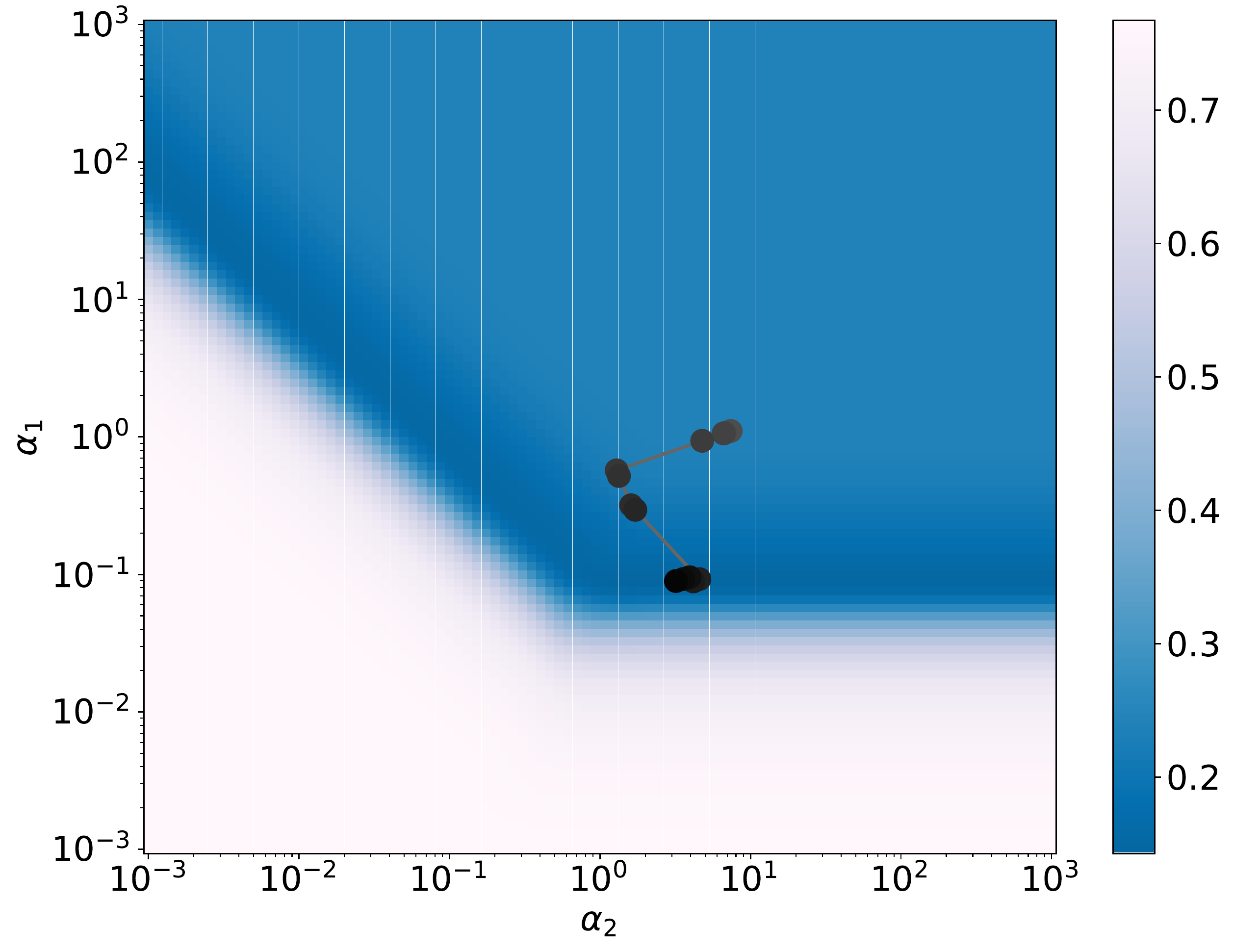}};
\node[draw, fill = white] at (-2.05,-1.7) {RIA} ;
\end{tikzpicture}
\end{subfigure}
\begin{subfigure}{0.45\textwidth}
\begin{tikzpicture}
\draw (0,0) node[inner sep = 0] {\includegraphics[height=5cm]{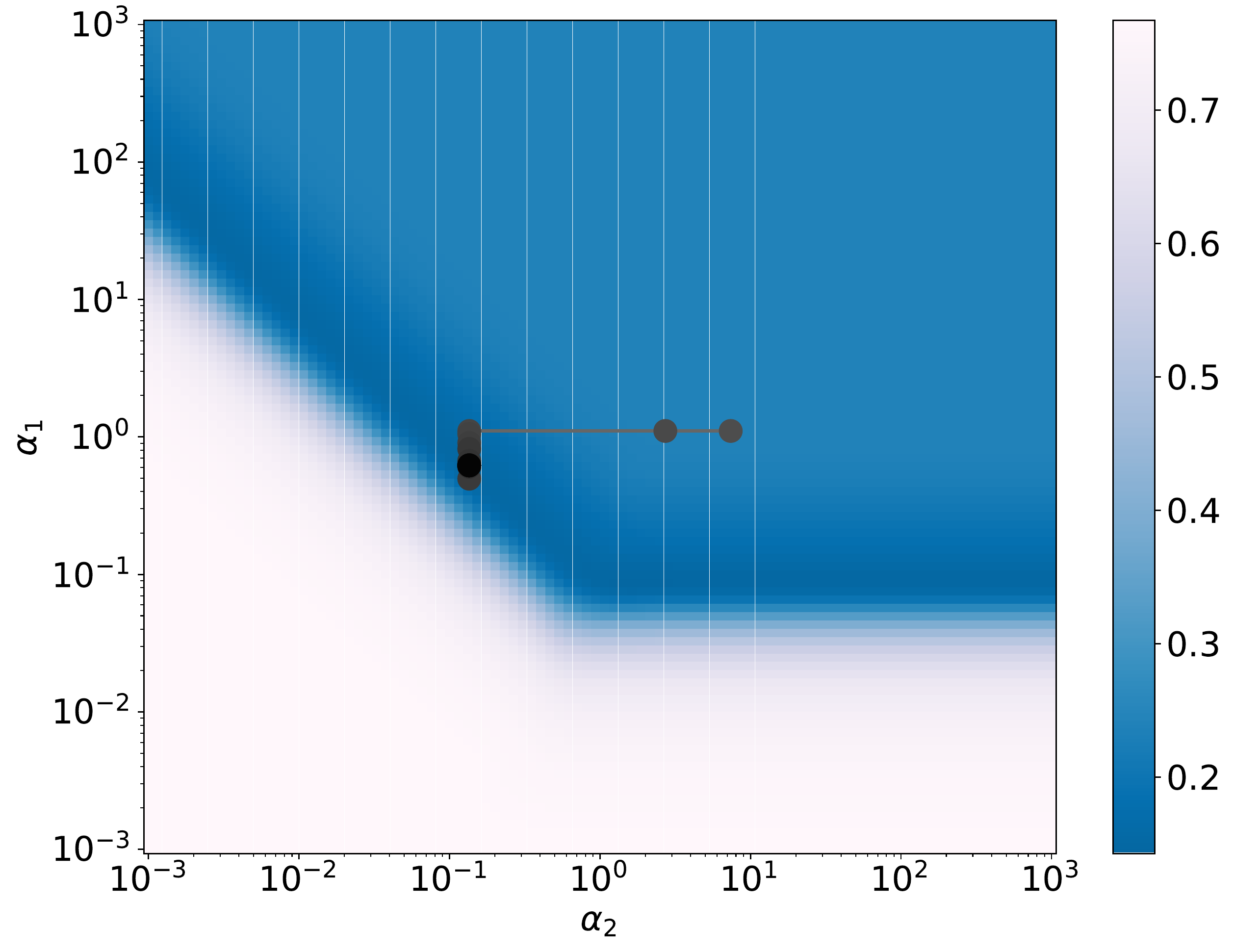}};
\node[draw, fill = white] at (-1.65,-1.7) {LT-MADS} ;
\end{tikzpicture}
\end{subfigure}
\begin{subfigure}{0.45\textwidth}
\begin{tikzpicture}
\draw (0,0) node[inner sep = 0] {\includegraphics[height=5cm]{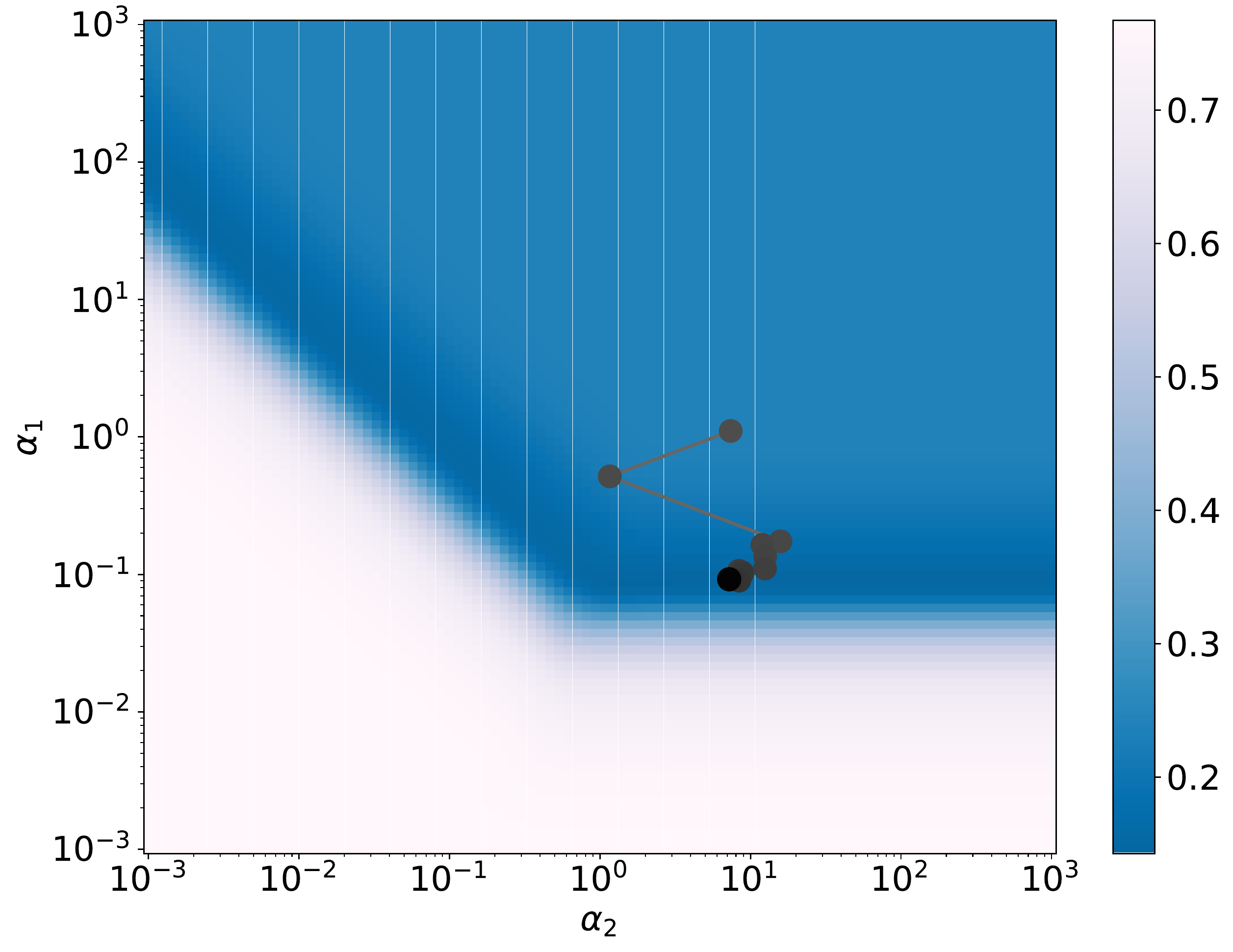}};
\node[draw, fill = white] at (-1.5,-1.7) {Py-BOBYQA} ;
\end{tikzpicture}
\end{subfigure}
\hspace{0.93cm}
\begin{subfigure}{0.45\textwidth}
\includegraphics[height=5cm]{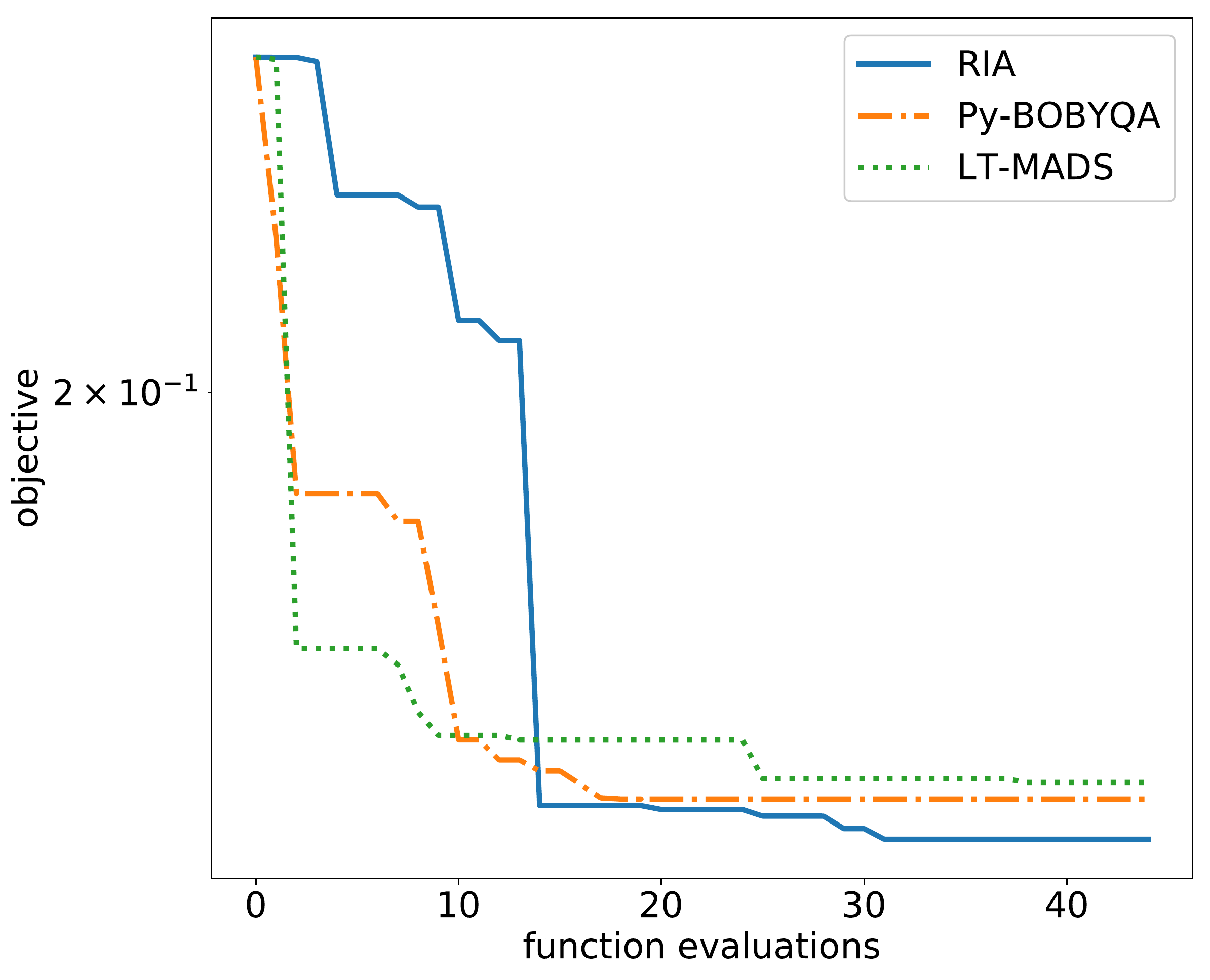}
\end{subfigure}
\caption{Comparison of optimisation methods for TGV denoising with SSIM scoring function for a different starting point. Top left: Plot of iterates of the Itoh--Abe method. Top right: Plot of iterates of the LT-MADS method. Bottom left: Plot of iterates of the Py-BOBYQA method. Bottom right: Comparison of convergence rates for the methods with respect to function evaluations.}
\label{fig:tgv_rate_2}
\end{figuretmp}

\section{Conclusion}

\label{sec:conclusion}

In this paper, we have shown that the randomised Itoh--Abe methods are efficient and robust schemes for solving unconstrained nonsmooth, nonconvex problems without the use of gradients or subgradients. Furthermore, the favourable rates of dissipativity that the discrete gradient method inherits from the gradient flow system extends to the nonsmooth case. We show, under minimal assumptions on the objective function, that the methods admit a solution that is computationally tractable, and the iterates converge to a connected set of Clarke stationary points. Through examples, the assumptions are also shown to be necessary.

The methods are shown to be robust and versatile optimisation schemes. It locates the global minimisers of the Rosenbrock function and a variant of Nesterov's nonsmooth Chebyshev--Rosenbrock functions. The efficiency of the Itoh--Abe discrete gradient method for smooth problems has already been demonstrated elsewhere \citep{gri17, miy17, rin17}. We also consider its application to bilevel learning problems and compare its performance to the derivative-free Py-BOBYQA and LT-MADS methods.

Future work will be dedicated to adapting the randomised Itoh--Abe methods for constrained optimisation problems, establishing convergence of the iterates of the method for Kurdyka-{\L}ojasiewicz functions \citep{att10}, and analysing the Lipschitz continuity properties of bilevel optimisation for variational regularisation problems.

\begin{acknowledgements}
The authors give thanks to Lindon Roberts for helpful discussions and for providing code for Py-BOBYQA, and to Antonin Chambolle for helpful discussions.
\end{acknowledgements}


\bibliographystyle{spmpsci}      

\bibliography{dg_refs}

\begin{thebibliography}{10}
\providecommand{\url}[1]{{#1}}
\providecommand{\urlprefix}{URL }
\expandafter\ifx\csname urlstyle\endcsname\relax
  \providecommand{\doi}[1]{DOI~\discretionary{}{}{}#1}\else
  \providecommand{\doi}{DOI~\discretionary{}{}{}\begingroup
  \urlstyle{rm}\Url}\fi

\bibitem{att10}
Attouch, H., Bolte, J., Redont, P., Soubeyran, A.: Proximal alternating
  minimization and projection methods for nonconvex problems: An approach based
  on the {K}urdyka-{{\L}}sojasiewicz inequality.
\newblock Mathematics of Operations Research \textbf{35}(2), 438--457 (2010)

\bibitem{aud06}
Audet, C., Dennis~Jr, J.E.: Mesh adaptive direct search algorithms for
  constrained optimization.
\newblock SIAM Journal on optimization \textbf{17}(1), 188--217 (2006)

\bibitem{aud17}
Audet, C., Hare, W.: Derivative-Free and Blackbox Optimization, 1st edn.
\newblock Springer Series in Operations Research and Financial Engineering.
  Springer International Publishing (2017)

\bibitem{aus98}
Aussel, D.: Subdifferential properties of quasiconvex and pseudoconvex
  functions: unified approach.
\newblock Journal of optimization theory and applications \textbf{97}(1),
  29--45 (1998)

\bibitem{bag08}
Bagirov, A.M., Karas{\"o}zen, B., Sezer, M.: Discrete gradient method:
  derivative-free method for nonsmooth optimization.
\newblock Journal of Optimization Theory and Applications \textbf{137}(2),
  317--334 (2008)

\bibitem{ben18}
{Benning}, M., {Burger}, M.: {Modern Regularization Methods for Inverse
  Problems}.
\newblock ArXiv e-prints  (2018).
\newblock \urlprefix\url{http://arxiv.org/abs/1801.09922}

\bibitem{bor99}
Borwein, J.M., Zhu, Q.J.: A survey of subdifferential calculus with
  applications.
\newblock Nonlinear Analysis: Theory, Methods \& Applications \textbf{38}(6),
  687--773 (1999)

\bibitem{bre15}
Bredies, K., Holler, M.: A {TGV}-based framework for variational image
  decompression, zooming, and reconstruction. part {I}: Analytics.
\newblock SIAM Journal on Imaging Sciences \textbf{8}(4), 2814--2850 (2015)

\bibitem{bre10}
Bredies, K., Kunisch, K., Pock, T.: Total generalized variation.
\newblock SIAM Journal on Imaging Sciences \textbf{3}(3), 492--526 (2010)

\bibitem{bur13}
Burger, M., Osher, S.: A guide to the {TV} zoo.
\newblock In: Level set and PDE based reconstruction methods in imaging, pp.
  1--70. Springer, Berlin (2013)

\bibitem{bur18}
Burke, J.V., Curtis, F.E., Lewis, A.S., Overton, M.L., Sim{\~o}es, L.E.:
  Gradient sampling methods for nonsmooth optimization.
\newblock arXiv e-prints  (2018).
\newblock \urlprefix\url{http://arxiv.org/abs/1804.11003}

\bibitem{cal15}
{Calatroni}, L., {Chung}, C., {De Los Reyes}, J.C., {Sch{\"o}nlieb}, C.B.,
  {Valkonen}, T.: {Bilevel approaches for learning of variational imaging
  models}.
\newblock ArXiv e-prints  (2015).
\newblock \urlprefix\url{http://arxiv.org/abs/1505.02120}

\bibitem{car18}
{Cartis}, C., {Fiala}, J., {Marteau}, B., {Roberts}, L.: {Improving the
  Flexibility and Robustness of Model-Based Derivative-Free Optimization
  Solvers}.
\newblock ArXiv e-prints  (2018).
\newblock \urlprefix\url{http://arxiv.org/abs/1804.00154}

\bibitem{cha11}
Chambolle, A., Pock, T.: A first-order primal-dual algorithm for convex
  problems with applications to imaging.
\newblock Journal of mathematical imaging and vision \textbf{40}(1), 120--145
  (2011)

\bibitem{cla73}
Clarke, F.H.: Necessary conditions for nonsmooth problems in optimal control
  and the calculus of variations.
\newblock Ph.D. thesis, University of Washington (1973)

\bibitem{cla90}
Clarke, F.H.: Optimization and Nonsmooth Analysis, 1st edn.
\newblock Classics in Applied Mathematics. SIAM, Philadelphia (1990)

\bibitem{cur13}
Curtis, F.E., Que, X.: An adaptive gradient sampling algorithm for non-smooth
  optimization.
\newblock Optimization Methods and Software \textbf{28}(6), 1302--1324 (2013)

\bibitem{cur15}
Curtis, F.E., Que, X.: A quasi-{N}ewton algorithm for nonconvex, nonsmooth
  optimization with global convergence guarantees.
\newblock Mathematical Programming Computation \textbf{7}(4), 399--428 (2015)

\bibitem{dup16}
DuPont, B., Cagan, J.: A hybrid extended pattern search/genetic algorithm for
  multi-stage wind farm optimization.
\newblock Optimization and Engineering \textbf{17}(1), 77--103 (2016)

\bibitem{rii18smooth}
Ehrhardt, M.J., Riis, E.S., Ringholm, T., Sch{\"o}nlieb, C.B.: A geometric
  integration approach to smooth optimisation: Foundations of the discrete
  gradient method.
\newblock ArXiv e-prints  (2018).
\newblock \urlprefix\url{http://arxiv.org/abs/1805.06444}

\bibitem{eke99}
Ekeland, I., T{\'e}man, R.: Convex Analysis and Variational Problems, 1st edn.
\newblock SIAM, Philadelphia, PA, USA (1999)

\bibitem{fas14}
Fasano, G., Liuzzi, G., Lucidi, S., Rinaldi, F.: A linesearch-based
  derivative-free approach for nonsmooth constrained optimization.
\newblock SIAM Journal on Optimization \textbf{24}(3), 959--992 (2014)

\bibitem{fow08}
Fowler, K.R., Reese, J.P., Kees, C.E., Dennis~Jr, J., Kelley, C.T., Miller,
  C.T., Audet, C., Booker, A.J., Couture, G., Darwin, R.W., et~al.: Comparison
  of derivative-free optimization methods for groundwater supply and hydraulic
  capture community problems.
\newblock Advances in Water Resources \textbf{31}(5), 743--757 (2008)

\bibitem{gil99}
Giles, J.R.: A survey of {C}larke’s subdifferential and the differentiability
  of locally {L}ipschitz functions.
\newblock In: Progress in Optimization, pp. 3--26. Springer, Boston (1999)

\bibitem{gon96}
Gonzalez, O.: Time integration and discrete {H}amiltonian systems.
\newblock Journal of Nonlinear Science \textbf{6}(5), 449--467 (1996)

\bibitem{gra04}
Gray, G.A., Kolda, T.G., Sale, K., Young, M.M.: Optimizing an empirical scoring
  function for transmembrane protein structure determination.
\newblock INFORMS Journal on Computing \textbf{16}(4), 406--418 (2004)

\bibitem{gri08}
Griewank, A., Walther, A.: Evaluating Derivatives: Principles and Techniques of
  Algorithmic Differentiation, 2nd edn.
\newblock Society for Industrial and Applied Mathematics, Philadelphia (2008)

\bibitem{gri17}
Grimm, V., McLachlan, R.I., McLaren, D.I., Quispel, G.R.W., Sch{\"o}nlieb,
  C.B.: {Discrete gradient methods for solving variational image regularisation
  models}.
\newblock Journal of Physics A: Mathematical and Theoretical \textbf{50}(29),
  295201 (2017)

\bibitem{gur12}
G{\"u}rb{\"u}zbalaban, M., Overton, M.L.: On {N}esterov’s nonsmooth
  {C}hebyshev--{R}osenbrock functions.
\newblock Nonlinear Analysis: Theory, Methods \& Applications \textbf{75}(3),
  1282--1289 (2012)

\bibitem{hai06}
Hairer, E., Lubich, C., Wanner, G.: Geometric numerical integration:
  structure-preserving algorithms for ordinary differential equations, vol.~31,
  2nd edn.
\newblock Springer Science \& Business Media, Berlin (2006)

\bibitem{hea02}
Heath, M.T.: Scientific Computing: An Introductory Survey, 1st edn.
\newblock McGraw-Hill, New York (2002)

\bibitem{hin01}
Hinterm{\"u}ller, M.: A proximal bundle method based on approximate
  subgradients.
\newblock Computational Optimization and Applications \textbf{20}(3), 245--266
  (2001)

\bibitem{ito14}
Ito, K., Jin, B.: Inverse Problems: {T}ikhonov Theory And Algorithms, 1st edn.
\newblock Series On Applied Mathematics. World Scientific Publishing Company,
  Singapore (2014)

\bibitem{ito88}
Itoh, T., Abe, K.: {H}amiltonian-conserving discrete canonical equations based
  on variational difference quotients.
\newblock Journal of Computational Physics \textbf{76}(1), 85--102 (1988)

\bibitem{kam10}
K{\"a}mpf, J.H., Robinson, D.: Optimisation of building form for solar energy
  utilisation using constrained evolutionary algorithms.
\newblock Energy and Buildings \textbf{42}(6), 807--814 (2010)

\bibitem{kar16}
Karimi, H., Nutini, J., Schmidt, M.: Linear convergence of gradient and
  proximal-gradient methods under the {P}olyak--{{\L}}ojasiewicz condition.
\newblock In: Joint European Conference on Machine Learning and Knowledge
  Discovery in Databases, pp. 795--811. Springer (2016)

\bibitem{kiw85}
Kiwiel, K.C.: Methods of descent for nondifferentiable optimization, vol. 1133,
  1st edn.
\newblock Springer, Berlin (1985)

\bibitem{kiw10}
Kiwiel, K.C.: A nonderivative version of the gradient sampling algorithm for
  nonsmooth nonconvex optimization.
\newblock SIAM Journal on Optimization \textbf{20}(4), 1983--1994 (2010)

\bibitem{kun13}
Kunisch, K., Pock, T.: A bilevel optimization approach for parameter learning
  in variational models.
\newblock SIAM Journal on Imaging Sciences \textbf{6}(2), 938--983 (2013)

\bibitem{led11}
Le~Digabel, S.: Algorithm 909: {NOMAD}: Nonlinear optimization with the {MADS}
  algorithm.
\newblock ACM Transactions on Mathematical Software \textbf{37}(4), 1--15
  (2011).
\newblock \urlprefix\url{https://www.gerad.ca/nomad/Project/Home.html}

\bibitem{led09}
Le~Digabel, S., Tribes, C., Audet, C.: {NOMAD} user guide. technical report
  g-2009-37.
\newblock Tech. rep., Les cahiers du GERAD (2009)

\bibitem{lew13}
Lewis, A.S., Overton, M.L.: Nonsmooth optimization via quasi-{N}ewton methods.
\newblock Mathematical Programming \textbf{141}, 135--163 (2013)

\bibitem{liu17}
Liuzzi, G., Truemper, K.: Parallelized hybrid optimization methods for
  nonsmooth problems using {NOMAD} and linesearch.
\newblock Computational and Applied Mathematics \textbf{37}, 3172--3207 (2018)

\bibitem{mar01}
Martin, D., Fowlkes, C., Tal, D., Malik, J.: A database of human segmented
  natural images and its application to evaluating segmentation algorithms and
  measuring ecological statistics.
\newblock In: Proceedings of the 8th International Conference on Computer
  Vision, vol.~2, pp. 416--423. IEEE (2001)

\bibitem{mcl01}
McLachlan, R.I., Quispel, G.R.W.: Six lectures on the geometric integration of
  {ODE}s, p. 155–210.
\newblock London Mathematical Society Lecture Note Series. Cambridge University
  Press, Cambridge (2001)

\bibitem{mcl99}
McLachlan, R.I., Quispel, G.R.W., Robidoux, N.: Geometric integration using
  discrete gradients.
\newblock Philosophical Transactions of the Royal Society of London A:
  Mathematical, Physical and Engineering Sciences \textbf{357}(1754),
  1021--1045 (1999)

\bibitem{mic84}
Michel, P., Penot, J.P.: Calcul sous-différentiel pour des fonctions
  lipschitziennes et non lipschitziennes.
\newblock C. R. Acad. Sci. Paris \textbf{1}, 269--272 (1984)

\bibitem{miy17}
Miyatake, Y., Sogabe, T., Zhang, S.L.: {On the equivalence between {SOR}-type
  methods for linear systems and discrete gradient methods for gradient
  systems}.
\newblock ArXiv e-prints  (2017).
\newblock \urlprefix\url{http://arxiv.org/abs/1711.02277}

\bibitem{nel65}
Nelder, J.A., Mead, R.: A simplex method for function minimization.
\newblock The computer journal \textbf{7}(4), 308--313 (1965)

\bibitem{nes17}
Nesterov, Y., Spokoiny, V.: Random gradient-free minimization of convex
  functions.
\newblock Foundations of Computational Mathematics \textbf{17}(2), 527--566
  (2017)

\bibitem{och15}
Ochs, P., Ranftl, R., Brox, T., Pock, T.: Bilevel optimization with nonsmooth
  lower level problems.
\newblock In: International Conference on Scale Space and Variational Methods
  in Computer Vision, pp. 654--665. Springer (2015)

\bibitem{oeu07}
Oeuvray, R., Bierlaire, M.: A new derivative-free algorithm for the medical
  image registration problem.
\newblock International Journal of Modelling and Simulation \textbf{27}(2),
  115--124 (2007)

\bibitem{pen97}
Penot, J.P., Quang, P.H.: Generalized convexity of functions and generalized
  monotonicity of set-valued maps.
\newblock Journal of Optimization Theory and Applications \textbf{92}(2),
  343--356 (1997)

\bibitem{pol87}
Polyak, B.T.: Introduction to Optimization, 1st edn.
\newblock Optimization Software, Inc., New York (1987)

\bibitem{pow06}
Powell, M.J.D.: The {NEWUOA} software for unconstrained optimization without
  derivatives.
\newblock In: Large-scale nonlinear optimization, 1st edn., pp. 255--297.
  Springer, Boston (2006)

\bibitem{pow09}
Powell, M.J.D.: The {BOBYQA} algorithm for bound constrained optimization
  without derivatives.
\newblock Tech. rep., University of Cambridge (2009)

\bibitem{qui96}
Quispel, G.R.W., Turner, G.S.: Discrete gradient methods for solving {ODE}s
  numerically while preserving a first integral.
\newblock Journal of Physics A: Mathematical and General \textbf{29}(13),
  341--349 (1996)

\bibitem{de17}
De~los Reyes, J.C., Sch{\"o}nlieb, C.B., Valkonen, T.: Bilevel parameter
  learning for higher-order total variation regularisation models.
\newblock Journal of Mathematical Imaging and Vision \textbf{57}(1), 1--25
  (2017)

\bibitem{rin17}
Ringholm, T., Lazi{\'c}, J., Sch{\"o}nlieb, C.B.: Variational image
  regularization with {E}uler's elastica using a discrete gradient scheme.
\newblock ArXiv e-prints  (2017).
\newblock \urlprefix\url{http://arxiv.org/abs/1712.07386}

\bibitem{ros60}
Rosenbrock, H.H.: An automatic method for finding the greatest or least value
  of a function.
\newblock The Computer Journal \textbf{3}(3), 175--184 (1960)

\bibitem{rud92}
Rudin, L.I., Osher, S., Fatemi, E.: Nonlinear total variation based noise
  removal algorithms.
\newblock Physica D: nonlinear phenomena \textbf{60}(1-4), 259--268 (1992)

\bibitem{rud76}
Rudin, W.: Principles of Mathematical Analysis, 3rd edn.
\newblock International series in pure and applied mathematics. McGraw-Hill,
  New York (1976)

\bibitem{sch08}
Scherzer, O., Grasmair, M., Grossauer, H., Haltmeier, M., Lenzen, F.:
  Variational Methods in Imaging, 1st edn.
\newblock Applied Mathematical Sciences. Springer, New York (2008)

\bibitem{wan04}
Wang, Z., Bovik, A.C., Sheikh, H.R., Simoncelli, E.P.: Image quality
  assessment: from error visibility to structural similarity.
\newblock IEEE transactions on image processing \textbf{13}(4), 600--612 (2004)

\end{thebibliography}

\end{document}